\documentclass[oneside,english]{amsart}
\usepackage[T1]{fontenc}
\usepackage[utf8]{inputenc}
\usepackage{amsthm}
\usepackage{amssymb}
\usepackage{esint}
\usepackage[all]{xy}
\usepackage{lmodern}
\usepackage{color}

\makeatletter

\numberwithin{equation}{section}
\numberwithin{figure}{section}
\theoremstyle{plain}
\newtheorem{thm}{\protect\theoremname}[section]
  \theoremstyle{plain}
  \newtheorem{cor}[thm]{\protect\corollaryname}
  \newtheorem{prop}[thm]{\protect\propositionname}
  \theoremstyle{definition}
  \newtheorem{defn}[thm]{\protect\definitionname}
  \theoremstyle{remark}
  \newtheorem{rem}[thm]{\protect\remarkname}
  \theoremstyle{plain}
  \newtheorem{lem}[thm]{\protect\lemmaname}
  \theoremstyle{definition}
  \newtheorem{example}[thm]{\protect\examplename}
  
  \theoremstyle{plain}

\makeatother

\usepackage{babel}
  \providecommand{\definitionname}{Definition}
  \providecommand{\examplename}{Example}
  \providecommand{\lemmaname}{Lemma}
  \providecommand{\propositionname}{Proposition}
  \providecommand{\remarkname}{Remark}
  \providecommand{\corollaryname}{Corollary}
\providecommand{\theoremname}{Theorem}

\begin{document}

\title[Bialgebras, Hochschild cohomology and formality]{Deformation theory of bialgebras, higher Hochschild cohomology and formality}

\author{Grégory Ginot, Sinan Yalin}

\begin{abstract}
A first goal of this paper is to precisely relate the homotopy theories of bialgebras and $E_2$-algebras.
For this, we construct a conservative and fully faithful $\infty$-functor from pointed conilpotent homotopy bialgebras to augmented  $E_2$-algebras which consists in an appropriate ``cobar'' construction.
Then we prove that the (derived) formal moduli problem of homotopy bialgebras structures on a bialgebra is equivalent to the (derived) formal moduli problem of $E_2$-algebra structures on this ``cobar'' construction.
We show consequently that the $E_3$-algebra structure on the higher Hochschild complex of this cobar construction, given by the solution to the higher Deligne conjecture, controls the deformation theory of this bialgebra. This implies the existence of an $E_3$-structure on the deformation complex of a dg bialgebra, solving a long-standing conjecture of Gerstenhaber-Schack.
On this basis we solve a long-standing conjecture of Kontsevich, by proving the $E_3$-formality of the deformation complex
of the symmetric bialgebra. This provides as a corollary a new proof of Etingof-Kazdhan deformation quantization of Lie
bialgebras which extends to homotopy dg Lie bialgebras and is independent from the choice of an associator.
Along the way, we establish new general results of independent interest about the deformation theory of algebraic structures, which shed a new light on various deformation complexes and cohomology theories studied in the literature.
\end{abstract}
\maketitle

\tableofcontents{}

\section*{Introduction}

A deep interplay between bialgebras and quantum groups on one side and higher algebras on the other side (for instance $E_2$ and $E_3$-algebras) has been noticed for many years and more have been expected for a long time.
This led people to investigate the relationship between these two kinds of structures with the hope to establish some equivalence between their deformation theories, which could help to understand various related problems on both sides. 
One of the first goals of this paper is to give a precise relation between their respective homotopy theories and to use it to study the deformation theory of dg-bialgebras.

Algebras governed by $E_n$-operads and their deformation theory play a prominent role in a variety of topics such as the study 
of iterated loop spaces, Goodwillie-Weiss calculus for embedding spaces, deformation quantization of Poisson manifolds and Lie bialgebras,
and factorization 
homology~\cite{Ko1, Ko2, Lur0, Lur2, CPTVV, FG, Fra, Fre5, GJ, GTZ, Hin, HiLe, Kap-TFT, KoSo, May, Preygel, Shoikhet, Tam1, Toen-ICM}. 
These algebras form a hierarchy of ``more and more'' commutative and homotopy associative structures,
interpolating between homotopy associative or $A_\infty$-algebras (the $E_1$ case) and $E_{\infty}$-algebras (the colimit of the $E_n$'s, 
which is homotopy equivalent to differential graded commutative algebras in characteristic zero).
Their natural cohomology theory, the higher Hochschild cohomology, is a suitable generalization of the Hochschild cohomology
of associative algebras.

On the other hand, bialgebras are central in various topics of algebraic topology, representation theory and mathematical physics
(often incarnated as Hopf algebras)~\cite{Dri, Dri2, Baues, EK1, EK2, GS, Mer1, Mer2}. They consist of two structures, an associative algebra structure and a coassociative coalgebra
structure, related by a compatibility condition such that bialgebras can be alternately defined as algebras in coalgebras or vice-versa. 
Their natural cohomology theory, the Gerstenhaber-Schack cohomology, is a suitable mixing between Hochschild cohomology of algebras
and co-Hochschild cohomology of coalgebras.

Two important and long-standing conjectures expected from the relationship between these two kind of structures are the following. The first one, enunciated by Gerstenhaber and Schack (in a wrong way) 
at the beginning of the 90's \cite{GS}, characterizes the structure of the complex controlling the deformation theory of bialgebras,  which remained quite mysterious for a while. It is a ``differential graded bialgebra version'' of the famous Deligne conjecture for associative differential graded algebras (see for instance \cite{Tam1} and \cite{Ko1}). 
The second one, enunciated by Kontsevich in his celebrated work on deformation quantization of Poisson manifolds \cite{Ko2},
is a formality statement for the deformation complex of the symmetric bialgebra, which should imply as a corollary Drinfeld's 
and Etingof-Kazdhan's deformation quantization theory (see \cite{Dri}, \cite{EK1} and \cite{EK2}).

This situation is somehow a generalization to bialgebras of the situation encountered in deformation quantization of Poisson manifolds: 
the formality of the $E_2$-operad implies the Deligne conjecture, which states the existence of an $E_2$-algebra structure on the Hochschild
complex of an associative algebra, and this allows to prove that the Hochschild complex of the ring of functions
on a Poisson manifold is formal. This formality implies, in turn, the deformation quantization of such a manifold. 
Here, the formality of the $E_3$-operad should allow to prove a bialgebra version of this Deligne conjecture, 
implying the formality of the Gerstenhaber-Schack complex of the symmetric bialgebra, inducing in turn Drinfeld 
and Etingof-Kazdhan deformation quantization of Lie bialgebras. 

In this paper, we solve this bialgebra version of the Deligne conjecture proposed by Gerstenhaber-Schack 
as well as Kontsevich's $E_3$-formality conjecture, both at a greater level of generality than the original statements,
by implementing methods coming from derived algebraic geometry (see in particular \cite{Lur0}, \cite{Lur2} and \cite{Toe})
as well as from previous results of the authors (see in particular \cite{GTZ} and \cite{Yal2}) and by solving partially a conjecture of 
Francis-Gaistgory. 

We deduce from it a generalization of Etingof-Kadhan's celebrated deformation quantization.

Moreover, we enlighten along the way the role of various versions of deformation complexes of algebraic structures 
considered in the literature, by characterizing their associated derived formal moduli problems in full generality. 
Our new methods have thus an independent interest going even further than the solutions to the two aforementioned conjectures.
We now describe our main results and conclude with some perspectives for future works.

\subsection{From bialgebras to $E_2$-algebras}

The cobar construction defining a Quillen equivalence between conilpotent dg coalgebras and augmented dg algebras can be lifted 
to a functor from the category of associative coassociative bialgebras to the category of $E_2$-algebras \cite{Kad}.
A natural question~\cite{FG} was thus to know whether this implies an equivalence of homotopy theories (in the sense of $(\infty,1)$-categories)
between bialgebras and $E_2$-algebras. The appropriate answer to this problem is more subtle, 
and we solve it by using an appropriate notion of cobar construction for bialgebras,
which intertwines a bar construction on the algebra part of the structure with a cobar construction on the coalgebra part: 
\begin{thm}\label{T:Barenhanced}
(1) There exists a bar-cobar adjunction
\[
\mathcal{B}^{enh}_{E_1}: E_1-Alg^{0-con}(dgCog^{conil})\rightleftarrows E_1-Cog^{conil}(dgCog^{conil}):\Omega^{enh}_{E_1}
\]
inducing an equivalence of $(\infty,1)$-categories between $0$-connected homotopy associative algebras in conilpotent dg coalgebras ($0$-connected conilpotent homotopy associative bialgebras) and conilpotent homotopy coassociative coalgebras in conilpotent dg coalgebras.

(2) The equivalence above induces an equivalence of $(\infty,1)$-categories
\[
(\mathcal{B}^{enh,pt}_{E_1}(-)_-)_+: E_1-Alg^{aug,con}(dgCog^{conil})\rightleftarrows E_1-Cog^{conil,pt}(dgCog^{conil}):(\Omega^{enh,pt}_{E_1}(-)_-)_+
\]
between connected augmented conilpotent homotopy associative bialgebras and pointed conilpotent homotopy coassociative coalgebras in conilpotent dg coalgebras.
\end{thm}
This solves a conjecture of Francis-Gaitsgory \cite{FG} in the case where the base category is the category of conilpotent dg coalgebras, respectively the category of pointed conilpotent dg coalgebras. Let us note that we work here with non-negatively graded chain complexes.
\begin{cor}\label{C:Barenhanced}
The left adjoint of Theorem 0.1(2) induces a fully faithful $\infty$-functor
\[
\tilde{\Omega}: E_1-Alg^{aug,con}(dgCog^{conil})\hookrightarrow E_2-Alg^{aug}
\]
embedding connected augmented conilpotent homotopy associative bialgebras into augmented $E_2$-algebras.
\end{cor}

We use the above results to study deformation  of bialgebras in terms of those of $E_2$-algebras, 
by the means of the theory of moduli spaces of algebraic structures.

\subsection{Moduli spaces associated to deformation complexes}
One of the main tool for studying deformation is given by deformation complexes of algebraic structures.
In a part of this paper which is of independent interest, we prove several new results about deformation theory of algebraic structures. 
In particular, we give a conceptual explanation of the differences between various deformation complexes appearing in the literature 
by explaining which kind of deformations each of these complexes controls. We later use these results to identify the deformation complexes
of dg-bialgebras with those of its (appropriate) cobar construction.

\smallskip

Deformation theory is classicaly encoded by moduli problem and their associated dg-Lie-algebras
(for instance see~\cite{Hin0, Ko2, Lur0, Pri}).
Formal moduli problems are simplicial presheaves over augmented artinian cdgas satisfying some extra properties with respect
to homotopy pullbacks, which control the infinitesimal deformation theory of points on a given moduli space 
(variety, scheme, derived stack).
According to Lurie's equivalence theorem \cite{Lur0}, (derived) formal moduli problems and dg Lie algebras are equivalent
as $\infty$-categories, so to a given formal moduli problem controlling the infinitesimal neighbourhood 
of a point on a moduli space corresponds a dg Lie algebra called the deformation complex of this point.
Here we are interested in moduli spaces of algebraic structures and formal moduli problems controlling the deformations of such structures.
A convenient formalism to deal with algebraic structures at a high level of generality,
in order to encompass not only algebras but also bialgebras, is the notion of properad \cite{Val}, a suitable generalization of operads.
In the differential graded setting, algebraic structures are deformed as algebraic structures up to homotopy, 
for instance dg associative algebras are deformed as $A_{\infty}$-algebras.
Such a notion can be properly defined by considering cofibrant resolutions of properads.
Given a cofibrant dg properad $P_{\infty}$ and a $P_{\infty}$-algebra structure $\psi:P_{\infty}\rightarrow End_X$ on a complex $X$,
there is a formal moduli problem $\underline{P_{\infty}\{X\}}^{\psi}$ controlling the deformation theory of $\psi$. 
The associated deformation complex is an explicit Lie algebra noted $g_{P,X}^{\psi}$. 
In this paper, we prove the following general criterion to compare formal moduli problems between two kind of algebras:
\begin{thm}
Let $F:P_{\infty}-Alg\rightarrow Q_{\infty}-Alg$ be a fully faithful and conservative $\infty$-functor inducing functorially in $A$,
for every augmented artinian cdga $A$, a fully faithful and conservative $\infty$-functor 
$F:P_{\infty}-Alg(Mod_A)\rightarrow Q_{\infty}-Alg(Mod_A)$. Then $F$ induces an equivalence of formal moduli problems
\[
\underline{P_{\infty}\{X\}}^{\psi} \sim  \underline{Q_{\infty}\{F(X)\}}^{F(\psi)},
\]
where $F(\psi)$ is the $Q_{\infty}$-algebra structure on the image $F(X,\psi)$ of $(X,\psi)$ under $F$, 
hence an equivalence of tangent $L_{\infty}$-algebras
\[
g_{P,X}^{\psi}\sim g_{Q,F(X)}^{F(\psi)}.
\]
\end{thm}
Moreover, we get the following fiber sequence comparing the deformation complex of an algebra $(X,\psi)$ 
with the homotopy Lie algebra $Lie(\underline{haut}_{P_{\infty}}(X,\psi))$ tangent to 
the derived algebraic group $\underline{haut}_{P_{\infty}}(X,\psi)$ (in the sense of \cite{Fra})
of homotopy automorphisms of this algebra:
\begin{prop}
There is a homotopy fiber sequence of $L_{\infty}$-algebras
\[
g_{P,X}^{\psi}\rightarrow Lie(\underline{haut}_{P_{\infty}}(X,\psi))\rightarrow Lie(\underline{haut}(X)).
\]
\end{prop}

However, one should note that the deformation complex $g_{P,X}^{\psi}$ does not give exactly the usual cohomology theories. 
For instance, let us consider the case of the Hochschild complex Hom $( A^{\otimes >0}, A)$ of a dg associative algebra $A$. 
This Hochschild complex is bigraded, with a cohomological grading induced by the grading of $A$ 
and a weight grading given by the tensor powers $A^{\otimes \bullet}$. It turns out that the part $Hom(A,A)$
of weight $1$ in the Hochschild complex is missing in $g_{Ass,A}^{\psi}$. 

To correct this, and in particular to get the correct deformation complexes, 
we use a ``plus'' construction $g_{P^+,X}^{\psi^+}$ which gives the right cohomology theory and can be obtained by a slight modification
of the properad $P$, see~\S~\ref{S:Plus}. 
Moreover, we relate it to the homotopy automorphisms of the algebra:
\begin{thm}
There is a quasi-isomorphism of $L_{\infty}$-algebras
\[
g_{P^+,X}^{\psi^+}\simeq Lie(\underline{haut}_{P_{\infty}}(X,\psi)).
\]
\end{thm}
The conceptual explanation behind this phenomenon is that $g_{P,X}^{\psi}$ controls the deformations of the $P_{\infty}$-algebra
structure over a fixed complex $X$, whereas $g_{P^+,X}^{\psi^+}$ \emph{controls deformations} of this  $P_{\infty}$-algebra structure
plus compatible deformations of the differential of $X$, that is, deformations of the $P_{\infty}$-algebra
structure \emph{up to self quasi-isomorphisms} of $X$ (and in particular up to automorphisms). This is the role of the part $Hom(X,X)$ appearing for instance in Hochschild cohomology.

For instance, in the case of a an associative dg algebra $A$, the complex $g_{Ass^+,A}^{\psi^+}\cong Hom ( A^{\otimes >0}, A) [1]$ 
computes the reduced Hochschild cohomology of $A$, where the right hand side is a sub-complex of the standard Hochschild cochain complex
shifted down by $1$ equipped with its standard Lie algebra structure (due to Gerstenhaber).
The complex $g_{Ass,A}^{\psi} \cong Hom ( A^{\otimes >1}, A)[1]$ is the one controlling the formal moduli problem of deformations of $A$
with fixed differential, where the right hand side is the subcomplex of the previous shifted Hochschild cochain complex 
where we have removed the $Hom(A,A)$ component. There is also a third complex, the full shifted Hochschild complex 
$Hom ( A^{\otimes\geq 0}, A)[1)$, which controls not the deformations of $A$ itself but the linear deformations of its dg category 
of modules $Mod_A$~\cite{KellerLowen, Preygel}.

In the case of $n$-Poisson algebras (Poisson algebras with a Poisson bracket of degree $1-n$), 
we prove in Section 6 that the \emph{Tamarkin deformation complex}~\cite{Ta-deformationofd-algebra} (which we denote $CH_{Pois_n}^{(\bullet>0)}(A)[n]$ 
since it is the part of positive weight in the full Poisson complex~\cite{CaWi}) 
\emph{controls deformations
of $A$ into dg-$Pois_n$-algebras}, that is, 
it is the tangent Lie algebra $g_{Pois_n^+,A}^{\psi^+}$ of $\underline{haut_{Pois_n}}(A)$. 
We also prove that the deformation complex
$g_{Pois_n,R}^{\psi}$ of the formal moduli problem $\underline{{Pois_{n}}_{\infty} \{A \}}^{\psi}$ 
of homotopy $n$-Poisson algebra structures deforming $\psi$ is given by the $L_\infty$-algebra  $CH_{Pois_n}^{(\bullet>1)}(A)[n]$,
which is a further truncation of $CH_{Pois_n}(A)[n]$. Concerning the full shifted Poisson complex, we conjecture the following:

\noindent
\textbf{Conjecture.} Let $n\geq 2$ and $A$ be an $n$-Poisson algebra. The $L_{\infty}$-algebra structure
of the full shifted Poisson complex $CH_{Pois_n}^*(A)[n]$ controls the deformations of $Mod_A$ into  $E_{n-1}$-monoidal dg categories.

In the same way, we relate deformation complexes of bialgebras with truncations of the \emph{full} Gerstenhaber-Schack complex 
(see~\S~\ref{SS:GScomplex}), and we conjecture the following:

\noindent
\textbf{Conjecture.}
Let $B$ be a conilpotent dg bialgebra. The full shifted Gerstenhaber-Schack complex $C^{full}_{GS}(B,B)[2]$ controls the deformations
of $Mod_B$ as an $E_2$-monoidal dg category, that is, a braided monoidal dg category.

\subsection{Gerstenhaber-Schack conjecture}

Another open problem is to prove that the higher Hochschild complex of the cobar construction of a bialgebra is a deformation complex 
of this bialgebra. It is important since it allows to reduce questions of deformations of the bialgebras to those of
$E_2$-structures for which more tools are available.

We deduce it as a consequence of a more general result stating 
the equivalence of formal moduli problems in the sense of Lurie \cite{Lur2} between the moduli problem 
$\underline{Bialg_{\infty}\{B\}}^{\varphi}$ of homotopy bialgebra structures on a bialgebra $B$ deforming the bialgebra structure 
$\varphi:Bialg\rightarrow End_B$ (where $Bialg$ is the prop of bialgebras) and the moduli problem 
$\underline{E_2\{\tilde{\Omega} B\}}^{\tilde{\Omega}\varphi}$ of $E_2$-algebra structures deforming the $E_2$-algebra structure
$\tilde{\Omega}\varphi:E_2\rightarrow End_{\tilde{\Omega}B}$ on its cobar construction $\tilde{\Omega} B$.
For this, we use the characterization of such moduli problems obtained in \cite{Yal2}. 
We deduce that the Gerstenhaber-Schack complex of $B$ is quasi-isomorphic to the higher Hochschild complex of
$\Omega B$ as $L_{\infty}$-algebras, for the $L_{\infty}$ structure on this Hochschild complex induced by its $E_3$ structure. 
The existence of this $E_3$ structure on the higher Hochschild complex is a higher version of Deligne conjecture proved in \cite{Fra}
and \cite{GTZ}. Precisely, we prove the following results:
\begin{thm}\label{T:EquivDef}
Let $B$ be a pointed conilpotent homotopy associative dg bialgebra. Let $\varphi:Bialg_{\infty}\rightarrow End_B$ be this homotopy 
bialgebra structure on $B$, and let $\tilde{\Omega}\varphi:E_2\rightarrow End_{\tilde{\Omega}B}$ be the corresponding $E_2$-algebra
structure on its cobar construction $\tilde{\Omega} B$.

(1) There is a homotopy equivalence of formal moduli problems
\[
\underline{Bialg_{\infty}\{B\}}^{\varphi} \simeq \underline{E_2\{\tilde{\Omega}B\}}^{\tilde{\Omega}\varphi}.
\]
This homotopy equivalence induces a quasi-isomorphism of $L_{\infty}$-algebras
\[
g_{Bialg,B}^{\varphi}\stackrel{\sim}{\rightarrow} g_{E_2,\tilde{\Omega}B}^{\tilde{\Omega}\varphi}
\]
between the deformation complex of $B$ and the deformation complex of $\tilde{\Omega}B$.

(2) There is a homotopy equivalence of formal moduli problems
\[
\underline{Bialg_{\infty}^+}\{B\}^{\varphi^+} \simeq \underline{E_2^+}\{\tilde{\Omega}B\}^{\tilde{\Omega}\varphi^+}.
\]
This homotopy equivalence induces a quasi-isomorphism of $L_{\infty}$-algebras
\[
C^*_{GS}(B,B)\stackrel{\sim}{\rightarrow} T_{\tilde{\Omega}(B)} 
\]
between the Gerstenhaber-Schack complex of $B$ and the (truncated) $E_2$-Hochschild complex\footnote{that is the $E_2$-cotangent 
complex $T_{\tilde{\Omega}(B)}$ which is equivalent to $\mathbb{R}Der_{E_2}(\tilde{\Omega}B,\tilde{\Omega}B)$} of $\tilde{\Omega}B$.
\end{thm}
Then we extend these equivalences to $E_3$-algebra equivalences, and even extend this to the full Gerstenhaber-Schack complex $C_{GS}^{full}(B,B)$. We get eventually the following solution to a longstanding problem:
\begin{cor}[Gerstenhaber-Schack conjecture]\label{C:GS=HH}
(1) There is an $E_3$-algebra structures  on $C^*_{GS}(B,B)$ and a unital $E_3$-algebra structure on $C_{GS}^{full}(B,B))$
such that the following diagram  \[\xymatrix{ \tilde{\Omega}B [-1] \ar[r] & T_{\tilde{\Omega}(B)} \ar[r] &  CH_{E_2}^{*}( \tilde{\Omega}B, \tilde{\Omega}B) \\ 
  \tilde{\Omega}B[-1] \ar@{=}[u] \ar[r] & C_{GS}^{*}(B,B) \ar[u]^{\simeq} \ar[r] & C_{GS}^{full}(B,B) \ar[u]_{\simeq} }\] 
  is 
a commutative diagram of non-unital $E_3$-algebras with vertical arrows being equivalences. 

(2) The $E_3$-algebra structure on $C^*_{GS}(B,B)$ is
 a refinement of its $L_{\infty}$-algebra structure controlling the deformation theory of the bialgebra $B$.
\end{cor}
This corollary provides a stronger version of a longstanding conjecture stated by Gerstenhaber-Schack in \cite[Section 8]{GS} 
(the existence of a homotopy graded Lie algebra structure 
on the Gerstenhaber-Schack complex) for the three different versions 
of the Gerstenhaber-Schack and $E_2$-Hochschild complexes. This also solves a stronger version of a conjecture proposed by Kontsevich
in \cite{Ko2} in 2003 (existence of an $E_3$ structure in the case where $B=Sym(V)$ is the symmetric bialgebra on a vector space $V$).

\subsection{Kontsevich formality conjecture and deformation quantization of Lie bialgebras}
In the introduction of his work on deformation quantization of Poisson manifolds \cite{Ko2}, 
Kontsevich conjectured that Etingof-Kazhdan quantization should be the ``degree zero part'' of a much more general formality theorem  of the deformation complex $Def(Sym(V))$ of the symmetric bialgebra $Sym(V)$ on a vector space $V$.
This deformation complex should possess an $E_3$-algebra structure whose underlying $L_{\infty}$-structure controls the deformations  of this bialgebra, and should be formal as an $E_3$-algebra.

The existence of such an $E_3$-algebra structure on $Def(Sym(V))$ controling the deformations of $Sym(V)$ is a particular case
of Corollary~\ref{C:GS=HH}. To prove the $E_3$ formality of $Def(Sym(V))$, one has to compare two $E_3$-algebra structures on the Gerstenhaber-Schack cohomology of $Sym(V)$, the one coming from the action of the operad of $3$-Poisson algebras $P_3\cong H^*E_3$  via the formality $E_3\stackrel{\sim}{\rightarrow}H^*E_3\cong P_3$ of $E_3$ operads (see \cite{LaV} over $\mathbb{R}$ 
and more generally \cite{FW} over $\mathbb{Q}$), and the one transferred from $Def(Sym(V))$.
The following formality result was conjectured by Kontsevich (in the case where $V$ is a vector space) in the introduction of \cite{Ko2}:
\begin{thm}\label{T:GSforSym}
The deformation complex of the symmetric bialgebra $Sym(V)$ on a $\mathbb{Z}$-graded cochain complex $V$ whose cohomology is of finite dimension in each degree is formal over $\mathbb{Q}$ as an $E_3$-algebra.
\end{thm}
We obtain it by using the relationship with $E_2$-Hochschild cohomology (Corollary~\ref{C:GS=HH}) 
and the higher HKR-theorem for the latter~\cite{CaWi}.

We then (see~\S~\ref{S:EKQ}) obtain a new proof of Etingof-Kazhdan quantization theorem from the underlying $L_{\infty}$  formality given by
Theorem~\ref{T:GSforSym},
 the properties of Maurer-Cartan spaces for filtered $L_{\infty}$-algebras \cite{Yal1} and by using the results of \cite{Mer2}. 
In fact, after the appropriate completion, $Def(Sym(V))$ controls the (homotopy) $\mathbb{K}[[\hbar]]$-bialgebra  
structures quantizing the (homotopy) Lie bialgebra structures on $V$, which are controlled by $H^*_{GS}(Sym(V),Sym(V))$ which by formality 
also controls deformation of dg-bialgebras.
In particular, our work shows that  quantization of bialgebras follows from a higher analogue of the Kontsevich approach 
to quantization of Poisson manifolds in~\cite{Ko2, Tam1} as was conjectured from the early 2000's.

\smallskip

Let us add two important remarks.
First, what we get is actually a generalization of Etingof-Kazhdan quantization to homotopy dg Lie bialgebras, 
which encompasses the case of usual Lie bialgebras. Second, since the $E_3$ formality does not rely on the choice of a Drinfeld associator 
(contrary to the $E_2$ formality), we thus solve another open problem by showing that such a quantization does not rely 
on the choice of an associator. Precisely, we get the following results:
\begin{cor}\label{C:EKQuant}

(1) The $L_{\infty}$-formality underlying Theorem~\ref{T:GSforSym} induces a generalization of Etingof-Kazdhan deformation quantization theorem to homotopy dg Lie bialgebras whose cohomology is of finite dimension in each degree. In the case where $V$ is a vector space, this gives a new proof of Etingof-Kazdhan's theorem.

(2) Deformation quantization of homotopy dg Lie bialgebras whose cohomology is of finite dimension in each degree does not rely on the choice of a Drinfeld associator.
\end{cor}
In the case of a Lie algebra equipped with a Casimir element, the invariance under the choice of an associator was already known by Drinfeld in \cite{Dri}. A similar result was announced in \cite{Enr} for finite dimensional Lie bialgebras but without complete proof. Our corollary provides an alternative complete proof of this statement, as well as an extension of this invariance result to homotopy dg Lie bialgebras with finite dimensional cohomology which is completely new.

\subsection{Perspectives}

The new methods developed in this paper to approach deformation theory and quantization problems have several possible continuations. A first perspective is to get a better understanding of the moduli problem associated to the full $E_n$-Hochschild complex, which should control the deformation theory of $E_n$-monoidal dg categories and is deeply related to the deformation quantization of shifted Poisson structures in derived algebraic geometry (see for instance \cite[Theorem 5.2]{Toen-ICM}).
In particular, in the case of the full $E_2$-Hochschild complex, the equivalence with the full Gerstenhaber-Schack complex $C_{GS}^{full}(B,B)$ proved in this paper (Corollary~\ref{C:GS=HH}) will imply that $C_{GS}^{full}(B,B)$ controls the deformations of $Mod_B$ as an $E_2$-monoidal dg category, that is, a braided monoidal dg category. This is still a conjecture at present. The $E_3$-formality of $C_{GS}^{full}(Sym(g),Sym(g))$, where $g$ is a Lie bialgebra, then corresponds to the deformation quantization of $Mod_g$ with the equivalence relation on quantizations given by Drinfeld's gauge equivalence \cite{Dri} (equivalence of braided monoidal categories). In the case where $g=Lie(G)$ is the tangent Lie algebra of a reductive algebraic group $G$, this result has to be compared with the quantization of the derived classifying stack $BG$ in the sense of \cite{CPTVV}.

A second perspective is to recover the major results in deformation quantization obtained in \cite{EH1} and \cite{EH2} 
as consequences of suitable variants of our Formality Theorem~\ref{T:GSforSym}. In particular, we expect the quantization
of coboundary Lie bialgebras of \cite{EH2} to be a consequence of this formality suitably extended to algebras over framed little disks
operads (see for instance \cite{SW} for a definition of framed little disks, and \cite{GiS},\cite{Sev} for formality results). 
The coboundary on the bialgebra side should correspond to the framing on the little disks side. 
Let us note that this quantization would be then automatically independent from the choice of an associator.

Both perspectives will be investigated by the authors in future works.


\section*{Notations}

The reader will find below a list of the main notations used at several places in this article.
\begin{itemize}
\item We work over a field of characteristic zero noted $\mathbb{K}$.

\item $Ch_{\mathbb{K}}$ is the category of $\mathbb{Z}$-graded chain complexes over $\mathbb{K}$.

\item Let $(\mathcal{C},W_{\mathcal{C}})$ be a relative category, also called a category with weak equivalences. Here $\mathcal{C}$ is a category and $W_{\mathcal{C}}$ its subcategory of weak equivalences. The hammock localization (see \cite{DK}) of such a category with respect to its weak equivalences is noted $L^H(\mathcal{C},W_{\mathcal{C}})$, and the mapping spaces of this simplicial localization are noted $L^H(\mathcal{C},W_{\mathcal{C}})(X,Y)$.

\item Several categories of algebras and coalgebras will have a dedicated notation: 
$cdga$ for the category of commutative differential graded algebras, $dgArt$ for the category of artinian cdgas, $dgCog$ for the category of dg coassociative coalgebras and $dgLie$ for the category of dg Lie algebras.

\item Given a cdga $A$, the category of $A$-modules is noted $Mod_A$. More generally, if $\mathcal{C}$ is a symmetric monoidal category tensored over $Ch_{\mathbb{K}}$, the category of $A$-modules in $\mathcal{C}$ is noted $Mod_A(\mathcal{C})$.

\item Given a dg Lie algebra $g$, its Chevalley-Eilenberg algebra is noted $C^*_{CE}(g)$ and its Chevalley-Eilenberg coalgebra is noted $C_*^{CE}(g)$.

\item More general categories of algebras and coalgebras over operads or properads will have the following generic notations:
given a properad $P$, we will note $P-Alg$ the category of dg $P$-algebras and given an operad $P$ we will note $P-Cog$ the category of dg $P$-coalgebras.

\item Given a properad $P$, a cofibrant resolution of $P$ is noted $P_{\infty}$.

\item When the base category is a symmetric monoidal category $\mathcal{C}$ other than $Ch_{\mathbb{K}}$, we note $P-Alg(\mathcal{C})$ the category of $P$-algebras in $\mathcal{C}$ and $P-Cog(\mathcal{C})$ the category of $P$-coalgebras in $\mathcal{C}$.

\item These various categories of algebras and coalgebras will be restricted to several subcategories for which we will add the following super scripts: $aug$ for augmented and $coaug$ for coaugmented (for example, $P-Alg^{aug}$ is the category of augmented $P$-algebras), $con$ for connected, $0-con$ for $0$-connected, $conil$ for conilpotent, and $pt$ for pointed.

\item Algebras over properads form a relative category for the weak equivalences defined by chain quasi-isomorphisms. The subcategory of weak equivalences of $P-Alg$ is noted $wP-Alg$.

\item Given a properad $P$ and a complex $X$, we will consider an associated convolution Lie algebra noted $g_{P,X}$ which will give rise to two deformation complexes: the deformation complex $g_{P,X}^{\varphi}$ controling the formal moduli problem of deformations of a $P$-algebra structure $\varphi$ on $X$, and a variant $g_{P^+,X}^{\varphi^+}$ whose role will be explained in Section 3.

\item We will consider various moduli functors in this paper, defined as simplicial presheaves over artinian augmented cdgas: the simplicial presheaf of $P_{\infty}$-algebra structures on $X$ noted $\underline{P_{\infty}\{X\}}$, the formal moduli problem of deformations of a given $P_{\infty}$-algebra structure $\varphi$ on $X$ noted $\underline{P_{\infty}\{X\}}^{\varphi}$, and the derived algebraic group of homotopy automorphisms of $(X,\varphi)$ noted $\underline{haut}_{P_{\infty}}(X,\varphi)$.
The derived algebraic group of automorphisms of $X$ as a chain complex will be noted $\underline{haut}(X)$.

\item The little $n$-disks operads have several variants we will distinguish as follows: $E_n$ for the operad of non-unitary $E_n$-algebras, $uE_n$ for the operad of $E_n$-algebras with a unit up to homotopy, and $suE_n$ for the operad of strictly unitary $E_n$-algebras. The same notations hold for the operad $Pois_n$ of $n$-Poisson algebras.

\item We will use two tensor products of operads. The tensor product arity by arity, also called the Hadamar tensor product, will be noted $P\otimes_H Q$. The Boardman-Vogt tensor product will be noted $\otimes$ (see \cite{FV} for a definition).

\item The standard (full) complex computing the (Andr\'e-Quillen) cohomology theory of $n$-Poisson algebra $A$ is 
denoted $CH_{Pois_n}^*(A,A)$. Tamarkin standard deformation complex is the subcomplex in weight greater than $0$, that we denote
$CH_{Pois_n}^{(\bullet > 0)}(A,A)$. There is also another truncation $CH_{Pois_n}^{(\bullet > 1)}(A,A)$ 
complex has two gradings, the cohomological grading from the one of $A$ and a weight grading for the subcomplex in weight greater than $1$.
 The same notations will be used for some standard complexes computing $E_n$-Hochschild cohomology  of $E_n$-algebras.

\item The tangent complex of an $E_n$-algebra $A$ in the sense of \cite{Fra} is noted $T_A$.

\item The centralizer of an $E_n$-algebra morphism $f$ in the sense of \cite{Lur2} is noted $\mathfrak{z}_{E_n}(f)$, the version for $n$-Poisson algebras is noted $\mathfrak{z}_{Pois_n}(f)$.

\item Several kinds of bar and cobar constructions will play an important role in this paper. The functors $Bar$ and $\Omega$ are the classical bar and cobar constructions respectively for augmented associative algebras and coaugmented coassociative coalgebras. The operadic Koszul duality functors $Bar^{(n)}$ and $Cobar^{(n)}$ are adjoint functors defining respectively the bar construction for augmented $E_n$-algebras and the cobar construction for conilpotent $E_n$-coalgebras, which can be modeled respectively by an $n$-iterated bar construction and an $n$-iterated cobar construction (hence the notations). The versions for $n$-Poisson algebras are noted $Bar^{(n)}_{Pois_n}$ and $Cobar^{(n)}_{Pois_n}$. Another Koszul duality functor is the one used in \cite{Lur0}, which is an endofunctor of the category of augmented $E_n$-algebras noted $\mathcal{D}^n$.

\item More generally, given a Koszul operad $P$ in a symmetric monoidal $\infty$-category $\mathcal{C}$, there is a bar construction $\mathcal{B}_P:P-Alg(\mathcal{C})\rightarrow\mathcal{C}$, which can be enhanced by the Barr-Beck-Lurie theorem \cite{Lur2} to a bar functor $\mathcal{B}_P^{enh}:P-Alg(\mathcal{C})\rightarrow P^!-Cog^{conil}(\mathcal{C})$ where $P^!$ is the Koszul dual operad of $P$. Similarly, there is a cobar construction $\Omega_P$ and an enhanced version $\Omega^{enh}_P$.
\end{itemize}

\section{Recollections}

The goal of this section is to briefly review the key notions and results at the heart of the present paper.

\subsection{Symmetric monoidal categories over a base category}

\begin{defn}
Let $\mathcal{C}$ be a symmetric monoidal category. A symmetric monoidal
category over $\mathcal{C}$ is a symmetric monoidal category $(\mathcal{E},\otimes_{\mathcal{E}},1_{\mathcal{E}})$
endowed with a symmetric monoidal functor $\eta:\mathcal{C}\rightarrow\mathcal{E}$,
that is, an object under $\mathcal{C}$ in the $2$-category of symmetric monoidal categories.

This defines on $\mathcal{E}$ an external tensor product $\otimes :\mathcal{C}\times\mathcal{E}\rightarrow\mathcal{E}$
by $C\otimes X = \eta(C)\otimes_{\mathcal{E}} X$ for every $C\in\mathcal{C}$ and $X\in\mathcal{E}$.
This external tensor product is equipped with the following natural unit, associativity and symmetry isomorphisms:

(1) $\forall X\in\mathcal{E},1_{\mathcal{C}}\otimes X\cong X$,

(2) $\forall X\in\mathcal{E},\forall C,D\in\mathcal{C},(C\otimes D)\otimes X\cong C\otimes (D\otimes X)$,

(3) $\forall C\in\mathcal{C},\forall X,Y\in\mathcal{E},C\otimes (X\otimes Y)\cong(C\otimes X)\otimes Y\cong X\otimes(C\otimes Y)$.

The coherence constraints of these natural isomorphisms
(associativity pentagons, symmetry hexagons and unit triangles which mix both internal and external tensor products)
come from the symmetric monoidal structure of the functor $\eta$.

We will implicitly assume throughout the paper that all small limits and small
colimits exist in $\mathcal{C}$ and $\mathcal{E}$, and that each of these categories admit an internal hom bifunctor.
We suppose moreover the existence of an
external hom bifunctor $Hom_{\mathcal{E}}(-,-):\mathcal{E}^{op}\times\mathcal{E}\rightarrow\mathcal{C}$
satisfying an adjunction relation
\[
\forall C\in\mathcal{C},\forall X,Y\in\mathcal{E},Mor_{\mathcal{E}}(C\otimes X,Y)\cong Mor_{\mathcal{C}}(C,Hom_{\mathcal{E}}(X,Y))
\]
(so $\mathcal{E}$ is naturally an enriched category over $\mathcal{C}$).
\end{defn}
\begin{example}
Let $I$ be a small category; the $I$-diagrams in a symmetric monoidal category $\mathcal{C}$ form a symmetric monoidal
category over $\mathcal{C}$. The internal tensor product is defined pointwise, and the external tensor product
is defined by the functor $\eta$ which associates to $X\in\mathcal{C}$ the constant $I$-diagram $C_X$ on $X$.
The external hom $Hom_{\mathcal{C}^I}(-,-):\mathcal{C}^I\times\mathcal{C}^I\rightarrow\mathcal{C}$
is given by
\[
Hom_{\mathcal{C}^I}(X,Y)=\int_{i\in I}Hom_{\mathcal{C}}(X(i),Y(i)).
\]
\end{example}

Throughout this paper we will deal with symmetric monoidal categories equipped with a model
structure. We assume that the reader is familiar with the basics of model categories.
We refer to to Hirschhorn \cite{Hir} and Hovey \cite{Hov} for a comprehensive
treatment of homotopical algebra. We just recall the axioms of symmetric monoidal model categories
formalizing the interplay between the tensor and the model structures.
\begin{defn}
(1) A symmetric monoidal model category is a symmetric monoidal category
$\mathcal{C}$ equipped with a model category structure such that
the following axioms holds:

\textbf{MM0.} For any cofibrant object $X$ of $\mathcal{C}$, the map $Q1_{\mathcal{C}}\otimes X\rightarrow 1_{\mathcal{C}}\otimes X\cong X$ induced by a cofibrant resolution $Q1_{\mathcal{C}}\rightarrow 1_{\mathcal{C}}$ of the unit $1_{\mathcal{C}}$ is a weak equivalence.

\textbf{MM1.} The pushout-product $(i_{*},j_{*}):A\otimes D\oplus_{A\otimes C}B\otimes C\rightarrow B\otimes D$
of cofibrations $i:A\rightarrowtail B$ and $j:C\rightarrowtail D$ is a cofibration
which is also acyclic as soon as $i$ or $j$ is so.

(2) Suppose that $\mathcal{C}$ is a symmetric monoidal model category.
A symmetric monoidal category $\mathcal{E}$ over $\mathcal{C}$ is
a symmetric monoidal model category over $\mathcal{C}$ if the axiom MM1 holds for both the internal and external tensor
products of $\mathcal{E}$.
\end{defn}
\begin{example}
The usual projective model category $Ch_{\mathbb{K}}$ of unbounded chain complexes over
a field $\mathbb{K}$ forms a symmetric monoidal model category.
\end{example}
A property of the pushout-product axiom MM1 which will be useful later is that it is equivalent to the following dual version:
\begin{lem}(cf. \cite[Lemma 4.2.2]{Hov})
In a symmetric monoidal model category $\mathcal{C}$, the axiom MM1 is equivalent to the
following one:

\textbf{MM1'.} The morphism
\[
(i^{*},p_{*}):Hom_{\mathcal{C}}(B,X)\rightarrow Hom_{\mathcal{C}}(A,X)\times_{Hom_{\mathcal{C}}(A,Y)}Hom_{\mathcal{C}}(B,Y)
\]
induced by a cofibration $i:A\rightarrowtail B$ and a fibration $p:X\twoheadrightarrow Y$ is a fibration in $\mathcal{C}$ which
is also acyclic as soon as $i$ or $p$ is so.
\end{lem}

\subsection{Operads, Props and their algebras}

\subsubsection{Props and their algebras}

Let $\mathcal{C}$ be a symmetric monoidal category.
A $\Sigma$-biobject is a double sequence $\{M(m,n)\in\mathcal{C}\}_{(m,n)\in\mathbb{N}^2}$
where each $M(m,n)$ is equipped with a right action of $\Sigma_{m}$
and a left action of $\Sigma_{n}$ commuting with each other.
\begin{defn}
A prop is a $\Sigma$-biobject endowed with associative horizontal composition products
\[
\circ_{h}:P(m_1,n_1)\otimes P(m_2,n_2)\rightarrow P(m_1+m_2,n_1+n_2),
\]
associative vertical composition products
\[
\circ_{v}:P(k,n)\otimes P(m,k)\rightarrow P(m,n)
\]
and units $1\rightarrow P(n,n)$ which are neutral for $\circ_v$.
These products satisfy the exchange law
\[
(f_1\circ_h f_2)\circ_v(g_1\circ_h g_2) = (f_1\circ_v g_1)\circ_h(f_2\circ_v g_2)
\]
and are compatible with the actions of symmetric groups.

Morphisms of props are equivariant morphisms of collections compatible with the composition products.
\end{defn}

\begin{defn}
(1) To any object $X$ of $\mathcal{C}$ we can associate an endomorphism prop $End_X$ defined by
\[
End_X(m,n)=Hom_{\mathcal{C}}(X^{\otimes m},X^{\otimes n}).
\]

(2) A $P$-algebra is an object $X\in\mathcal{C}$ equipped with a prop morphism $P\rightarrow End_X$.
\end{defn}
We can also define a $P$-algebra in a symmetric monoidal category
over $\mathcal{C}$:
\begin{defn}
Let $\mathcal{E}$ be a symmetric monoidal category over $\mathcal{C}$.

(1) The endomorphism prop of $X\in\mathcal{E}$ is given by $End_X(m,n)=Hom_{\mathcal{E}}(X^{\otimes m},X^{\otimes n})$
where $Hom_{\mathcal{E}}(-,-)$ is the external hom bifunctor of $\mathcal{E}$.

(2) Let $P$ be a prop in $\mathcal{C}$. A $P$-algebra in $\mathcal{E}$
is an object $X\in\mathcal{E}$ equipped with a prop morphism $P\rightarrow End_X$.
\end{defn}

There is a functorial free prop construction $F$ leading to an adjunction
\[
F:\mathcal{C}^{\mathbb{S}}\rightleftarrows Prop:U
\]
with the forgetful functor $U$. The category of $\Sigma$-biobjects $\mathcal{C}^{\mathbb{S}}$
is a diagram category over $\mathcal{C}$, so it inherits the usual projective model structure of diagrams, which is a cofibrantly
generated model category structure. In the case of unbounded chain complexes over a field of characteristic zero, this model structure can be transferred along the free-forgetful adjunction:
\begin{thm}
(cf. \cite[Theorem 5.5]{Frep}) The category of dg props $Prop$ equipped with the
classes of componentwise weak equivalences and componentwise fibrations forms a cofibrantly generated model category.
\end{thm}

\subsubsection{Properads}

Composing operations of two $\Sigma$-biobjects $M$ and $N$ amounts to consider $2$-levelled directed graphs
(with no loops) with the first level indexed by operations of $M$ and the second level by operations of $N$.
Vertical composition by grafting and horizontal composition by concatenation allows one to define props as before.
The idea of properads is to mimick operads (for operations with several outputs), which are defined as monoids in $\Sigma$-objects,
by restricting the vertical composition product to connected graphs.
The unit for this connected composition product $\boxtimes_c$ is the $\Sigma$-biobject $I$ given by $I(1,1)=\mathbb{K}$ and $I(m,n)=0$ otherwise.
The category of $\Sigma$-biobjects then forms a symmetric monoidal category $(Ch_{\mathbb{K}}^{\mathbb{S}},\boxtimes_c,I)$.
\begin{defn}
A dg properad $(P,\mu,\eta)$ is a monoid in $(Ch_{\mathbb{K}}^{\mathbb{S}},\boxtimes_c,I)$,
where $\mu$ denotes the product and $\eta$ the unit.
It is augmented if there exists a morphism of properads $\epsilon:P\rightarrow I$.
In this case, there is a canonical isomorphism $P\cong I\oplus\overline{P}$
where $\overline{P}=ker(\epsilon)$ is called the augmentation ideal of $P$.

Morphisms of properads are morphisms of monoids in $(Ch_{\mathbb{K}}^{\mathbb{S}},\boxtimes_c,I)$.
\end{defn}
Properads have also their dual notion, namely coproperads:
\begin{defn}
A dg coproperad $(C,\Delta,\epsilon)$ is a comonoid in $(Ch_{\mathbb{K}}^{\mathbb{S}},\boxtimes_c,I)$.
\end{defn}
As in the prop case, there exists a free properad functor $\mathcal{F}$ forming an adjunction
\[
\mathcal{F}:Ch_{\mathbb{K}}^{\mathbb{S}}\rightleftarrows Properad :U
\]
with the forgetful functor $U$.
There is an explicit construction of the free properad in terms of direct sums of labelled graphs for which
we refer the reader to \cite{Val}. Dually, there exists a cofree coproperad functor denoted $\mathcal{F}_c(-)$ having the same underlying $\Sigma$-biobject. Moreover, according to \cite{MV2}, this adjunction equips dg properads with a cofibrantly generated model category structure with componentwise fibrations and weak equivalences.

There is also a notion of algebra over a properad similar to an algebra over a prop,
since the endomorphism prop restricts to an endomorphism properad.
Properads are general enough to encode a wide range of bialgebra structures such as associative and coassociative bialgebras, Lie bialgebras, Poisson bialgebras, Frobenius bialgebras for instance.

\subsubsection{Algebras over operads}

Operads are used to parametrize various kind of algebraic structures consisting of operations with
one single output. Fundamental examples of operads include the operad $As$ encoding associative algebras,
the operad $Com$ of commutative algebras, the operad $Lie$ of Lie algebras and the operad
$Pois$ of Poisson algebras. They can be defined as monoids in the category of $\Sigma$-objects (collections $\{P(n)\}_{n\in\mathbb{N}}$ with an action of $\Sigma_n$ on each $P(n)$), for the appropriate composition product. Dg operads form a model category with bar-cobar resolutions and Koszul duality \cite{LV}.
An algebra $X$ over a dg operad  $P$ can be defined in any symmetric monoidal category $\mathcal{E}$ over $Ch_{\mathbb{K}}$, alternatively as an algebra over the corresponding monad or as an operad morphism $P\rightarrow End_X$ where $End_X(n)=Hom_{\mathcal{E}}(X^{\otimes n},X)$ and $Hom_{\mathcal{E}}$ is the external hom bifunctor.
The category of $P$-algebras satisfies good homotopical properties, namely:
\begin{thm}(see \cite[Theorem 12.3.A]{Fre3})
Let $P$ be a $\Sigma$-cofibrant dg operad and $\mathcal{E}$ be a cofibrantly generated symmetric monoidal model category over $Ch_{\mathbb{K}}$.
Then the category $\mathcal{E}^P$ of $P$-algebras in $\mathcal{E}$ inherits a cofibrantly generated semi-model category structure such that the forgetful functor $U:\mathcal{E}^P\rightarrow\mathcal{E}$ creates weak equivalences and fibrations, and that the generating (acyclic) cofibrations are the images under the free $P$-algebra functor of the generating (acyclic) cofibrations of $\mathcal{E}$.
\end{thm}
\begin{rem}
In general, algebras over properads and props do not inherit such a model category structure, since there is no free algebra functor in this case.
\end{rem}

Dual to operads is the notion of cooperad, defined as a comonoid in the category of $\Sigma$-objects. A coalgebra over a cooperad is a coalgebra over the associated comonad.
We can go from operads to cooperads and vice-versa by dualization.
Indeed, if $C$ is a cooperad, then the $\Sigma$-module $P$ defined by $P(n)=C(n)^*=Hom_{\mathbb{K}}(C(n),\mathbb{K})$
form an operad. Conversely, suppose that $\mathbb{K}$ is of characteristic zero and $P$ is an operad such that 
each $P(n)$ is finite dimensional. Then the $P(n)^*$ form a cooperad in the sense of \cite{LV}.
The additional hypotheses are needed because we have to use, for finite dimensional vector spaces $V$ and $W$, the isomorphism
$(V\otimes W)^*\cong V^*\otimes W^*$ to define properly the cooperad coproduct.
We also give the definition of coalgebras over an operad:
\begin{defn}
(1) Let $P$ be an operad. A $P$-coalgebra is a complex $C$ equiped with linear applications
$\rho_n:P(n)\otimes C \rightarrow C^{\otimes n}$ for every $n\geq0$. These maps are $\Sigma_n$-equivariant
and associative with respect to the operadic compositions.

(2) Each $p\in P(n)$ gives rise to a cooperation $p^*:C\rightarrow C^{\otimes n}$.
The coalgebra $C$ is usually said to be conilpotent if for each $c\in C$, there exists $N\in\mathbb{N}$
so that $p^*(c)=0$ when we have $p\in P(n)$ with $n>N$.
\end{defn}
If $\mathbb{K}$ is a field of characteristic zero and the $P(n)$ are finite dimensional, then
it is equivalent to define a $P$-coalgebra via a family of applications
$\overline{\rho}_n:C\rightarrow P(n)^*\otimes_{\Sigma_n} C^{\otimes n}$.

\begin{thm}
The category $dgCog^{conil}$ of non-negatively graded conilpotent coassociative dg coalgebras forms a cofibrantly generated model category with quasi-isomorphisms as weak equivalences and degreewise injections as cofibrations.
\end{thm}
This follows verbatim from the argument line of \cite[Theorem 0.1]{Yal0} adapted to the non-negatively graded setting, or the argument line of \cite{GG} adapted to the conilpotent case. In both cases, the model structure is transferred via the adjunction between the forgetful functor and the reduced tensor coalgebra functor (which is both the cofree coalgebra functor in the positively graded setting of \cite{Yal0} and the cofree conilpotent coalgebra functor in the non-negatively graded setting of \cite{GG}). The fact that this model category is cofibrantly generated is crucial in the proof of Theorem 0.1(2).

We will also need a more general model category structure on conilpotent coalgebras over dg operads, however without needing it to be cofibrantly generated. The existence of such a model category results from the following arguments.
There are very general assumptions under which a category of coalgebras over a comonad in a model category inherits a model category structure \cite{HS}, which can be applied to the case of coalgebras over differential non-negatively graded operads. Given a differential non-negatively graded operad $P$, to check that the category $Ch_{\mathbb{K}}$ has the needed properties to get a left-induced model structure on the category $P-Cog^{conil}$ of dg $P$-coalgebras from \cite[Theorem 5.8]{HS}, one just specializes the results of \cite[Section 6]{HS} (enunciated for right dg modules over a dga $A$) to the case $A=\mathbb{K}$:
\begin{thm}
Let $P$ be a differential non-negatively graded operad. The category $P-Cog^{conil}$
inherits a model category structure such that the forgetful functor $U: P-Cog^{conil}\rightarrow Ch_{\mathbb{K}}$ creates the cofibrations and the weak equivalences.
\end{thm}

\subsection{$E_n$-operads, higher Hochschild cohomology and the Deligne conjecture}\label{S:EnHoch}

Configuration spaces of $n$-disks into a bigger $n$-disk gather into a topological operad $D_n$ called the little $n$-disks operad. 
An $E_n$-operad is a dg operad quasi-isomorphic to the singular chains $C_*D_n$ of the little $n$-disks operad. 
The formality of the little $2$-disks operad (i.e. the fact that it is quasi-isomorphic as an operad to its cohomology) 
is the key point to prove Deligne conjecture about the homotopy Gerstenhaber structure (that is, the $E_2$-algebra structure) 
of the Hochschild complex, which in turn provided an alternative method for deformation quantization of Poisson manifolds \cite{Ko1}, 
\cite{Ko2}, \cite{Tam1}, \cite{Tam2}. Its formality holds over $\mathbb{Q}$ and relies on the choice of a Drinfeld associator 
(which exists over $\mathbb{Q}$, as proved in \cite{Dri}), a process closely related to deformation quantization of Lie bialgebras as well. 
The formality of $E_n$-operads for $n\geq 3$ was sketched over $\mathbb{R}$ in \cite{Ko1} and fully proved in \cite{LaV}. 
This result was recently superseded in \cite{FW}, where this formality was proved to hold over $\mathbb{Q}$ and to be intrinsic
(under a mild technical assumption satisfied in particular by the little $n$-disks operads). 
An important observation to note here is that contrary to the case $n=2$, the formality of $E_n$-operads
for $n\geq 3$ does not rely on the choice of an associator.

Given an ordinary associative (or $E_1$) algebra $A$, its endomorphisms $Hom_{biMod_A}(A,A)$ in the category $biMod_A$ 
of $A$-bimodules form nothing but the center $Z(A)$ of $A$. Deriving this hom object gives the Hochschild complex $CH^*(A,A)$ of $A$,
and the Hochschild cohomology $HH^*(A,A)$ of $A$ satisfies $HH^0(A,A)=Z(A)$. 
One says that the Hoschchild complex is the derived center of $A$, and the Deligne conjecture can then be reformulated 
as \lq\lq{}the derived center of an $E_1$-algebra forms an $E_2$-algebra\rq\rq{}. This sentence enlarges to a similar statement 
for $E_n$-algebras. 
\begin{defn}\label{D:HochCochain}
The (full) Hochschild complex of an $E_n$-algebra $A$, computing its higher Hochschild cohomology,
is the derived hom $CH^*_{E_n}(A,A)=\mathbb{R}Hom^{E_n}_A(A,A)$ in the category of (operadic) $A$-modules over $E_n$.
\end{defn} 
The Deligne conjecture endows the Hochschild cochain complex with an $E_{n+1}$-algebra structure~\cite[Theorem 6.28]{GTZ} 
or~\cite{Fra, Lur2}.  Associated to an $E_n$-algebra $A$, one also has its cotangent complex $L_A$, 
which classifies square-zero extensions of $A$~\cite{Fra, Lur2}. 
\begin{defn}[\cite{Fra}]\label{D:HochTangent} The tangent complex $T_A$ of an $E_n$-algebra $A$ is the dual
$T_A:= Hom^{E_n}_A(L_A,A)\cong \mathbb{R}Der(A,A)$.
\end{defn}
Francis~\cite{Fra, Lur2} has proved that $T_A[-n]$ has a canonical structure of $E_{n+1}$-algebra 
and further that we have a fiber sequence $T_A[-n]\to  CH^*_{E_n}(A,A) \to A$ where the first map is a map of $E_{n+1}$-algebras.

\subsection{Deformation quantization of Lie bialgebras}\label{SS:DQLieBialg}

Lie bialgebras originally arose from Poisson-Lie groups in mathematical physics.
Poisson-Lie groups are Lie groups with a compatible Poisson structure, and appear as gauge groups 
of various classical mechanical systems such as integrable systems. The tangent space $T_eG$ 
of a Poisson-Lie group $G$ has more structure than a Lie algebra, because of the Poisson bracket, which induces a compatible Lie coalgebra structure on $T_eG$ so that $T_eG$ forms Lie bialgebra.
The compatibility relation between the bracket and the cobracket is called the Drinfeld's compatibility relation or the cocycle relation. Quantizing tangent spaces of Poisson-Lie groups gives solutions to the quantum Yang-Baxter equation, which allow to build ``exactly solvable'' quantum systems from classical mechanical systems.  Deformation quantization of Lie bialgebras produces quantum groups, which are relevant not only for mathematical physics but also for low-dimensional topology (quantum invariants of knots and $3$-manifolds), and are deeply related to algebraic geometry and number theory (moduli spaces of curves, Grothendieck-Teichmüller groups, multizeta values) via the Drinfeld associators \cite{Dri}.

The problem of a universal quantization of Lie bialgebras raised by Drinfeld was solved by Etingof and Kazhdan \cite{EK1}, \cite{EK2}. A deformation quantization of a Lie bialgebra $g$ is a topologically free Hopf algebra $H$ over the ring of formal power series $\mathbb{K}[[\hbar]]$ such that $H/\hbar H$ is isomorphic to $U(g)$ (the enveloping algebra of $g$)as a co-Poisson bialgebra. Such a Hopf algebra is called a quantum universal enveloping algebra (QUE for short). Conversely, the quasi-classical limit of a QUE algebra $H$ is the Lie bialgebra of primitive elements $g=Prim(H/\hbar H)$ whose cobracket is induced by the coproduct of $H$. For technical reasons (the need of a ``quasi-triangular structure'' on the Lie bialgebra, that is, of a classical $r$-matrix), the initial construction of Etingof-Kazdhan takes as input a quasi-triangular Lie bialgebra $D(g)$ called the double of $g$. Then they equip the category $Mod_{D(g)}[[\hbar]]$ (same objects than $Mod_{D(g)}$ but the vector spaces of morphisms are tensored by the ring of formal series), called the Drinfeld category, with a braided monoidal structure induced by the $r$-matrix of $D(g)$ and the choice of a Drinfeld associator. The forgetful functor from the Drinfeld category to the category of modules is a braided monoidal functor, so by the tannakian formalism the Drinfeld category is equivalent to the category of modules over a topologically free Hopf algebra quantizing $D(g)$. The QUE algebra of $g$ is then a certain Hopf subalgebra of this Hopf algebra.

As explained in the sequel \cite{EK2} with the formalism of cyclic categories, and later in \cite{EE} in the formalism of props, deformation quantization of Lie bialgebras can be reformulated as the existence of a prop morphism
\[
QUE\rightarrow UE_{cP}[[\hbar]]
\]
between the prop of $QUE$ algebras and the prop $UE_{cP}[[\hbar]]$ of topologically free co-Poisson $\mathbb{K}[[\hbar]]$-bialgebras. This prop morphism induces a quantization functor from the category of $UE_{cP}[[\hbar]]$-algebras to the category of $QUE$-algebras. The category of $UE_{cP}[[\hbar]]$-algebras is equivalent to the category of topologically free Lie bialgebras over $\mathbb{K}[[\hbar]]$ (the enveloping algebra of a Lie bialgebra is a co-Poisson bialgebra). The existence of such a prop morphism relies on the use of a Drinfeld associator.
An alternative propic approach to deformation quantization, which will be useful for us in this paper, is presented in \cite{Mer2}. We will go back to this with more details in Section 8.

\section{Moduli problems of algebraic structures and deformation complexes}

\subsection{Formal moduli problems and (homotopy) Lie algebras}

Formal moduli problems arise when one wants to study the deformation theory of a point of a given moduli space (variety, scheme, stack, derived stack). The general principle of moduli problems is that the deformation theory of a given point in a formal neighbourhood of this point (that is, the formal completion of the moduli space at this point) is controlled by a certain tangent dg Lie algebra. This idea of a correspondence between formal moduli problems and dg Lie algebras arose from unpublished work of Deligne, Drinfed and Feigin, and was developed further by Goldman-Millson, Hinich, Kontsevich, Manetti among others.
However, there was no systematic recipe to build a dg Lie algebra for a given moduli problem, and even worse, different dg Lie algebras could represent the same moduli problem. To overcome these difficulties, one has to consider moduli problems in a derived setting. The rigorous statement of an equivalence between derived formal moduli problems and dg Lie algebras was proved independently by Lurie in \cite{Lur0} and by Pridham in \cite{Pri}.In this paper, what we will call moduli problems are actually derived moduli problems.

Briefly, formal moduli problems are functors $F:dgArt_{\mathbb{K}}^{aug}\rightarrow sSet$ from augmented artinian commutative differential graded algebras to simplicial sets, such that $F(\mathbb{K})\simeq pt$ and $F$ preserves certain pullbacks (we refer the reader to \cite{Lur0} for more details). The value $F(\mathbb{K})$ corresponds to the point of which we study the formal neighbourhood, the evaluation $F(\mathbb{K}[t]/(t^2))$ on the algebra of dual numbers encodes infinitesimal deformations of this point, and the $F(\mathbb{K}[t]/(t^n))$ are polynomial deformations of a higher order, for instance.
Formal moduli problems form a full sub-$\infty$-category noted $FMP_{\mathbb{K}}$ of the $\infty$-category of simplicial presheaves over augmented artinian cdgas. By \cite[Theorem 2.0.2]{Lur0}, this $\infty$-category is equivalent to the $\infty$-category $dgLie_{\mathbb{K}}$ of dg Lie algebras. Moreover, one side of the equivalence is made explicit, and is equivalent to the nerve construction of dg Lie algebras studied thoroughly by Hinich in \cite{Hin0}. The homotopy invariance of the nerve relies on nilpotence conditions on the dg Lie algebra. In the case of formal moduli problems, this nilpotence condition is always satisfied because one tensors the Lie algebra with the maximal ideal of an augmented artinian cdga.

It turns out that this nerve construction can be extended to homotopy Lie algebras, that is, $L_{\infty}$-algebras.
There are two equivalent definitions of an $L_{\infty}$-algebra:
\begin{defn}
(1) An $L_{\infty}$-algebra is a graded vector space $g=\{g_n\}_{n\in\mathbb{Z}}$ equipped with maps
$l_k:g^{\otimes k}\rightarrow g$ of degree $2-k$, for $k\geq 1$, satisfying the following properties:
\begin{itemize}
\item $l_k(...,x_i,x_{i+1},...)=-(-1)^{|x_i||x_{i+1}|}l_k(...,x_{i+1},x_i,...)$
\item for every $k\geq 1$, the generalized Jacobi identities
\[
\sum_{i=1}^k\sum_{\sigma\in Sh(i,k-i)}(-1)^{\epsilon(i)}l_k(l_i(x_{\sigma(1)},...,x_{\sigma(i)}),x_{\sigma(i+1)},...,x_{\sigma(k)})=0
\]
where $\sigma$ ranges over the $(i,k-i)$-shuffles and
\[
\epsilon(i) = i+\sum_{j_1<j_2,\sigma(j_1)>\sigma(j_2)}(|x_{j_1}||x_{j_2}|+1).
\]
\end{itemize}

(2) An $L_{\infty}$-algebra structure on a graded vector space $g=\{g_n\}_{n\in\mathbb{Z}}$ is a
coderivation $Q:Sym^{\bullet\geq 1}(g[1])\rightarrow Sym^{\bullet\geq 1}(g[1])$ of degree $1$ of the cofree cocommutative coalgebra 
$ Sym^{\bullet\geq 1}(g[1])$ such that $Q^2=0$.
\end{defn}
The bracket $l_1$ is actually the differential of $g$ as a cochain complex. When the brackets $l_k$ vanish
for $k\geq 3$, then one gets a dg Lie algebra.
The dg algebra $C^*(g)$ obtained by dualizing the dg coalgebra of (2) is called the Chevalley-Eilenberg algebra of $g$.

A $L_{\infty}$ algebra $g$ is filtered if it admits a decreasing filtration
\[
g=F_1g\supseteq F_2g\supseteq...\supseteq F_rg\supseteq ...
\]
compatible with the brackets: for every $k\geq 1$,
\[
l_k(F_rg,g,...,g)\in F_rg.
\]
We suppose moreover that for every $r$, there exists an integer $N(r)$ such that $l_k(g,...,g)\subseteq F_rg$
for every $k>N(r)$.
A filtered $L_{\infty}$ algebra $g$ is complete if the canonical map $g\rightarrow lim_rg/F_rg$ is an isomorphism.

The completeness of a $L_{\infty}$ algebra allows to define properly the notion of Maurer-Cartan element:
\begin{defn}
(1) Let $g$ be a dg $L_{\infty}$-algebra and $\tau\in g^1$, we say that $\tau$ is a Maurer-Cartan element of $g$ if
\[
\sum_{k\geq 1} \frac{1}{k!} l_k(\tau,...,\tau)=0.
\]
The set of Maurer-Cartan elements of $g$ is noted $MC(g)$.

(2) The simplicial Maurer-Cartan set is then defined by
\[
MC_{\bullet}(g)=MC(g\hat{\otimes}\Omega_{\bullet}),
\],
where $\Omega_{\bullet}$ is the Sullivan cdga of de Rham polynomial forms on the standard simplex $\Delta^{\bullet}$ (see \cite{Sul})
and $\hat{\otimes}$ is the completed tensor product with respect to the filtration induced by $g$.
\end{defn}
The simplicial Maurer-Cartan set is a Kan complex, functorial in $g$ and preserves quasi-isomorphisms of complete $L_{\infty}$-algebras.
The Maurer-Cartan moduli set of $g$ is $\mathcal{MC}(g)=\pi_0MC_{\bullet}(g)$: it is the quotient of the set
of Maurer-Cartan elements of $g$ by the homotopy relation defined by the $1$-simplices.
When $g$ is a complete dg Lie algebra, it turns out that this homotopy relation is equivalent to the action of the gauge
group $exp(g^0)$ (a prounipotent algebraic group acting on Maurer-Cartan elements), so in this case
this moduli set coincides with the one usually known for Lie algebras.
We refer the reader to \cite{Yal2} for more details about all these results.
We also recall briefly the notion of twisting by a Maurer-Cartan element.
The twisting of a complete $L_{\infty}$ algebra $g$ by a Maurer-Cartan element $\tau$ is the complete $L_{\infty}$ algebra $g^{\tau}$
with the same underlying graded vector space and new brackets $l_k^{\tau}$ defined by
\[
l_k^{\tau}(x_1,...,x_k)=\sum_{i\geq 0}\frac{1}{i!}l_{k+i}(\underbrace{\tau,...,\tau}_i,x_1,...,x_k)
\]
where the $l_k$ are the brackets of $g$.

To conclude, we explain why Lurie's equivalence \cite[Theorem 2.0.2]{Lur0} lifts from the $\infty$-category of dg Lie algebras $dgLie$ to the $\infty$-category of $L_{\infty}$-algebras $L_{\infty}-Alg$.
Let $p:L_{\infty}\stackrel{\sim}{\rightarrow}Lie$ be the cofibrant resolution of the operad $Lie$ encoding $L_{\infty}$-algebras. This morphism induces a functor $p^*:dgLie\rightarrow L_{\infty}-Alg$ which associates to any dg Lie algebra the $L_{\infty}$-algebra with the same differential, the same bracket of arity $2$ and trivial higher brackets in arities greater than $2$.
It turns out that this functor is a right Quillen functor belonging to a Quillen equivalence
\[
p_{!}:L_{\infty}-Alg\leftrightarrows dgLie :p^*,
\]
since $p$ is a quasi-isomorphism of $\Sigma$-cofibrant operads (see \cite[Theorem 16.A]{Fre3} for the general result).
Quillen equivalences induce equivalences of the $\infty$-categories associated to these model categories (which can be realized, for instance, by taking the coherent nerve of the simplicial localization of these model categories). We have a commutative triangle of $\infty$-categories
\[
\xymatrix{
L_{\infty}-Alg\ar[dr]^-{\tilde{\psi}} & \\
dgLie \ar[u]^-{p^*}\ar[r]_-{\psi} & FMP_{\mathbb{K}}
}
\]
where $\psi$ and $\tilde{\psi}$ send a Lie algebra, respectively an $L_{\infty}$-algebra, to its nerve functor or Maurer-Cartan space. The maps $p^*$ and $\psi$ are weak equivalences of $\infty$-categories in the chosen model category of $\infty$-categories (which can be for instance the one of quasicategories \cite{Lur1}, but actually any model works). By the two-out-of-three property of weak equivalences, this implies that $\tilde{\psi}:L_{\infty}-Alg\rightarrow FMP_{\mathbb{K}}$ is a weak equivalence of $\infty$-categories.

\subsection{Fully faithful $\infty$-functors}

In the next section, we will need the following useful criteria to obtain fully faithful $\infty$-functors between categories of algebras, and see under which conditions they induce equivalences between the associated formal moduli problems:
\begin{lem}\label{L: stricteq}
Let $F:(\mathcal{C},W_{\mathcal{C}})\rightleftarrows (\mathcal{D},W_{\mathcal{D}}):G$ be an adjunction of relative categories (that is, the functors $F$ and $G$ preserves weak equivalences) such that the unit and counit of this adjunction are pointwise weak equivalences. Then $F$ induces an equivalence of $\infty$-categories with inverse $G$.
\end{lem}
\begin{proof}
Let us denote by $RelCat$ the category of relative categories. The objects are the relative categories
and the morphisms are the relative functors, that is, the functors restricting to functors between the
categories of weak equivalences.
By \cite[Theorem 6.1]{BK}, there is an adjunction between
the category of bisimplicial sets and the category of relative categories
\[
K_{\xi}:sSets^{\Delta^{op}}\leftrightarrows RelCat:N_{\xi}
\]
(where $K_{\xi}$ is the left adjoint and $N_{\xi}$ the right adjoint)
which lifts any Bousfield localization of the Reedy model structure of bisimplicial sets into a model structure on $RelCat$.
In the particular case of the Bousfield localization defining the model category $CSS$ of complete Segal spaces \cite[Theorem 7.2]{Rez2}, one obtains a Quillen equivalent homotopy theory of $\infty$-categories in $RelCat$ \cite{BK}.

A way to build the $\infty$-category associated to a relative category $(\mathcal{C},W_{\mathcal{C}})$ is to take a functorial fibrant resolution $N_{\xi}(\mathcal{C},W_{\mathcal{C}})^f$ of the bisimplicial set $N_{\xi}(\mathcal{C},W_{\mathcal{C}})$ in $CSS$ to get a complete Segal space. So we want to prove that $N_{\xi}F^f$ is a weak equivalence of $CSS$. For this, let us note first that the assumption on the adjunction between $F$ and $G$ implies that $F$ is a strict homotopy equivalence in $RelCat$ in the sense of \cite{BK}. By \cite[Proposition 7.5 (iii)]{BK}, the functor $N_{\xi}$ preserves homotopy equivalences, so $N_{\xi}F$ is a homotopy equivalence of bisimplicial sets, hence a Reedy weak equivalence. Since $CSS$ is a Bousfield localization of the Reedy model structure on bisimplicial sets, Reedy weak equivalences are weak equivalences in $CSS$, then by applying the fibrant resolution functor $(-)^f$ we conclude that $N_{\xi}F^f$ is a weak equivalence of complete Segal spaces.
\end{proof}
From this we deduce:
\begin{prop}\label{P: fullfaith}
Let $F:(\mathcal{C},W_{\mathcal{C}})\rightarrow (\mathcal{D},W_{\mathcal{D}})$ be a functor of relative categories (that is, a functor preserving weak equivalences). If $F$ is fully faithful, then it induces a fully faithful $\infty$-functor at the level of the associated $\infty$-categories.
\end{prop}
\begin{proof}
Since $F$ is fully faithful, it induces an equivalence of relative categories between $\mathcal{C}$ and $F(\mathcal{C})$, hence an equivalence of the associated $\infty$-categories by Lemma~\ref{L: stricteq}, which exactly means that $F$ induces a fully faithful $\infty$-functor between $\mathcal{C}$ and $\mathcal{D}$.
\end{proof}
\begin{cor}\label{C: fullfaithprop}
A surjection of props $\varphi:P\twoheadrightarrow Q$ induces a fully faithful $\infty$-functor
\[
\varphi^*:Q-Alg\hookrightarrow P-Alg,
\]
where weak equivalences on both sides are defined by quasi-isomorphisms.
\end{cor}

In the formalism of Dwyer-Kan's hammock localization, a fully faithful functor $F:\mathcal{C}\rightarrow \mathcal{D}$ is a functor satisfying the following property:
for every two objects $X$ and $Y$ of $\mathcal{C}$, it induces a weak equivalence of simplicial mapping spaces
\[
L^H(\mathcal{C},W_{\mathcal{C}})(X,Y)\stackrel{\sim}{\rightarrow}L^H(\mathcal{D},W_{\mathcal{D}})(F(X),F(Y)).
\]
In particular, the associated functor $Ho(F)$ at the level of homotopy categories is fully faithful (but not an equivalence).
We would like this weak equivalence to restrict at the level of homotopy automorphisms. For this, let us recall that a functor $F:\mathcal{C}\rightarrow\mathcal{D}$ is conservative if the following property holds: let $f$ be a morphism of $\mathcal{C}$, if $F(f)$ is a weak equivalence of $\mathcal{D}$ then $f$ is a weak equivalence of $\mathcal{C}$.
\begin{lem}\label{L: equivhaut}
Let $F:(\mathcal{C},W_{\mathcal{C}})\rightarrow (\mathcal{D},W_{\mathcal{D}})$ be a functor of relative categories. Let us suppose that $F$ is conservative and fully faithful. Then the restriction of $F$ to the subcategories of weak equivalences
\[
wF:W_{\mathcal{C}}\rightarrow W_{\mathcal{D}}
\]
is fully faithful and induces, for every $X\in\mathcal{C}$, a weak equivalence of homotopy automorphisms
\[
L^HW_{\mathcal{C}}(X,X)\stackrel{\sim}{\rightarrow}L^HW_{\mathcal{D}}(F(X),F(X)),
\]
where $L^HW_{\mathcal{C}}$ is Dwyer-Kan's hammock localization of $W_{\mathcal{C}}$ with respect to itself.
\end{lem}
\begin{proof}
Since $F$ preserves weak equivalences, it induces a functor $wF:W_{\mathcal{C}}\rightarrow W_{\mathcal{D}}$.
Moreover, the functor $wF$ is clearly faithful as well. Now let $f:F(X)\rightarrow F(Y)$ be a weak equivalence of $\mathcal{D}$. Since $F$ is full, there exists a morphism $g$ of $\mathcal{C}$ such that $f=F(g)$.Since $F$ is conservative, the fact that $f$ is a weak equivalence of $\mathcal{D}$ implies that $g$ is a weak equivalence of $\mathcal{C}$, hence $wF$ is full. According to Proposition 2.4, the functor $wF$ gives a fully faithful $\infty$-functor, so at the level of hammock localizations it induces a weak equivalence
\[
L^HW_{\mathcal{C}}(X,Y)\stackrel{\sim}{\rightarrow}L^HW_{\mathcal{D}}(F(X),F(Y))
\]
for every two objects $X$ and $Y$ of $\mathcal{C}$, in particular
\[
L^HW_{\mathcal{C}}(X,X)\stackrel{\sim}{\rightarrow}L^HW_{\mathcal{D}}(F(X),F(X)).
\]
\end{proof}
\begin{cor}\label{C: surjprop}
A surjection of props $\varphi:P\twoheadrightarrow Q$ induces weak equivalences of homotopy automorphisms
\[
L^HwQ-Alg(X,X)\stackrel{\sim}{\rightarrow}L^HwP-Alg(\varphi^*X,\varphi^*X)
\]
for every $Q$-algebra $X$.
\end{cor}

\subsection{Moduli problems of algebraic structures} \label{S:ModuliProblems}

Moduli spaces of algebraic structures were originally defined as simplicial sets, in the setting of simplicial operads \cite{Rez}. This notion can be extended to algebras over differential graded props as follows (see \cite{Yal2}):
\begin{defn}
Let $P_{\infty}$ be a cofibrant prop and $X$ be a complex. The moduli space of $P_{\infty}$-algebra structures on $X$ is the simplicial set $P_{\infty}\{X\}$ defined by
\[
P_{\infty} = Mor_{Prop}(P_{\infty},End_X\otimes\Omega_{\bullet}),
\]
where the prop $End_X\otimes\Omega_{\bullet}$ is defined by $End_X\otimes\Omega_{\bullet})(m,n)=End_X(m,n)\otimes\Omega_{\bullet}$ and $\Omega_{\bullet}$ is the Sullivan cdga of the standard simplex $\Delta^{\bullet}$.
\end{defn}
This simplicial set enjoys the following properties:
\begin{prop}\label{P: classif}
(1) The simplicial set $P_{\infty}\{X\}$ is a Kan complex and
\[
\pi_0P_{\infty}\{X\} = [P_{\infty},End_X]_{Ho(Prop)}
\]
is the set of homotopy classes of $P_{\infty}$-algebra structures on $X$.

(2) Any weak equivalence of cofibrant props $P_{\infty}\stackrel{\sim}{\rightarrow}Q_{\infty}$ induces a weak equivalence of Kan complexes $Q_{\infty}\{X\}\stackrel{\sim}{\rightarrow}P_{\infty}\{X\}$.
\end{prop}

For the remaining part of the paper, we will restrict ourselves to dg properads.
Cofibrant resolutions of a properad $P$ can always be obtained as a cobar construction $\Omega(C)$
on some coproperad $C$ (which is usually the bar construction or the Koszul dual if $P$ is Koszul).
Given a cofibrant resolution $\Omega(C)\stackrel{\sim}{\rightarrow}P$ of $P$ and another properad $Q$,
one considers the convolution dg Lie algebra $Hom_{\Sigma}(\overline{C},Q)$ consisting in morphisms
of $\Sigma$-biobjects from the augmentation ideal of $C$ to $Q$. The Lie bracket is the antisymmetrization
of the convolution product. This convolution product is defined similarly to the convolution product of
morphisms from a coalgebra to an algebra, using the infinitesimal coproduct of $C$ and the infinitesimal
product of $Q$. The total complex $Hom_{\Sigma}(\overline{C},Q)$ is a complete dg Lie algebra.
More generally, if $P$ is a properad with minimal model $(\mathcal{F}(s^{-1}C),\partial)\stackrel{\sim}{\rightarrow}P$ for a certain homotopy coproperad $C$ (see \cite[Section 4]{MV1} for the definition of homotopy coproperads), and $Q$ is any properad, then the complex $Hom_{\Sigma}(\overline{C},Q)$ is a complete dg $L_{\infty}$ algebra. The relationship between the simplicial mapping space of morphisms $P_{\infty}\rightarrow Q$ and the Lie theory of the convolution $L_{\infty}$-algebra $Hom_{\Sigma}(\overline{C},Q)$ is provided by the following theorem:
\begin{thm}(cf. \cite[Theorem 2.10,Corollary 4.21]{Yal2})\label{T:Yal2}
Let $P$ be a dg properad equipped with a minimal model $P_{\infty}:=(\mathcal{F}(s^{-1}C),\partial)\stackrel{\sim}{\rightarrow}P$ and $Q$ be a dg properad. The simplicial presheaf
\[
\underline{Map}(P_{\infty},Q):A\in dgArt_{\mathbb{K}}^{aug}\mapsto Map_{Prop}(P_{\infty},Q\otimes A)
\]
is equivalent to the simplicial presheaf
\[
\underline{MC_{\bullet}}(Hom_{\Sigma}(\overline{C},Q)):A\in dgArt_{\mathbb{K}}^{aug}\mapsto MC_{\bullet}(Hom_{\Sigma}(\overline{C},Q)\otimes A)
\]
associated to the complete $L_{\infty}$-algebra $Hom_{\Sigma}(\overline{C},Q)$.
\end{thm}
By \cite[Corollary 2.4]{Yal1}, the tensor product $MC_{\bullet}(Hom_{\Sigma}(\overline{C},Q)\otimes A)$ does not need to be completed because $A$ is artinian. In the following, we will also consider the simplicial presheaf
\[
\underline{MC_{\bullet}^{fmp}}(Hom_{\Sigma}(\overline{C},Q)):A\in dgArt_{\mathbb{K}}^{aug}\mapsto MC_{\bullet}(Hom_{\Sigma}(\overline{C},Q)\otimes m_A),
\]
where $m_A$ is the maximal ideal of $A$. This presheaf is the formal moduli problem associated to $Hom_{\Sigma}(\overline{C},Q)$ under Lurie's equivalence theorem.

We will use such results in the case where $Q=End_X$. In this case, the simplicial presheaf $\underline{Map}(P_{\infty},Q)$ will be noted $\underline{P_{\infty}\{X\}}$. Let us note that $\underline{P_{\infty}\{X\}}$ can be alternately defined by
\[
A\mapsto P_{\infty}\otimes A\{X\otimes A\}_{Mod_A},
\]
where $P_{\infty}\otimes A\{X\otimes A\}_{Mod_A}$ is the mapping space of dg props in $A$-modules $Map(P_{\infty}\otimes A,End_{X\otimes A}^{Mod_A})$ where $End_{X\otimes A}^{Mod_A}$ is the endormorphism prop of $X\otimes A$ taken in the category of $A$-modules. That is, the simplicial moduli space of $P_{\infty}$-algebra structures on $X\otimes A$ in the category of $A$-modules. Indeed, since $Mod_A$ is tensored over $Ch_{\mathbb{K}}$, on can make $P_{\infty}$ act on $A$-modules either by morphisms of dg props in $A$-modules from $P_{\infty}\otimes A$ to the endomorphism prop defined by the internal hom of $Mod_A$, or by morphisms of dg props from $P_{\infty}$ to the endomorphism prop defined by the external hom of $Mod_A$.

This theorem applies in particular to the case of a Koszul properad, which includes for instance Frobenius algebras, Lie bialgebras and their variants such as involutive Lie bialgebras in string topology. It applies also to more general situations such as the properad $Bialg$ encoding associative and coassociative bialgebras, which is homotopy Koszul \cite[Proposition 41]{MV1}.

By Proposition~\ref{P: classif}, the simplicial set $\underline{P_{\infty}\{X\}}(A)$ classifies $P_{\infty}\otimes A$-algebra structures on $X\otimes A$. However, the simplicial presheaf $\underline{P_{\infty}\{X\}}$ is not a formal moduli problem, since $\underline{P_{\infty}\{X\}}(\mathbb{K})$ is in general not contractible. The formal moduli problem $\underline{P_{\infty}\{X\}}^{\psi}$ controlling the formal deformations of a $P_{\infty}$-algebra structure $\psi:P_{\infty}\rightarrow End_X$ on $X$ is defined, on any augmented artinian cdga $A$, by the homotopy fiber
\[
\underline{P_{\infty}\{X\}}^{\psi}(A)=hofib(\underline{P_{\infty}\{X\}}(A)\rightarrow \underline{P_{\infty}\{X\}}(\mathbb{K}))
\]
taken over the base point $\psi$, the map being induced by the augmentation $A\rightarrow\mathbb{K}$.
The twisting of $Hom_{\Sigma}(\overline{C},End_X)$ by a properad morphism $\psi:P_{\infty}\rightarrow End_X$ is the deformation complex of $\psi$, and we have an isomorphism
\[
g_{P,X}^{\psi} = Hom_{\Sigma}(\overline{C},End_X)^{\psi} \cong Der_{\psi}(\Omega(C),End_X)
\]
where the right-hand term is the complex of derivations with respect to $\psi$ \cite[Theorem 12]{MV2}.
\begin{prop}
The tangent $L_{\infty}$-algebra of the formal moduli problem $\underline{P_{\infty}\{X\}}^{\psi}$ is given by
\[
g_{P,X}^{\psi} = Hom_{\Sigma}(\overline{C},End_X)^{\psi}.
\]
\end{prop}
\begin{proof}
Let $A$ be an augmented artinian cdga.
By Theorem~\ref{T:Yal2}, we have the homotopy equivalences
\begin{eqnarray*}
\underline{P_{\infty}\{X\}}^{\psi}(A) & \simeq & hofib(\underline{MC_{\bullet}}(g_{P,X})(A)\rightarrow \underline{MC_{\bullet}}(g_{P,X})(\mathbb{K})) \\
 & = & hofib(MC_{\bullet}(g_{P,X}\otimes A)\rightarrow MC_{\bullet}(g_{P,X})) \\
 & \simeq & MC_{\bullet}(hofib_{L_{\infty}}(g_{P,X}\otimes A\rightarrow g_{P,X}))
\end{eqnarray*}
where $hofib_{L_{\infty}}(g_{P,X}\otimes A\rightarrow g_{P,X})$ is the homotopy fiber, over the Maurer-Cartan element $\psi$, of the $L_{\infty}$-algebra morphism $g_{P,X}\otimes A\rightarrow g_{P,X}$ given by the tensor product of the augmentation $A\rightarrow\mathbb{K}$ with $g_{P,X}$. This homotopy fiber is nothing but $g_{P,X}^{\varphi}\otimes m_A$, where $m_A$ is the maximal ideal of $A$, so there is an equivalence of formal moduli problems
\[
\underline{P_{\infty}\{X\}}^{\psi} \simeq \underline{MC_{\bullet}^{fmp}}(g_{P,X}^{\psi}).
\]
By Lurie's equivalence theorem, this means that $g_{P,X}^{\psi}$ is the Lie algebra of the formal moduli problem $\underline{P_{\infty}\{X\}}^{\psi}$.
\end{proof}

In this paper, we will also be interested in derived algebraic groups in the sense of \cite{Fra} and deformation complexes governing homotopy automorphisms of algebras over properads. Let us first note that for any cdga $A$, the category $Mod_A$ is a (cofibrantly generated) symmetric monoidal model category tensored over chain complexes, so that one can define the category $P_{\infty}-Alg(Mod_A)$ of $P_{\infty}$-algebras in $Mod_A$.
\begin{defn}\label{D:haut}
For a chain complex $X$, we denote $\underline{haut}(X)$  
the derived algebraic group of homotopy automorphisms of the underlying complex $X$ 
taken in the model category of chain complexes. It is defined by
\[
A\mapsto haut_{Mod_A}(X\otimes A),
\]
where $haut_{Mod_A}$ is the simplicial monoid of homotopy automorphisms in the category of $A$-modules. Since for every $A$ we have $haut(X)\sim haut_{Mod_A}(X\otimes A)$, it can be defined alternately by the constant functor
\[
A\mapsto haut(X).
\]

Further, let  $O$ be an operad, let $O_\infty$ be a cofibrant resolution of $O$, and let $(X,\psi: O_\infty\to End_X)$ be an $O$-algebra structure on $X$. 
We define $\underline{haut}_{O_{\infty}}(X,\psi)$ to be the derived algebraic group associated to the simplicial monoid of homotopy automorphisms of $(X,\psi)$ in the model category of $O_{\infty}$-algebras, that is  the simplicial sub-monoid of self weak equivalences in the usual homotopy mapping space $Map_{O_{\infty}-Alg}(X,X)$.  Such simplicial monoids of self weak-equivalences  are defined in any model category, see for instance \cite[Chapter 17]{Hir}. The (weak) simplicial presheaf is defined by
\[
A\mapsto haut_{O_{\infty}}(X\otimes A,\psi\otimes A)_{Mod_A}
\]
where $haut_{O_{\infty}}(X\otimes A,\psi\otimes A)_{Mod_A}$ is the simplicial monoid of homotopy automorphisms
of $(X\otimes A,\psi\otimes A)\in O_{\infty}-Alg(Mod_A)$.
\end{defn}
By \cite{DK} we have a homotopy equivalence $\underline{haut}_{O_{\infty}}(X,\psi)\, \simeq \, \underline{L^HwO_{\infty}-Alg}(X,\psi)$,
where $\underline{L^HwO_{\infty}-Alg}(X,\psi)$ is the derived algebraic group of homotopy automorphisms in the Dwyer-Kan hammock localization of $O_{\infty}$-algebras with respect to quasi-isomorphisms \cite{DK}.
When $P$ is a prop but not an operad, we do not have a model structure on $P_\infty$-algebras to define $\underline{haut}_{P}(X,\psi)$ but we can still define the
hammock localization $\underline{L^HwP_{\infty}-Alg}(X,\psi)$.
To such a derived algebraic group, one can associate a formal moduli problem $B\underline{L^HwP_{\infty}-Alg}(X,\psi)$ given by the strictification (see \cite[Section I.2.3.1]{Ane}) of the weak simplicial presheaf
\[
A\in dgArt_{\mathbb{K}}^{aug}\mapsto \underline{L^HwP_{\infty}-Alg(Mod_A)}(X\otimes A,\psi\otimes A),
\]
where $B$ is the classifying complex functor for simplicial monoids (see \cite{Fra} for the definition of such a formal moduli problem in the case of algebras over operads).
We introduce two related and useful classification space constructions. The assignment
\[
A\mapsto wP_{\infty}-Alg(Mod_A),
\]
where the $w(-)$ stands for the subcategory of weak equivalences, defines a weak presheaf of categories in the sense of \cite[Definition I.56]{Ane}, sending a morphism $A\rightarrow B$ to the functor $-\otimes_A B$, which is symmetric monoidal, hence lifted at the level of $P_{\infty}$-algebras. This weak presheaf can be strictified into a presheaf of categories. Applying the nerve functor then defines a simplicial presheaf of Dwyer-Kan classification spaces that we note $\underline{\mathcal{N}wP_{\infty}-Alg}$. The simplicial presheaf associated to $A\mapsto Mod_A$ is the simplicial presheaf of quasi-coherent modules of \cite[Definition 1.3.7.1]{TV} that we note $\underline{\mathcal{N}wCh_{\mathbb{K}}}$. The loop construction on these simplicial presheaves is defined as follows.
The loop space on $\underline{\mathcal{N}wP_{\infty}-Alg}$ based at a $P_{\infty}$-algebra $(X,\psi)$ is the strictification of the weak simplicial presheaf
\[
\Omega_{(X,\psi)}\underline{\mathcal{N}wP_{\infty}-Alg}:A\mapsto \Omega_{(X\otimes A,\psi\otimes A)}\mathcal{N}wP_{\infty}-Alg(Mod_A).
\]
The loop space on $\underline{P_{\infty}\{X\}}$ based at a morphism $\psi:P_{\infty}\rightarrow End_X$ is the strictification of the weak simplicial presheaf
\[
\Omega_{\psi}\underline{P_{\infty}\{X\}}:A\mapsto \Omega_{\psi\otimes A}P_{\infty}\otimes A\{X\otimes A\}_{Mod_A}.
\]
One defines similarly the loop space $\Omega_{\psi}\underline{P_{\infty}\{X\}}^{\psi}$ on the formal moduli problem $\underline{P_{\infty}\{X\}}^{\psi}$.

\medskip

Rezk's homotopy pullback theorem \cite{Rez}, generalized to algebras over cofibrant dg props in \cite[Theorem 0.1]{Yal3}, 
states that the simplicial moduli space of $P_{\infty}$-algebras structures on $X$ is the homotopy fiber of a forgetful
map between the Dwyer-Kan classification space of the $\infty$-category of $P_{\infty}$-algebras and the classification space 
of the $\infty$-category of chain complexes, over the base point given by a $P_{\infty}$-algebra structure on $X$. 
This result points towards two crucial ideas about the deformation theory of algebras over a prop, which will be crucial
in our paper but are of course of independent interest as well.

First, in derived algebraic geometry, a Zariski open immersion of derived Artin stacks $F\hookrightarrow G$ induces a weak equivalence between the tangent complex over a given point of $F$ and the tangent complex over its image in $G$ \cite{TV}. We would like a somehow similar statement about the tangent Lie algebras of our formal moduli problems of algebraic structures, when an $\infty$-category of algebras ``embeds'' into another $\infty$-category of algebras. The idea is that \cite[Theorem 0.1]{Yal3} tells us that the formal moduli problem $\underline{P_{\infty}\{X\}}^{\psi}$ is the ``tangent space'' over $(X,\psi)$ to the Dwyer-Kan classification space of the $\infty$-category of $P_{\infty}$-algebras. An ``immersion'' $F:P_{\infty}-Alg\hookrightarrow Q_{\infty}-Alg$ should induce an equivalence between the formal moduli problem of formal deformations of $(X,\psi)$ in $P_{\infty}$-algebras and the formal moduli problem of formal deformations of $F(X,\psi)$ in $Q_{\infty}$-algebras, hence an equivalence between their tangent $L_{\infty}$-algebras. Here the word ``immersion'' has to be understood as ``fully faithful conservative $\infty$-functor''.

Second, the result \cite[Theorem 0.1]{Yal3} implies the existence of a long exact sequence relating the homotopy type of $\underline{P_{\infty}\{X\}}$ to the homotopy type of the homotopy automorphisms $haut_{P_{\infty}}(X,\psi)$ of $(X,\psi)$ inside the $\infty$-category of $P_{\infty}$-algebras. These homotopy automorphisms form a derived algebraic group \cite{Fra} with an associated $L_{\infty}$-algebra $Lie(\underline{haut}_{P_{\infty}}(X,\psi))$, thus there should be a homotopy fiber sequence of $L_{\infty}$-algebras relating $Lie(\underline{haut}_{P_{\infty}}(X,\psi))$ to the tangent $L_{\infty}$-algebra of $\underline{P_{\infty}\{X\}}^{\psi}$.

To formalize properly the two ideas above, we need the following generalization of Rezk's homotopy pullback theorem:
\begin{prop}\label{P: Rezkgen}
Let $P_{\infty}$ be a cofibrant prop and $X$ be a chain complex.

(1) The forgetful functor $P_{\infty}-Alg\rightarrow Ch_{\mathbb{K}}$ induces a homotopy fiber sequence
\[
\underline{P_{\infty}\{X\}}\rightarrow \underline{\mathcal{N}wP_{\infty}-Alg} \rightarrow \underline{\mathcal{N}wCh_{\mathbb{K}}}
\]
of simplicial presheaves over augmented artinian cdgas, taken over the base point $X$.

(2) This homotopy fiber sequence induces an equivalence of derived algebraic groups
\[
\Omega_{\psi}\underline{P_{\infty}\{X\}}\simeq \underline{L^HwP_{\infty}-Alg}(X,\psi).
\]
\end{prop}
\begin{proof}
(1) We explain briefly how \cite[Theorem 0.1]{Yal3} can be transposed in the context of simplicial presheaves of cdgas. The identification of the homotopy fiber of the forgetful map
\[
\underline{\mathcal{N}wP_{\infty}-Alg}\rightarrow \underline{\mathcal{N}wCh_{\mathbb{K}}}
\]
with the simplicial presheaf $\underline{P_{\infty}\{X\}}$ follows from the two following facts. First, we can identify it pointwise with $Map(P_{\infty}\otimes A, End_{X\otimes A}^{Mod_A})$, where $End_{X\otimes A}^{Mod_A}$ is the endomorphism prop of $X\otimes A$ in the category of $A$-modules. This comes from the extension of \cite[Theorem 0.1]{Yal3} to $P_{\infty}$-algebras in $A$-modules, which holds true trivially by replacing chain complexes by $A$-modules as target category in the universal functorial constructions of \cite[Section 2.2]{Yal3} ($A$-modules are equipped with exactly the same operations than chain complexes which are needed in this construction: directs sums, suspensions, twisting cochains). Second, there is an isomorphism of simplicial sets
\[
Map(P_{\infty}\otimes A, End_{X\otimes A}^{Mod_A}) \cong \underline{P_{\infty}\{X\}}(A)
\]
(see for instance \cite[Section 3]{Yal2}).

(2) The connected components of $\mathcal{N}wP_{\infty}-Alg$ are the classifying complexes of homotopy automorphisms of $P_{\infty}$-algebras, hence a decomposition
\[
\mathcal{N}wP_{\infty}-Alg \cong \coprod_{[Y,\phi]\in\pi_0\mathcal{N}wP_{\infty}-Alg}BL^HwP_{\infty}-Alg(Y,\phi)
\]
where $[Y,\phi]$ ranges over quasi-isomorphism classes of $P_{\infty}$-algebras. The moduli space $P_{\infty}\{X\}$ thus decomposes into
\[
P_{\infty}\{X\} \cong \coprod_{[X,\psi]\in\pi_0\mathcal{N}wP_{\infty}-Alg}BL^HwP_{\infty}-Alg(X,\psi)
\]
where $[X,\psi]$ ranges over quasi-isomorphism classes of $P_{\infty}$-algebras having $X$ as underlying complex.
This induces a homotopy equivalence
\[
\Omega_{\psi}P_{\infty}\{X\} \simeq L^HwP_{\infty}-Alg(X,\psi)
\]
which sends the constant loop at $\psi$ to $id_{(X,\psi)}$ and any loop to a homotopy automorphism of $(X,\psi)$.
The naturality of this homotopy equivalence follows from the commutativity of the following square for every morphism $f:A\rightarrow B$ of augmented artinian cdgas:
\[
\xymatrix{
\Omega_{\psi}\underline{P_{\infty}\{X\}}(A)\ar[r]^-{\sim}\ar[d]_-{-\otimes_A B} & L^HwP_{\infty}-Alg(Mod_A)(X\otimes A,\psi\otimes A)\ar[d]_-{-\otimes_A B} \\
\Omega_{\psi}\underline{P_{\infty}\{X\}}(B)\ar[r]^-{\sim} & L^HwP_{\infty}-Alg(Mod_B)(X\otimes B,\psi\otimes B)
}
\]

Hence we get the desired equivalence
\[
\Omega_{\psi}\underline{P_{\infty}\{X\}} \simeq \underline{L^HwP_{\infty}-Alg}(X,\psi).
\]
\end{proof}
The relation between $\Omega_{\psi}\underline{P_{\infty}\{X\}}^{\psi}$ and the derived algebraic group of homotopy automorphisms of $(X,\psi)$ is then given by the following homotopy fiber sequence:
\begin{prop}\label{P: hofibgroups}
There is a homotopy fiber sequence of derived algebraic groups
\[
\Omega_{\psi}\underline{P_{\infty}\{X\}}^{\psi}\rightarrow \underline{L^HwP_{\infty}-Alg}(X,\psi) \rightarrow \underline{haut}(X),
\]
hence a homotopy fiber sequence of the associated $L_{\infty}$-algebras
\[
g_{P,X}^{\psi}\rightarrow Lie(\underline{L^HwP_{\infty}-Alg}(X,\psi))\rightarrow Lie(\underline{haut}(X)).
\]
\end{prop}
\begin{proof}
Let $A$ be an augmented artinian cdga.
On the one hand, the simplicial set $\underline{P_{\infty}\{X\}}^{\psi}(A)$ is identified to the Kan subcomplex of $\underline{P_{\infty}\{X\}}(A)$ generated by the set of points $\phi:P_{\infty}\otimes A\rightarrow End_{X\otimes A}^{Mod_A}$ such that their reduction modulo $A$ satisfies $\phi\otimes_A\mathbb{K}\sim \psi$, where $\mathbb{K}$ is an $A$-module via the augmentation $A\rightarrow\mathbb{K}$. That is, the simplicial set $\underline{P_{\infty}\{X\}}^{\psi}(A)$ is the moduli space of $A$-linear deformations of $\psi$.

On the other hand, the homotopy fiber of the map
\[
-\otimes_A\mathbb{K}:\underline{L^HwP_{\infty}-Alg}(X,\psi)(A) \rightarrow \underline{haut(X)},
\]
which is the reduction modulo $A$ induced by the augmentation $A\rightarrow\mathbb{K}$,
is identified with the Kan subcomplex of $\underline{L^HwP_{\infty}-Alg}(X,\psi)(A)=L^HwP_{\infty}-Alg(Mod_A)(X\otimes A,\psi\otimes A)$ generated by the set of points $f:(X\otimes A,\psi\otimes A)\stackrel{\sim}{\leftarrow}\bullet\stackrel{\sim}{\rightarrow}(X\otimes A,\psi\otimes A)$ such that $f\otimes_A\mathbb{K}\sim id_X$ (recall that vertices of the hammock localization are given by finite zigzags of weak equivalences).

Now recall from Proposition~\ref{P: Rezkgen} that there is a simplicial homotopy equivalence
\[
\Omega_{\psi}\underline{P_{\infty}\{X\}}(A)\cong \Omega_{\psi\otimes A}P_{\infty}\otimes A\{X\otimes A\}_{Mod_A}\simeq L^HwP_{\infty}-Alg(Mod_A)(X\otimes A,\psi\otimes A)
\]
sending the constant loop at $\psi\otimes A$ to $id_{(X\otimes A,\psi\otimes A)}$ and generally any loop to a homotopy automorphism of $(X\otimes A,\psi\otimes A)$.
Such a loop is given by a pointed simplicial map $\partial\Delta^1\rightarrow P_{\infty}\otimes A\{X\otimes A\}_{Mod_A}$ sending the point $0$ to $\psi\otimes A$. At the level of geometric realizations of these Kan complexes, this is equivalent to define a pointed topological map $l:S^1\rightarrow |P_{\infty}\otimes A\{X\otimes A\}_{Mod_A}|$ such that $l(0)=\psi\otimes A$, where $0$ is the base point of $S^1$. Here we use the fact that the geometric realization is a Quillen equivalence of model categories between simplicial sets with the Kan-Quillen model structure and topological spaces with the model structure induced by weak homotopy equivalences and Serre fibrations.
The space $\Omega_{\psi\otimes A}|\underline{P_{\infty}\{X\}}^{\psi}(A)|$ is then the subspace of loops $l:S^1\rightarrow |P_{\infty}\otimes A\{X\otimes A\}_{Mod_A}|$ such that for every point $p$ of $S^1$, we have $l(p)\otimes_A\mathbb{K}\sim\psi$.
By the geometric realization of the homotopy equivalence above, such loops $l$ are sent to homotopy automorphisms $f_l:(X\otimes A,\psi\otimes A)\stackrel{\sim}{\leftarrow}\bullet\stackrel{\sim}{\rightarrow}(X\otimes A,\psi\otimes A)$ of $(X\otimes A,\psi\otimes A)$ such that $f_l\otimes_A\mathbb{K}\sim id_X$. That is, we get a homotopy equivalence
\[
\Omega_{\psi\otimes A}|\underline{P_{\infty}\{X\}}^{\psi}(A)|\sim |hofib(-\otimes_A\mathbb{K})|,
\]
hence the simplicial homotopy equivalence above restricts to
\[
\Omega_{\psi\otimes A}\underline{P_{\infty}\{X\}}^{\psi}(A)\sim hofib(-\otimes_A\mathbb{K}).
\]
The fiber sequence of $L_{\infty}$-algebras then follows by taking the Lie algebras of the derived algebraic groups as in \cite{Fra}. Concerning the identification of the fiber with $g_{P,X}^{\psi}$, recall that the Lie algebra of loops over a formal moduli problem gives the tangent Lie algebra of this formal moduli problem.
\end{proof}
\begin{rem}
Let us note that $g_{P,X}^{\psi}$ is thus different from the Lie algebra of the loop space $\Omega_{\psi}\underline{P_{\infty}\{X\}}\simeq \underline{L^HwP_{\infty}-Alg}(X,\psi)$. The later will be described in Section 3.
\end{rem}
Relying on these results, we prove the following comparison result:
\begin{thm}\label{T:fiberseqDefwith+}
Let $F:P_{\infty}-Alg\rightarrow Q_{\infty}-Alg$ be a fully faithful and conservative $\infty$-functor inducing functorially in $A$, for every augmented artinian cdga $A$, a fully faithful and conservative $\infty$-functor $F:P_{\infty}-Alg(Mod_A)\rightarrow Q_{\infty}-Alg(Mod_A)$. Then $F$ induces an equivalence of formal moduli problems
\[
\underline{P_{\infty}\{X\}}^{\psi} \sim  \underline{Q_{\infty}\{F(X)\}}^{F(\psi)},
\]
where $F(\psi)$ is the $Q_{\infty}$-algebra structure on the image $F(X,\psi)$ of $X,\psi$ under $F$, hence an equivalence of the associated $L_{\infty}$-algebras
\[
g_{P,X}^{\psi}\sim g_{Q,F(X)}^{F(\psi)}.
\]
\end{thm}
\begin{proof}
Let $F:P_{\infty}-Alg\rightarrow Q_{\infty}-Alg$ be a fully faithful $\infty$-functor inducing a fully faithful and conservative $\infty$-functor $F:P_{\infty}-Alg(Mod_A)\rightarrow Q_{\infty}-Alg(Mod_A)$ for every augmented artinian cdga $A$.
At the level of homotopy automorphisms, which can be seen equivalently as loops in Dwyer-Kan's classification spaces, this means that the commutative triangle
\[
\xymatrix{
\underline{\mathcal{N}wP_{\infty}-Alg} \ar[rr]^-{\mathcal{N}wF}\ar[dr]_-{\mathcal{N}w(U\circ F)} & & \underline{\mathcal{N}wQ_{\infty}-Alg} \ar[dl]_-{\mathcal{N}wU} \\
 & \underline{\mathcal{N}wCh_{\mathbb{K}}} & },
\]
where $\mathcal{N}w(-)$ is the simplicial nerve of the subcategory of weak equivalences (here quasi-isomorphisms) and $U$ is the forgetful functor, induces a commutative triangle
\[
\xymatrix{
\Omega_{(X,\psi)}\underline{\mathcal{N}wP_{\infty}-Alg} \ar[rr]^-{\sim}\ar[dr] & & \Omega_{F(X,\psi)}\underline{\mathcal{N}wQ_{\infty}-Alg} \ar[dl] \\
 & \Omega_X\underline{\mathcal{N}wCh_{\mathbb{K}}} & }
\]
where the horizontal arrow induced by $\mathcal{N}wF$ is a weak equivalence by Lemma~\ref{L: equivhaut}.

Consequently, we get the morphism of homotopy fiber sequences of derived algebraic groups
\[
\xymatrix{
\Omega_{\psi}\underline{P_{\infty}\{X\}}^{\psi} \ar[d]\ar[rr]^-{\sim} & & \Omega_{F(\psi)}\underline{Q_{\infty}\{F(X)\}}^{F(\psi)}\ar[d] \\
\Omega_{(X,\psi)}\underline{\mathcal{N}wP_{\infty}-Alg} \ar[rr]^-{\sim}\ar[dr] & & \Omega_{F(X,\psi)}\underline{\mathcal{N}wQ_{\infty}-Alg} \ar[dl] \\
 & \Omega_X\underline{\mathcal{N}wCh_{\mathbb{K}}} & },
\]
where the homotopy fibers are given by Proposition~\ref{P: hofibgroups}. The upper horizontal arrow is a weak equivalence of derived algebraic groups
\[
\Omega_{\psi}\underline{P_{\infty}\{X\}}^{\psi} \stackrel{\sim}{\rightarrow} \Omega_{F(\psi)}\underline{Q_{\infty}\{F(X)\}}^{F(\psi)},
\]
hence an equivalence of the associated $L_{\infty}$-algebras, which are the tangent $L_{\infty}$-algebras of the associated formal moduli problems. By Lurie's equivalence theorem, this means that we have a weak equivalence of formal moduli problems
\[
\underline{P_{\infty}\{X\}}^{\psi} \stackrel{\sim}{\rightarrow} \underline{Q_{\infty}\{F(X)\}}^{F(\psi)}
\]
as well.
\end{proof}

In the case of algebras and coalgebras over operads, we can get a similar result for a more general base category:
\begin{prop}\label{P:InvModuliunderWE}
(1) Let $P$ and $Q$ be two dg operads. Let $\mathcal{C}$ be a cofibrantly generated symmetric monoidal model category over $Ch_{\mathbb{K}}$. We assume that the categories of $P$-algebras and $Q$-algebras are equipped with their standard model structure. Let
\[
F:P-Alg(\mathcal{C})\rightarrow Q-Alg(\mathcal{C})
\]
be a fully faithful and conservative $\infty$-functor inducing functorially in $A$, for every augmented artinian cdga $A$, a fully faithful and conservative $\infty$-functor $F:P_{\infty}-Alg(Mod_A(\mathcal{C}))\rightarrow Q_{\infty}-Alg(Mod_A(\mathcal{C}))$, where $Mod_A(\mathcal{C})$ is the category of $A$-modules in $\mathcal{C}$.
Let $X\in\mathcal{C}$ be equipped with a $P$-algebra structure $\varphi:\rightarrow End_X$ and let $\phi:Q\rightarrow End_{F(X)}$
be its image under the functor $F$. Then there exists a homotopy equivalence of formal moduli problems
\[
\underline{P\{X\}_{\mathcal{C}}}^{\varphi} \simeq \underline{Q\{F(X)\}_{\mathcal{C}}}^{\phi}.
\]

(2) Let us now assume that $Q$ satisfies the necessary assumptions so that $Q-Cog(\mathcal{C})$ forms a model category.
Let
\[
F:P-Alg(\mathcal{C})\rightarrow Q-Cog(\mathcal{C})
\]
be a fully faithful $\infty$-functor. Let $X\in\mathcal{C}$ be equipped with a $P$-algebra structure $\varphi:\rightarrow End_X$ and let $\phi:Q\rightarrow coEnd_{F(X)}$
be its image under the functor $F$. Then there exists a homotopy equivalence of formal moduli problems
\[
\underline{P\{X\}_{\mathcal{C}}}^{\varphi} \simeq \underline{Q^{\vee}\{F(X)\}_{\mathcal{C}}}^{\phi},
\]
where $\underline{Q^{\vee}\{F(X)\}_{\mathcal{C}}}$ denotes the simplicial presheaf of $Q$-coalgebra structures on $X$.
\end{prop}
\begin{proof}
\noindent
\textbf{Proof of (1).}
The functor $F$ induces a commutative square of categories
\[
\xymatrix{
P-Alg(\mathcal{C})\ar[r]^-{U\circ F}\ar[d]_{F} & \mathcal{C}\ar[d]^-{=} \\
Q-Alg(\mathcal{C})\ar[r]_-{U} & \mathcal{C}
}
\]
where $U:Q-Alg(\mathcal{C})\rightarrow \mathcal{C}$ is the forgetful functor. In \cite{Mur}, Muro builds a stack of algebras over a nonsymmetric operad $\underline{Alg_{\mathcal{C}}}(P)$,
whose evaluation at an algebra $A$ is the classification space of cofibrant $P$-algebras in $A$-modules (that is, the nerve
of the subcategory of weak equivalences between cofibrant objects).
For every algebra $A$, there is a forgetful functor from $P$-algebras
in $A$-modules to $A$-modules, inducing a morphism of simplicial sets $\mathcal{N}wP-Alg(Mod_A)\rightarrow
\mathcal{N}wMod_A$, hence a morphism of stacks
$\underline{Alg_{\mathcal{C}}}(P)\rightarrow \underline{QCoh}$ (where $\underline{QCoh}$ is the stack of
quasi-coherent modules). For operads in chain complexes over a field of characteristic zero, this construction extends to symmetric operads. Thus the commutative square above also induces a commutative square of stacks
\[
\xymatrix{
\underline{\mathcal{N}wP-Alg(\mathcal{C})}\ar[r]\ar[d] & \underline{\mathcal{N}w\mathcal{C}}\ar[d]^-{=} \\
\underline{\mathcal{N}wQ-Alg(\mathcal{C})}\ar[r] & \underline{\mathcal{N}w\mathcal{C}}
}
\]
via this construction, hence a morphism of the induced homotopy fiber sequences of stacks.

In order to relate these homotopy fibers to moduli spaces of algebraic structures, we need a symmetric version of \cite[Theorem 4.6]{Mur}, which consists in a generalization of Rezk's homotopy pullback theorem \cite{Rez} to non-symmetric operads in monoidal model categories.
A key point in the proof of \cite[Theorem 4.6]{Mur} is that extension and restriction of structures along a weak equivalence of admissible non-symmetric operads form a Quillen pair. However, if one restricts to the setting of algebras in a cofibrantly generated symmetric monoidal model category tensored over another cofibrantly generated symmetric monoidal model category, then a weak equivalence of symmetric $\Sigma$-cofibrant operads similarly induces a Quillen pair between the corresponding categories of algebras by \cite[Theorem 16.A]{Fre3}. In this framework, \cite[Theorem 4.6]{Mur} can be extended to the symmetric case by following exactly the same argument line.
When the base category is $Ch_{\mathbb{K}}$, then the endomorphism operad of any object of $\mathcal{C}$, defined by the external dg hom of $\mathcal{C}$, is fibrant because fibrations of operads are defined aritywise and every chain complex over a field of characteristic zero is fibrant. In our situation we can thus drop the fibrancy-cofibrancy condition on $X$ stated in \cite[Theorem 4.6]{Mur}. The stack of $P$-algebra structures on an $A$-module $M$ is the homotopy pullback of the morphism
$\underline{Alg_{\mathcal{C}}}(P)\rightarrow \underline{QCoh}$
along the morphism $\mathbb{R}Spec(A)\rightarrow \underline{QCoh}$ representing $M$ \cite[Section 5]{Mur}.
This gives a stack version of Rezk's homotopy pullback theorem \cite{Rez}. Under our assumptions, this theorem  holds also in the symmetric case.

The proof concludes as follows. The commutative square of stacks
\[
\xymatrix{
\underline{\mathcal{N}wP-Alg(\mathcal{C})}\ar[r]\ar[d] & \underline{\mathcal{N}w\mathcal{C}}\ar[d]^-{=} \\
\underline{\mathcal{N}wQ-Alg(\mathcal{C})}\ar[r] & \underline{\mathcal{N}w\mathcal{C}}
}
\]
induces, for every choice of a basepoint $X\in\mathcal{C}$, a commutative square of derived algebraic groups
\[
\xymatrix{
\Omega_{(X,\varphi)}\underline{\mathcal{N}wP-Alg(\mathcal{C})}\ar[r]^-{\Omega\mathcal{N}w(U\circ F)}\ar[d]_{\Omega\mathcal{N}wF} & \Omega_X\underline{\mathcal{N}w\mathcal{C}}\ar[d]^-{=} \\
\Omega_{(F(X),\phi)}\underline{\mathcal{N}wQ-Alg(\mathcal{C})}\ar[r]_-{\Omega\mathcal{N}wU} & \Omega_X\underline{\mathcal{N}w\mathcal{C}}
}.
\]
The homotopy fibers of $\Omega\mathcal{N}w(U\circ F)$ and $\Omega\mathcal{N}wU$ are identified respectively with the formal moduli problems $\underline{P\{X\}_{\mathcal{C}}}^{\varphi}$ and $\underline{Q\{F(X)\}_{\mathcal{C}}}^{\phi}$.
The map $\Omega\mathcal{N}wF$ is a weak equivalence, so the homotopy fibers are weakly equivalent as well.

\noindent
\textbf{Proof of (2).} The arguments are the same than in the proof of (1), using that $Q$-coalgebra structures on a given object $X$ can be defined by operad morphisms $Q\rightarrow coEnd_X$, where $coEnd_X$ is the coendomorphism operad defined by $coEnd_X(n)=Hom_{\mathcal{C}}(X,X^{\otimes n})$.
\end{proof}

\section{The ``plus'' construction and the corresponding deformation theory}\label{S:Plus}

The plus construction, originally due to Merkulov, will be crucial later in the paper, 
to consider deformation complexes of algebraic structures which also encode compatible deformations of the differential. 
In this section, we explain the homotopical counterpart of this construction, that is, how the corresponding deformation complex 
controls the derived algebraic group of homotopy automorphisms of an algebra over a properad.

Recall from \cite{Mer2} that there is an endofunctor $(-)^+:Prop\rightarrow Prop$ on the category of properads which associates to any dg properad $P$ with presentation $\mathcal{F}(E)/(R)$ and differential $\delta$ a properad $P^+$ with presentation $\mathcal{F}(E^+)/(R)$ and differential $\delta^+$. The $\Sigma$-biobject $E^+$ is defined by $E^+(1,1)=E(1,1)\oplus\mathbb{K}[1]$ and $E^+(m,n)=E(m,n)$ otherwise. This means that we added a generating operation $u$ of degree $1$, with one input and one output. We would like this generator to twist the differential of a complex $X$ when we consider a $P^+$-algebra structure on $X$. For this, the differential $\delta$ is modified accordingly to have in particular $\delta^+(u)=u\otimes u\in E(1,1)\otimes E(1,1)$. Consequently, properad morphisms $\varphi^+:P^+\rightarrow End_{(X,d)}$ for a given complex $X$ with differential $d$ corresponds to properad morphisms $P\rightarrow End_{(X,d-\varphi^+(u))}$ for $X$ equipped with the twisted differential $d-\varphi^+(u)$. In particular, if $X$ is a graded vector space then $P^+$-algebra structures on $X$ equip $X$ simultaneously with a $P$-algebra structure and a compatible differential.
Let us reinterpret this construction by defining the following operad:
\begin{defn}
The operad of differentials $Di$ is the quasi-free operad $Di=(\mathcal{F}(E),\partial)$, where $E(1)=\mathbb{K}\delta$ with $\delta$ a generator of degree $1$, $E(n)=0$ for $n\neq 1$ and $\partial(\delta)=\delta\circ \delta$ is the operadic composition $\circ:Di(1)\otimes Di(1)\rightarrow Di(1)$.
\end{defn}
We will do an abuse of notation and still note $Di$ the properad freely generated by this operad.
\begin{lem}
Let $(V,d_V)$ be a complex.

(1) A $Di$-algebra structure $\phi:Di\rightarrow End_V$ on $V$ is a twisted complex $(V,d_V-\delta_V)$ where $\delta_V$ is the image of the operadic generator $\delta$ under $\phi$.

(2) A morphism of $Di$-algebras $f:(V,d_V-\delta_V)\rightarrow (W,d_W-\delta_W)$ is a chain morphism $f:(V,d_V)\rightarrow (W,d_W)$ which satifies moreover $f\circ (d_V-\delta_V) = (d_W-\delta_W)\circ f$ (it is a morphism of twisted complexes).
\end{lem}
\begin{proof}
(1) The morphism $\phi$ is entirely determined by the image of the generator $\delta$. Since
\[
Di(1)\rightarrow Hom(V,V)
\]
is a morphism of complexes, its compatibility with the differentials reads
\[
\phi(\partial(\delta)) = d_V\circ\delta_V  + \delta_V\circ d_V
\]
which gives the equation of twisting cochains
\[
\delta_V^2 = d_V\circ \delta_V + \delta_V\circ d_V,
\]
hence
\[
(d_V-\delta_V)^2 = d_V^2+\delta_V^2 - d_V\circ\delta_V - \delta_V\circ d_V = 0.
\]

(2) A $Di$-algebra structure on $V$ is given by a morphism $Di(V)\rightarrow V$, and a $Di$-algebra morphism $f:V\rightarrow W$
is a chain morphism fitting in the commutative square
\[
\xymatrix{
Di(V)\ar[r]^-{Di(f)}\ar[d] & Di(W)\ar[d] \\
V\ar[r]_-f & W
}.
\]
Since a $Di$-algebra structure is determined by the image of the generator $\delta$ via $Di(1)\otimes V\rightarrow V$, this amounts to the commutativity of the square
\[
\xymatrix{
Di(1)\otimes V\ar[r]^-{Di(1)\otimes f}\ar[d] & Di(1)\otimes W\ar[d] \\
V\ar[r]_-f & W
},
\]
which is exactly saying that $f$ is a morphism of twisted complexes.
\end{proof}
Let us note that $Di$ is a non-negatively graded quasi-free operad, hence a cofibrant operad, so we do not need to take a resolution of it to consider the associated simplicial presheaf of $Di$-algebra structures.
\begin{lem}
The based loop space at a given complex of the simplicial presheaf of $Di$-algebra structures is equivalent to the derived algebraic group of homotopy automorphisms of this complex, that is
\[
\Omega_{(V,d_V)}\underline{Di\{(V,d_V)\}}\sim \underline{haut(V,d_V)}
\]
where the loops are based at the trivial $Di$-algebra structure $Di\stackrel{0}{\rightarrow}End_{(V,d_V)}$.
\end{lem}
\begin{proof}
First, recall from Proposition 2.13(2) the equivalence
\[
\Omega_{(V,d_V)}\underline{Di\{(V,d_V)\}}\sim \underline{haut_{Di}(V,d_V)}.
\]
There is a fully faithful functor $i:Ch_{\mathbb{K}}\hookrightarrow Di-Alg$ which sends any chain complex $X$ to itself seen as a twisted complex with twisting $\delta_X=0$, that is, the $Di$-algebra $(X,0)$ for the trivial $Di$-algebra structure $0:Di\rightarrow End_X$. By definition, it sends quasi-isomorphisms of chain complexes to quasi-isomorphisms of complexes with trivial twisting, that is, it preserves weak equivalences. It is clearly conservative as well, and it satisfies $i(X\otimes A)=i(X)\otimes A$ for every augmented artinian cdga $A$, so it induces a weak equivalence
\[
\underline{haut_{Di}(V,d_V)} \sim \underline{haut(V,d_V)}.
\]
\end{proof}
Let us note that the plus construction obviously preserves quasi-isomorphisms of props. Indeed, the only effect of the plus construction on the cohomology of $P$ is to add a new generator of arity $(1,1)$ to $H^*P$ whose square is zero. Let $\varphi:P\stackrel{\sim}{\rightarrow}Q$ be a quasi-isomorphism of dg props whose collections of generators are respectively $E_P$ and $E_Q$, such that $E_P^+(1,1)=E(1,1)\oplus\mathbb{K}u_P$ and $E_Q^+(1,1)=E(1,1)\oplus\mathbb{K}u_Q$. Then $H^*(\varphi^+)$ sends $[u_P]$ to $[u_Q]$ (where $[-]$ denotes the cohomology class) and coincides with $H^*(\varphi)$ on the other generators. The only relations satisfied by $[u_P]$ and $[u_Q]$ are that they are both of square zero so $H^*(\varphi^+)$ is still a prop isomorphism, hence $\varphi^+$ is a quasi-isomorphism.

The functor $(-)^+$ takes quasi-free properads satisfying the conditions of \cite[Corollary 40]{MV2} to quasi-free properads satisfying the same conditions. By \cite[Corollary 40]{MV2}, such properads are cofibrant and by \cite[Theorem 42]{MV2},
every properad admits a cofibrant resolution of this form. So to any dg properad $P$ we can associate a cofibrant properad $P_{\infty}^+$ resolving $P^+$.
\begin{lem}
There is a homotopy cofiber sequence of properads
\[
Di\rightarrow P_{\infty}^+\rightarrow P_{\infty}.
\]
\end{lem}
\begin{proof}
The initial morphism $i:I\rightarrow P_{\infty}$ induces a morphism $i^+:I^+\rightarrow P_{\infty}^+$, and $I^+$ is nothing but $Di$ so we get our first morphism
\[
i^+:Di\rightarrow P_{\infty}^+,
\]
which is a morphism of cofibrant properads. Now we would like to compare $P_{\infty}^+$ and $P_{\infty}\vee Di$, where $\vee$ stands for the coproduct of properads (see \cite[Appendix A.3]{MV2} for its definition). Since the free properad functor $\mathcal{F}$ is a left adjoint, it preserves coproducts and thus comes with natural isomorphisms $\mathcal{F}(M\oplus N)\cong\mathcal{F}(M)\vee\mathcal{F}(N)$. If we consider the coproduct $P_{\infty}\vee Q_{\infty}$ of two quasi-free properads $P_{\infty}=(\mathcal{F}(M),\partial_P)$ and $Q_{\infty}=(\mathcal{F}(M),\partial_Q)$, then via the previous isomorphism we can define a differential on $\mathcal{F}(M\oplus N)$ by taking the derivation associated to
\[
\partial_P|_M\oplus \partial_Q|_N:M\oplus N\rightarrow \mathcal{F}(M)\oplus \mathcal{F}(N)\hookrightarrow \mathcal{F}(M\oplus N)
\]
by universal property of derivations and the fact that this morphism satisfies the twisting cochain equation.
In the case where $Q=Di$, it turns out that the free properad underlying $P_{\infty}^+$ is $\mathcal{F}(M\oplus\mathbb{K}d)$ and the differential on $P_{\infty}^+$ (see \cite{Mer2}) coincides with the one above, yielding a properad isomorphism
\[
P_{\infty}^+\cong P_{\infty}\vee Di.
\]
We deduce that the homotopy fiber of $i^+$ is the homotopy fiber of the canonical map $Di\rightarrow P_{\infty}\vee Di$. This map is a cofibration between cofibrant objects, so its homotopy cofiber is a strict cofiber, hence the cokernel of $i^+$ which is exactly $P_{\infty}$ as expected. 
\end{proof}
Now we fix a diagram of properad morphisms
\[
\xymatrix{
Di \ar@{^{(}->}[r] \ar[dr]^-0 & P_{\infty}^+ \ar@{->>}[r] \ar[d]^-{\psi^+} & P_{\infty} \ar[dl]_-{\psi} \\
 & End_X & }
\]
which induces a homotopy fiber sequence of formal moduli problems
\[
\underline{P_{\infty}\{X\}}^{\psi}\rightarrow \underline{P_{\infty}^+\{X\}}^{\psi^+}\rightarrow \underline{Di\{X\}}^0
\]
because the functor $Map(-,End_X)$ is a simplicial mapping space with fibrant target in a model category, so it sends homotopy colimits to homotopy limits. Equivalently this gives a homotopy fiber sequence of the associated $L_{\infty}$-algebras
\[
g_{P,X}^{\psi}\rightarrow g_{P^+,X}^{\psi^+}\rightarrow g_{Di,X}^0.
\]
We obtain
\begin{thm}\label{T:Def+=hAut}
There is a quasi-isomorphism of $L_{\infty}$-algebras
\[
g_{P^+,X}^{\psi^+}\simeq Lie(\underline{L^HwP_{\infty}-Alg}(X,\psi)),
\]
in particular
\[
g_{P^+,X}^{\psi^+}\simeq Lie(\underline{haut}_{P_{\infty}}(X,\psi))
\]
if $P_{\infty}$ is an operad (where $\underline{haut}_{P_{\infty}}(X,\psi)$ is the derived algebraic group from Definition~\ref{D:haut}).
\end{thm}
\begin{proof}
We have to build a comparison quasi-isomorphism between the homotopy fiber sequence
\[
g_{P,X}^{\psi}\rightarrow g_{P^+,X}^{\psi^+}\rightarrow g_{Di,X}^0
\]
and the homotopy fiber sequence
\[
g_{P,X}^{\psi}\rightarrow Lie(\underline{L^HwP_{\infty}-Alg}(X,\psi))\rightarrow Lie(\underline{haut}(X)).
\]
of Proposition 2.14.
For this, we consider the following commutative triangle
\[
\xymatrix{
\underline{L^HwP_{\infty}-Alg}(X,\psi)\ar[rr]^-{U}\ar[dr]_-{H_*\circ U} & & \underline{haut}(X)\ar[dl]^-{H^*} \\
 & \underline{Aut}(H_*X) & 
}
\]
where $U$ is the map induced by the forgetful functor, $H_*$ is the map induced by the homology functor and $\underline{Aut}(H_*X)$ is the constant functor $A\mapsto Aut(H_*X)$.
This triangle induces a morphism at the level of homotopy fiber sequences
\[
\xymatrix{
hofib_1 \ar[d]\ar[rr] & & hofib_2 \ar[d] \\
\underline{L^HwP_{\infty}-Alg}(X,\psi)\ar[rr]^-{U}\ar[dr]_-{H_*\circ U} & & \underline{haut}(X)\ar[dl]^-{H^*} \\
 & \underline{Aut}(H_*X) & 
}.
\]
Using the same arguments as in Proposition 2.14, one relates the homotopy automorphisms to loops over the appropriate moduli problems
\[
hofib_1\simeq \Omega_{\psi^+}\underline{P_{\infty}^+\{X\}}^{\psi^+}
\]
and
\[
hofib_2\simeq \Omega_0\underline{Di\{X\}}^0.
\]
Moreover, the base of these homotopy fiber sequences is a discrete space, which implies that the homotopy fiber and the total space have the same homotopy type: indeed, a homotopy fiber sequence $F\rightarrow E\rightarrow B$ induces a homotopy fiber sequence $\Omega B\rightarrow F\rightarrow E$. If $B$ is discrete, then $\Omega B$ is contractible, which implies that the map $F\rightarrow E$ is a homotopy equivalence.

The resulting commutative square
\[
\xymatrix{
\Omega_{\psi^+}\underline{P_{\infty}^+\{X\}}^{\psi^+}\ar[r]\ar[d]_-{\sim} & \Omega_0\underline{Di\{X\}}^0\ar[d]^-{\sim} \\
\underline{L^HwP_{\infty}-Alg}(X,\psi)\ar[r]_-U & \underline{haut}(X)
}
\]
induces a morphism of homotopy fiber sequences
\[
\xymatrix{
hofib\ar[d]\ar[r] & \Omega_{\psi^+}\underline{P_{\infty}^+\{X\}}^{\psi^+}\ar[r]\ar[d]_-{\sim} & \Omega_0\underline{Di\{X\}}^0\ar[d]^-{\sim} \\
\Omega_{\psi}\underline{P_{\infty}\{X\}}^{\psi}\ar[r] & \underline{L^HwP_{\infty}-Alg}(X,\psi)\ar[r]_-U & \underline{haut}(X)
}
\]
where the bottom fiber sequence is the one of Proposition 2.14. The upper homotopy fiber $hofib$ can be identified with those loops in
$\Omega_{\psi^+}\underline{P_{\infty}^+\{X\}}^{\psi^+}$ which preserves the trivial $Di$-algebra structure $0:Di\rightarrow End_X$ on $X$, that is, preserving the differential on $X$. So $hofib$ is nothing but $\Omega_{\psi}\underline{P_{\infty}\{X\}}^{\psi}$. We thus get the desired quasi-isomorphism of homotopy fiber sequences.
\end{proof}
A useful corollary for us will be the following identification of the tangent complex $T_A$ of
an $E_n$-algebra (Definition~\ref{D:HochTangent}):
\begin{cor}\label{L:gE2+=TA}
The $E_n$-Hochschild tangent complex $T_A$ of an $E_n$-algebra $A$ is naturally weakly equivalent as an $L_\infty$-algebra   to $g_{E_n^+,A}^{\psi^+}$:
\[
T_A\simeq Lie(\underline{haut}_{E_n}(A, \psi))\simeq g_{E_n^+,A}^{\psi^+},
\]
where $\psi^+$ is the $E_n^+$-algebra structure on $A$ trivially induced by its $E_n$-algebra structure $\psi:E_n\rightarrow End_A$ as above, and $\underline{haut}_{E_n}(A)$ is the derived algebraic group of homotopy automorphisms of $A$ as an $E_n$-algebra.
\end{cor}
\begin{proof}
According to \cite[Lemma 4.31]{Fra}, the homotopy Lie algebra of homotopy automorphisms $Lie(\underline{haut}_{E_n}(A, \psi))$
is equivalent to the tangent complex $T_A$ of $A$. Hence Theorem~\ref{T:Def+=hAut} implies the corollary.
\end{proof}
\begin{rem} Theorem~\ref{T:Def+=hAut} shows that the $+$ construction is crucial to study 
deformation of dg-algebras and not just deformations of algebraic structures on a fixed complex.
Let us illustrate the difference between these two moduli problems in a standard example: dg associative algebras.   We recall that the operad $E_1$ is weakly equivalent to the operad $Ass$ of associative dg algebras.
By  Lemma~\ref{L:gE2+=TA}, we have that $g_{Ass^+,A}^{\psi^+} \cong T_A$ and by Francis~\cite{Fra}, the upper fiber sequence of Theorem~\ref{T:Defcomplexesallagree} shows
that $T_A \cong Hom ( A^{\otimes >0}, A) [1]$ where the right hand side is a sub-complex of the standard Hochschild cochain complex~\cite{LV} shifted down by 1, with its standard Lie algebra structure (due to Gerstenhaber). On the other hand, a computation similar to the one of Corollary~\ref{L:gE2+=TA} using the operad of associative algebras instead of $E_n$ shows that 
$$g_{Ass,A}^{\psi} \cong Hom ( A^{\otimes >1}, A)[1]$$,
where the right hand side is just the subcomplex of the previous shifted Hochschild cochain complex where we have removed 
the $Hom(A,A)$ component. 

In general, if $\mathcal{O}$ is an operad, and $X$ is a non-graded $\mathcal{O}$-algebra,  
$\underline{\mathcal{O}_{\infty}\{X\}}^{\psi}$ is the moduli space of all $\mathcal{O}$-algebra structure
on $X$, while $\underline{\mathcal{O}_{\infty}^+\{X\}}^{\psi^+}$ is the moduli space of  $\mathcal{O}$-algebra structures
on $X$ up to automorphisms (if $X$ is a chain complex, then the same is true with dg-algebras structures and self-quasi-isomorphisms instead).
\end{rem}

\section{Bialgebras versus iterated coalgebras}

The main purpose of this section is to prove Theorem 0.1(1). Along the way, we establish general results about triple coresolutions and totalizations for dg coalgebras which dualize the ones obtained in \cite{Fre6} and are of independent interest.
Part (2) of Theorem 0.1 as well as Corollary 0.2 follows from an adaptation of Part (1) to the context of pointed algebras and will be proved in Section 5.

\subsection{Bar-cobar adjunction}

We recall the construction of bar-cobar adjunctions from \cite{FG} for algebras and coalgebras in a stable symmetric monoidal $(\infty,1)$-category $\mathcal{C}$.
Let $O$ be an operad with an augmentation $\epsilon:O\rightarrow I$, then the associated functor
\[
\epsilon^*=triv_O:\mathcal{C}\rightarrow O-Alg(\mathcal{C})
\]
has a left adjoint
\[
\mathcal{B}_O:O-Alg(\mathcal{C})\rightarrow\mathcal{C}.
\]
It follows from Barr-Beck-Lurie's theorem \cite[Theorem 4.7.4.5]{Lur2} that this bar construction can be enhanced in
\[
\xymatrix{
 & (\mathcal{B}_O\circ triv_O)-Cog(\mathcal{C}) \ar[dr]^-{oblv_O} & \\
O-Alg(\mathcal{C}) \ar[ur]^-{\mathcal{B}_O^{enh}} \ar[rr]_-{\mathcal{B}_O} & & \mathcal{C}
}
\]
where $oblv_O$ is the forgetful functor and $(\mathcal{B}_O\circ triv_O)-Cog(\mathcal{C})$ is the category of coalgebras over the comonad $\mathcal{B}_O\circ triv_O$. Now, by \cite[Lemma 3.3.4]{FG} there is a morphism of comonads
\[
\mathcal{B}_O\circ triv_O\stackrel{\sim}{\rightarrow} \overline{F^c_{BO}},
\]
where $BO$ is the operadic bar construction on $O$ and $\overline{F^c_{BO}}$ is the comonad whose category of coalgebras is the category $BO-Cog^{conil}(\mathcal{C})$ of conilpotent $BO$-coalgebras,  inducing a new commutative triangle (see \cite[Corollary 3.3.5]{FG})
\[
\xymatrix{
 & BO-Cog^{conil}(\mathcal{C}) \ar[dr]^-{oblv_{BO}} & \\
O-Alg(\mathcal{C}) \ar[ur]^-{\mathcal{B}_O^{enh}} \ar[rr]_-{\mathcal{B}_O} & & \mathcal{C}
}.
\]

In the dual situation, given a cooperad $P$ with a coaugmentation $\eta:P\rightarrow I$, there is an adjunction
\[
\eta^*=triv_P:\mathcal{C}\leftrightarrows P-Cog^{conil}(\mathcal{C}): \Omega_P
\]
which can be enhanced by Barr-Beck-Lurie's theorem in a commutative triangle
\[
\xymatrix{
 & (\Omega_P\circ triv_P)-Alg(\mathcal{C}) \ar[dr]^-{oblv_P} & \\
P-Cog^{conil}(\mathcal{C}) \ar[ur]^-{\Omega_P^{enh}} \ar[rr]_-{\Omega_P} & & \mathcal{C}
}
\]
where $oblv_P$ is the forgetful functor and $(\Omega_P\circ triv_P)-Alg(\mathcal{C})$ is the category of algebras over the monad $\Omega_P\circ triv_P$. By \cite[Lemma 3.3.9]{FG} there is a morphism of monads
\[
\overline{F_{\Omega P}}\rightarrow \Omega_P\circ triv_P,
\]
where $\Omega P$ is the operadic cobar construction on $P$ and $\overline{F_{\Omega P}}$ is the monad whose category of algebras is the category $\Omega P - Alg(\mathcal{C})$ of $\Omega P$-algebras, inducing a new commutative triangle (see \cite[Corollary 3.3.11]{FG})
\[
\xymatrix{
 & \Omega P - Alg(\mathcal{C}) \ar[dr]^-{oblv_{\Omega P}} & \\
P-Cog^{conil}(\mathcal{C}) \ar[ur]^-{\Omega_P^{enh}} \ar[rr]_-{\Omega_P} & & \mathcal{C}
}.
\]

If the operad $O$ is derived Koszul in the sense that the canonical map $\Omega\circ B (O)\stackrel{\sim}{\rightarrow}O$ is a weak equivalence (something always true in our setting), then the two constructions above return an adjunction of $\infty$-categories
\[
\mathcal{B}_O^{enh}:O-Alg(\mathcal{C})\rightleftarrows BO-Cog^{conil}(\mathcal{C}):\Omega_O^{enh}.
\]
Let us note that if the operad $O$ is Koszul in the usual sense, for instance when $O$ is an $E_n$-operad \cite{Fre4},
then one can replace $BO$ by the Koszul dual operad $O^{!}$ in the adjunction above.
One then wonders whether such an adjunction is an equivalence of $\infty$-categories. Such a result cannot be true in full generality, however, it is conjectured \cite[Conjecture 3.4.5]{FG} that it holds true when one restricts $O$-algebras to the subcategory of nilpotent $O$-algebras.
An $O$-algebra $A$ is nilpotent if there exists an integer $N$ such that the morphism
\[
P(n)\rightarrow Hom_{\mathcal{C}}(A^{\otimes n},A),
\]
which is the arity $n$ component of the operad morphism $P\rightarrow End_A$ defining the $P$-algebra structure of $A$,
is homotopic to the zero map. However, in our framework we will prove a version of this conjecture that holds for the notion of $0$-connected dg algebra, that is, an algebra whose underlying complex is concentrated in positive degrees.

\subsection{Plan of the proof of Theorem 0.1(1)}

We are going to use the following dual version of \cite[Theorem 4.7.4.5]{Lur2} in the comonadic setting:
\begin{thm}\label{T:BarBeckdual}
Let $\mathcal{C}$ be an $\infty$-category.

(1) An adjunction of $\infty$-categories
\[
F:\mathcal{C}\rightleftarrows\mathcal{D}:G
\]
induces a commutative triangle
\[
\xymatrix{
 & (F\circ G)-Cog(\mathcal{D}) \ar[dr]^-{oblv} & \\
 \mathcal{C} \ar[ur]^-{F^{enh}} \ar[rr]_-F & & \mathcal{D}
}
\]
where $(F\circ G)-Cog(\mathcal{D})$ is the category of coalgebras over the comonad $F\circ G$ in $\mathcal{D}$ and $oblv$ is the forgetful functor.

(2) Let us suppose that $\mathcal{C}$ admits totalizations $Tot:\mathcal{C}^{\Delta}\rightarrow\mathcal{C}$. If $F$ is conservative and preserves totalizations, then $F^{enh}$ is an equivalence of $\infty$-categories.
\end{thm}
\begin{rem}
As we will explain with more details in the next subsection, cosimplicial objects in any model category admits a totalization functor, so the theorem above applies in particular to $\infty$-categories associated to model categories.
\end{rem}
Since we know that the comonads $\mathcal{B}_O\circ triv_O$ and $\overline{F^c_{BO}}$ are weakly equivalent,
the associated $\infty$-categories of coalgebras are equivalent as well, and the later is nothing but the $\infty$-category of conilpotent $BO$-coalgebras. Thus it remains to prove that $\mathcal{B}_O$ is conservative and commutes with totalizations up to weak equivalence. Our goal is to prove such a result in the case where $\mathcal{C}=dgCog^{conil}$ is the category of conilpotent coassociative coalgebras in non-negatively graded chain complexes.

In order to get explicit models for the bar construction and the totalization, we work in the framework of model categories.
If $\mathcal{E}$ is a cofibrantly generated symmetric monoidal model category over a base category $\mathcal{C}$ and $O$ is a $\Sigma$-cofibrant operad in $\mathcal{C}$, then the category of $O$-algebras form a cofibrantly generated semi-model category, a slight weakening of the notion of model category sufficient for our purposes (see \cite[Section 12.1]{Fre3} for the definition and properties of semi-model categories, and \cite[Theorem 12.3.A]{Fre3} for the proof of this result). Moreover, the functor $\eta^*$ is a right Quillen functor and his left adjoint $\eta_{!}=\mathcal{B}_O$ is consequently a left Quillen functor \cite[Theorem 16.A]{Fre3}. In this context, one can show that $\mathcal{B}_O$ is weakly equivalent to the derived functor of indecomposables (see for instance \cite[Chapter 12]{LV}). Here we would like to consider dg operads acting on conilpotent dg coalgebras. For this, we need the following result:
\begin{prop}
The category $dgCog^{conil}$ is a cofibrantly generated closed symmetric monoidal model category.
\end{prop}
\begin{proof}
We know that $dgCog^{conil}$ forms a cofibrantly generated model category with colimits, weak equivalences and cofibrations created by the forgetful functor. We also know from \cite[Lemma 1.3.36]{AJ} that the chain tensor product of two conilpotent coalgebras is again a conilpotent coalgebra, so that the forgetful functor creates also the tensor product of $dgCog^{conil}$.
Since $Ch_{\mathbb{K}}$ satisfies the pushout-product axiom, which is by definition made out of pushout, coproduct, cofibrations and weak equivalences, this implies that $dgCog^{conil}$ satisfies the pushout-product axiom as well.
The existence of an internal hom bifunctor follows verbatim from the argument line of \cite[Theorem 2.5.1]{AJ}, given that:
\begin{itemize}
\item \cite[Proposition 2.1.10]{AJ} restricts to the full subcategory $dgCog^{conil}$ of $dgCog$;
\item \cite[Proposition 2.1.12]{AJ} restricts to the nilpotent case, either by following verbatim the proof of \cite[Section 2.1]{Yal0}, or by \cite[Proposition 1.20]{GJ} (since we work over a field of characteristic zero, every cooperad is exact);
\item the forgetful functor creates small colimits;
\item the forgetful functor creates the tensor product of conilpotent coalgebras.
\end{itemize}
This concludes the proof.
\end{proof}
Taking $\mathcal{E}=dgCog^{conil}$, $\mathcal{C}=Ch_{\mathbb{K}}$ and $O$ an $E_1$-operad,
we get a bar-cobar adjunction
\[
\mathcal{B}_{E_1}^{enh}:E_1-Alg(dgCog^{conil})\rightleftarrows E_1-Cog^{conil}(dgCog^{conil}):\Omega_{E_1}^{enh}
\]
which forms a Quillen adjunction of semi-model categories.
Our main goal is now to prove the following properties of the (not enhanced) bar construction functor:
\begin{thm}\label{T: BarConsTot}
The functor $\mathcal{B}_{E_1}$ is conservative when restricted to $0$-connected $E_1$-algebras, and commutes with totalizations up to weak equivalence when restricted to the essential image of a functorial fibrant resolution.
\end{thm}
Since we wants to apply the comonadic Barr-Beck-Lurie theorem, we work in an $\infty$-categorical setting where the $\infty$-category of fibrant objects of a model category is equivalent to the $\infty$-category associated to this model category, hence the assumptions of Theorem~\ref{T: BarConsTot}. More precisely, the inclusion of the full subcategory of fibrant objects into the model category induces a Dwyer-Kan equivalence of simplicial localizations \cite[Proposition 5.2]{DK}.

\noindent
\subsubsection{Conservativity.} We note
\[
\mathcal{B}_{E_1}^{\mathbb{K}}:E_1-Alg(Ch_{\mathbb{K}})\rightarrow Ch_{\mathbb{K}}
\]
the left adjoint of the functor $\eta^*$ induced by the augmentation $\eta:E_1\rightarrow I$ at the level of dg algebras, to distinguish it from the functor
\[
\mathcal{B}_{E_1}:E_1-Alg(dgCog^{conil})\rightarrow dgCog^{conil}
\]
induced by the augmentation $\eta:E_1\rightarrow I$ at the level of dg algebras in $dgCog^{conil}$.
Recall that this left adjoint is by construction weakly equivalent to the derived functor of the indecomposables, which is the Quillen homology complex computing Quillen homology of algebras over the corresponding operad \cite[Section 12.1]{LV}.

First we want to prove the conservativity of $\mathcal{B}_{E_1}^{\mathbb{K}}$ on $0$-connected $E_1$-algebras.
For this, we use the formalism of right modules over operads thoroughly developed in \cite{Fre3}, characterizing in particular bar constructions as functors naturally associated to right modules. We prove that checking the conservativity of $\mathcal{B}_{E_1}^{\mathbb{K}}$ boils down to the conservativity of the classical bar construction of dg associative algebras.

Given the projection morphism $\pi:E_1\stackrel{\sim}{\rightarrow}Ass$, according to \cite[Theorem 7.2.2]{Fre3},
there is a natural isomorphism
\[
S_{E_1}(B_{E_1},-)\circ\pi^*\cong S_{Ass}(\pi_{!}B_{E_1},-)
\]
where $\pi^*$ is the restriction of structures fitting in the adjunction
\[
\pi_{!}:E_1-Alg\rightleftarrows Ass-Alg:\pi^*
\]
and $\pi_{!}$ is the extension of structures fitting in the adjunction
\[
\pi_{!}:E_1-Mod\rightleftarrows Ass-Mod:\pi^*
\]
between right $Ass$-modules and right $E_1$-modules.
Since $\pi$ is a weak equivalence, by \cite[Theorem 16.B]{Fre3} the adjunction above between right modules is a Quillen
equivalence, so the weak equivalence of right $E_1$-modules $B_{E_1}\stackrel{\sim}{\rightarrow}\pi^*B_{Ass}$ corresponds by adjunction to a weak equivalence of right $Ass$-modules $\pi_{!}B_{E_1}\stackrel{\sim}{\rightarrow}B_{Ass}$.
Moreover, since $\pi_{!}$ is a left Quillen functor, the right module $\pi_{!}B_{E_1}$ is still cofibrant, so this is a weak equivalence of cofibrant right $Ass$-modules. By \cite[Theorem 15.1.A]{Fre3}, it induces consequently a natural weak equivalence
\[
S_{E_1}(B_{E_1},-)\circ\pi^*\cong S_{Ass}(\pi_{!}B_{E_1},-)\stackrel{\sim}{\rightarrow}S_{Ass}(B_{Ass},-),
\]
hence
\[
\mathcal{B}_{E_1}^{\mathbb{K}}\circ\pi^*\sim Bar
\]
where $\pi^*$ is the restriction of structures and $Bar$ is the operadic bar construction on dg associative algebras.
The functor $Bar$ is conservative on $0$-connected associative algebras (alternately, the classical bar construction is conservative on connected associative algebras), see for instance \cite[Chapter 11]{LV}.
By the natural weak equivalence above, this implies that $\mathcal{B}_{E_1}^{\mathbb{K}}\circ\pi^*$ is conservative on $0$-connected associative algebras as well.
By the Quillen equivalence between dg associative algebras and $E_1$-algebras, the conservativity of $\mathcal{B}_{E_1}^{\mathbb{K}}\circ\pi^*$ implies that the functor $\mathcal{B}_{E_1}^{\mathbb{K}}$ itself is conservative when restricted to $0$-connected $E_1$-algebras. Now recall from \cite[Section 3.3.5]{Fre3} that $\mathcal{B}_{E_1}$, as an extension functor adjoint to the restriction functor $\eta^*$, is obtained by a certain reflexive coequalizer
\[
E_1(A)\rightrightarrows A \rightarrow \mathcal{B}_{E_1}(A)
\]
in $dgCog^{conil}$ for every $A\in ob(E_1-Alg(dgCog^{conil}))$, where $E_1(A)$ is the free $E_1$-algebra on $A$ determined by the formula
\[
E_1(A)=\bigoplus_{n\geq 1}(E_1(n)\otimes A^{\otimes n})_{\Sigma_n}.
\]
Here $(-)_{\Sigma_n}$ denotes the coinvariants under the diagonal action of the symmetric group $\Sigma_n$.
Since the forgetful functor $U:dgCog^{conil}\rightarrow Ch_{\mathbb{K}}$ creates the tensor product and the small colimits,
we actually have
\[
U\circ\mathcal{B}_{E_1}=\mathcal{B}_{E_1}^{\mathbb{K}}.
\]
Weak equivalences in $E_1-Alg(dgCog^{conil})$ and $dgCog^{conil}$ are created in chain complexes, so the conservativity of $\mathcal{B}_{E_1}^{\mathbb{K}}$ on $E_1-Alg^{0-con}(Ch_{\mathbb{K}})$, where the upper script $0-con$ stands for $0$-connected algebras, implies the conservativity of $\mathcal{B}_{E_1}$ on $E_1-Alg^{0-con}(dgCog^{conil})$.

\noindent
\subsubsection{Commuting with totalizations.}
We need the conormalization of a cosimplicial coalgebra structure to be compatible with an $E_1$-algebra structure. The conormalization of cosimplicial coalgebras forms a lax monoidal functor (see \cite[Section 2]{SS}) for a good review of its properties in the simplicial case) that consequently lifts to cosimplicial $E_1$-algebras in dg coalgebras.
Let us note
\[
Tot_{E_1-Ass^{op}}:E_1-Alg(dgCog^{conil})^{\Delta}\rightarrow E_1-Alg(dgCog^{conil})
\]
and
\[
Tot_{Ass^{op}}:(dgCog^{conil})^{\Delta}\rightarrow dgCog^{conil}
\]
the totalization functors, which exist since we work with model categories). Let us also note $Res^{\bullet}(C)\in E_1-Alg(dgCog^{conil})^{\Delta}$, for $C\in E_1-Alg(dgCog^{conil})$, a certain functorial cosimplicial resolution of $C$ which will be defined in the next section and whose totalization gives a functorial fibrant resolution of $C$. Our aim is to prove the following properties of the bar construction with respect to these functors:
\begin{prop}\label{P: barvstot}

(1) There is an isomorphism
\[
Tot_{E_1-Ass^{op}}(Res^{\bullet}(C))\stackrel{\cong}{\rightarrow}Tot_{Ass^{op}}(Res^{\bullet}(C))
\]
in $E_1-Alg(dgCog^{conil})$.

(2) There is an isomorphism
\[
\overline{N^*}Res^{\bullet}(C)\stackrel{\sim}{\cong} Tot_{Ass^{op}}(Res^{\bullet}(C))
\]
in $E_1-Alg(dgCog^{conil})$, where $\overline{N^*}$ is the reduced conormalization.

(3) There is a weak equivalence
\[
\mathcal{B}_{E_1}(\overline{N^*}Res^{\bullet}(C))\stackrel{\sim}{\rightarrow} \overline{N^*}\mathcal{B}_{E_1}(Res^{\bullet}(C))
\]
in $dgCog^{conil}$.

\end{prop}
This gives the commutation of the bar construction $\mathcal{B}_{E_1}$ with the totalizations, up to weak equivalence, on the image of a functorial fibrant resolution on $E_1-Alg(dgCog^{conil})$.
The two next subsections are devoted to define explicit models for totalizations and cosimplicial resolutions that will allow us to prove this proposition.

\subsection{Triple coresolution and totalization for dg coalgebras}

The crux of our method here is to adapt to dg coalgebras a dual version of the operadic arguments of \cite[Appendix B]{Fre5}, \cite{Fre6} and \cite[Lemma B.8]{FW}.

The adjunction
\[
U:dgCog^{conil}\rightleftarrows Ch_{\mathbb{K}}:F^c
\]
between the forgetful functor $U$ and the cofree coalgebra functor $F^c$ gives rise to a comonad $\overline{F^c}=U\circ F^c$
over $Ch_{\mathbb{K}}$ equipped with a coproduct $\nu:\overline{F^c}\rightarrow\overline{F^c}\circ\overline{F^c}$, a counit $\epsilon:\overline{F^c}\rightarrow Id$ and a coaction $\rho:F^c\rightarrow F^c\circ\overline{F^c}$ of $F^c$ on $\overline{F^c}$.
For every $C\in dgCog^{conil}$, the triple coresolution $Res^{\bullet}(C)\in (dgCog^{conil})^{\Delta}$ is defined by
\begin{eqnarray*}
Res^n(C) & = & F^c\circ\underbrace{\overline{F^c}\circ\dots\circ\overline{F^c}}_{n}(U(C)) \\
 & = & F^c(K^n(C)) \\
\end{eqnarray*}
where
\[
K^n(C)= \underbrace{\overline{F^c}\circ\dots\circ\overline{F^c}}_{n}(U(C)).
\]
The cofaces
\[
d^i:Res^{n-1}(C)\rightarrow Res^n(C)
\]
are induced by $\rho$ for $i=0$, by $\nu$ on the $i^{th}$ factor for $1\leq i\leq n-1$ and by the coalgebra structure $U(C)\rightarrow F^c(U(C))$ of $C$ for $i=n$.
The codegeneracies
\[
s^j:Res^{n+1}(C)\rightarrow Res^n(C)
\]
are induced by $\epsilon$ on the $(j+i)^{th}$ factor.

Now recall that for every model category $\mathcal{C}$, the category of cosimplicial objects $\mathcal{C}^{\Delta}$ can be equipped with the Reedy model structure. The existence of simplicial frames in any model category (see \cite[Chapter 16]{Hir}) ensures that for any cosimplicial object $C^{\bullet}\in \mathcal{C}^{\Delta}$, we can pick a simplicial frame
$(C^{\bullet})^{\Delta^{\bullet}}$ which forms a simplicial object in $\mathcal{C}^{\Delta}$. This allows us to define the totalization of $C^{\bullet}$ by the end formula
\[
Tot(C^{\bullet})=\int_{\underline{n}\in\Delta}(C^n)^{\Delta^n}.
\]
When $C^{\bullet}$ is Reedy fibrant, any simplicial frame on $C^{\bullet}$ is a simplicial resolution of $C^{\bullet}$,
and the properties of such resolutions implies that the totalization of $C^{\bullet}$ is a fibrant object of $\mathcal{C}$ which is invariant, up to weak equivalence, under a change of simplicial frame.
\begin{prop}
For any dg coalgebra $C$, the cosimplicial dg coalgebra $Res^{\bullet}(C)$ is Reedy fibrant.
\end{prop}
\begin{proof}
The cofaces and codegeneracies of $Res^{\bullet}(C)$ restrict to $K^{\bullet}(C)$ except for $d^0$, so we set $d^0=0$
to get a full cosimplicial structure on $K^{\bullet}(C)$ which induces the cosimplicial structure of $Res^{\bullet}(C)$
by applying the cofree coalgebra functor. Then the remaining part of the proof is essentially a dual version of the proof of\cite[Proposition 2.2]{Fre6}. The tree-decomposition argument of \cite[Proposition 2.2]{Fre6} dualizes to the coalgebra setting by using trees with the converse orientation (one single input and several outputs), so that the $n$-simplices $K^n(C)$ decompose into a direct sum
\[
K^n(C)=M^nK^{\bullet}(C)\oplus N^nK^{\bullet}(C)
\]
where $M^nK^{\bullet}(C)$ is the $n^{th}$ matching object of $K^{\bullet}(C)$ and $N^nK^{\bullet}(C)$ is the degree $n$ part
of the conormalization of $C$ (for the graduation induced by the cosimplicial structure of $K^{\bullet}(C)$).
Moreover, we have $M^nRes^{\bullet}(C) = F^c(M^nK^{\bullet}(C))$ because $F^c$ commutes with small limits and the codegeneracies of $Res^{\bullet}(C)$ are induced by those of $K^{\bullet}(C)$ under $F^c$, hence
\[
Res^n(C)=M^nRes^{\bullet}(C)\wedge F^c(N^nK^{\bullet}(C))
\]
where $\wedge$ stands for the product in $dgCog^{conil}$.
The matching morphisms
\[
\mu_n:Res^n(C)\rightarrow M^nRes^{\bullet}(C)
\]
are the projections on the first factor defined by the pullbacks
\[
\xymatrix{
Res^n(C)\ar[r]\ar[d]^-{\mu_n} & F^c(N^nK^{\bullet}(C))\ar[d] \\
M^nRes^{\bullet}(C)\ar[r] & 0.
}
\]
By definition of the model structure on $dgCog^{conil}$, the cofree coalgebra functor $F^c$ is a right Quillen functor so $F^c(N^nK^{\bullet}(C))$ is fibrant (every chain complex over a field is fibrant in the projective model structure). Since fibrations are stable under pullbacks, the maps $\mu_n$ are fibrations for every integer $n$, which exactly means that $Res^{\bullet}(C)$ is Reedy fibrant.
\end{proof}

Our goal is now to prove that the totalization of the triple coresolution gives a (functorial) fibrant resolution in $dgCog^{conil}$. For this, we give an explicit model of this totalization by chosing an appropriate simplicial frame.
Recall that for any cosimplicial set $K^{\bullet}$, its conormalization $N^*(K^{\bullet})$ forms an augmented simplicial associative algebra, and dually its normalization $N_*(K_{\bullet})$ forms a coaugmented cosimplicial coassociative coalgebra
(they have actually a richer structure, respectively that of an algebra and a coalgebra over the Barratt-Eccles operad \cite{BF}). The coaugmentation ideal of $N_*(K_{\bullet})$, that is, its reduced normalization $\overline{N_*}(K_{\bullet})$, hence forms an object of $dgCog^{conil}$.
\begin{prop}\label{P: simpframe}
The collection $Res^{\bullet}(C)^{\Delta^{\bullet}}=\{F^c(K^n(C)\otimes \overline{N^*}(\Delta^m))\}_{n,m\in\mathbb{N}}$ forms a simplicial frame of $Res^{\bullet}(C)$ in $(dgCog^{conil})^{\Delta}$. 
\end{prop}
\begin{proof}
We first show that the functor $(-)\otimes \overline{N^*}(\Delta^{\bullet})$ defines a simplicial frame on the full subcategory $F^c(Ch_{\mathbb{K}})$ of $dgCog^{conil}$ formed by cofree coalgebras. Let $C\in dgCog^{conil}$ and $X\in Ch_{\mathbb{K}}$, we have a sequence of natural isomorphisms
\begin{eqnarray*}
Mor_{dgCog^{conil}}(C,F^c(X\otimes\overline{N^*}(\Delta^m))) & \stackrel{\cong}{\rightarrow} & Mor_{Ch_{\mathbb{K}}}(C,X\otimes\overline{N^*}(\Delta^m)) \\
 & \stackrel{\cong}{\rightarrow} & Mor_{Ch_{\mathbb{K}}}(C\otimes\overline{N_*}(\Delta^m),X)\\
 & \stackrel{\cong}{\rightarrow} & Mor_{dgCog^{conil}}(C\otimes\overline{N_*}(\Delta^m),F^c(X))\\
\end{eqnarray*}
where the first line follows from the universal property of the free-forgetful adjunction,
the second line follows from the dualization of the conormalization,
and the third line from the free-forgetful adjunction.
Let us note that the third line makes sense because the reduced normalization is a coalgebra and the tensor product of coalgebra is defined in chain complexes. We deduce that the functor
\[
(-)^{\Delta^m}:F^c(X)\mapsto F^c(X\otimes\overline{N^*}(\Delta^m))
\]
is right adjoint to the functor
\[
(-)\otimes \Delta^m:C\mapsto C\otimes\overline{N_*}(\Delta^m)
\]
on the full subcategory $F^c(Ch_{\mathbb{K}})$ of $dgCog^{conil}$. Consequently, the functor $(-)^{\Delta^{\bullet}}$ defines a simplicial frame on $F^c(Ch_{\mathbb{K}})$ if and only if $(-)\otimes \Delta^{\bullet}$ defines a cosimplicial frame.
Since the forgetful functor $dgCog^{conil}\rightarrow Ch_{\mathbb{K}}$ defines small colimits, cofibrations and weak equivalences, the later property follows from the fact that $(-)\otimes\overline{N_*}(\Delta^{\bullet})$ is a cosimplicial frame on $Ch_{\mathbb{K}}$.

Finally, the fact that $(-)^{\Delta^{\bullet}}$ is a simplicial frame on $F^c(Ch_{\mathbb{K}})$ implies that
$Res^{\bullet}(C)^{\Delta^{\bullet}}$ is a simplicial frame of $Res^{\bullet}(C)$.
Indeed, the morphism
\[
\epsilon^*:Res^{\bullet}(C)^{\Delta^{0}}\rightarrow Res^{\bullet}(C)^{\Delta^{n}}
\]
induced by the constant map $\epsilon:\{0,\cdot,n\}\rightarrow\{0\}$ defines a weak equivalence in each dimension because $Res^{\bullet}(C)$ is fibrant in each dimension (since it is Reedy fibrant). By definition of weak equivalences in the Reedy model structure of $(dgCog^{conil})^{\Delta}$, this means that $\epsilon^*$ is a weak equivalence of $(dgCog^{conil})^{\Delta}$ for every integer $n$. The simplicial frame construction $Res^{\bullet}(C)^{\Delta^{\bullet}}$ can be extended to an arbitrary simplicial set $K\in sSet$ to give a simplicial-cosimplicial object $Res^{\bullet}(C)^K$, and checking the fibration condition in the definition of a simplicial frame amounts to check that for every integer $n$, the map
\[
i^*:Res^{\bullet}(C)^{\Delta^n}\rightarrow Res^{\bullet}(C)^{\partial\Delta^n}
\]
induced by the inclusion $i^*:\partial\Delta^n\hookrightarrow \Delta^n$ forms a fibration in the base category for any $n>0$.
Here the base category is $(dgCog^{conil})^{\Delta}$, which means that $i^*$ has to be a Reedy fibration of cosimplicial objects.
By definition, we thus have to prove that the pullback-corner map
\[
(i^*,\mu^*):Res^r(C)^{\Delta^n}\rightarrow M_rRes^{\bullet}(C)^{\Delta^n}\times_{M_rRes^{\bullet}(C)^{\partial\Delta^n}}Res^r(C)^{\partial\Delta^n}
\]
induced by the matching map $\mu:Res^r(C)\rightarrow M_rRes^{\bullet}(C)$ and by $i$ is a fibration of dg coalgebras.
The cofree coalgebra functor $F^c$ commutes with limits, so the pullback-corner map above is actually given by
\begin{eqnarray*}
(i^*,\mu^*):F^c(K^r(C)\otimes\overline{N_*}(\Delta^n)) & \rightarrow & \\
F^c\left(M_rK^{\bullet}(C)\otimes\overline{N_*}(\Delta^n)\times_{M_rK^{\bullet}(C)\otimes\overline{N_*}(\partial\Delta^n)} K^r(C)\otimes\overline{N_*}(\partial\Delta^n)\right) & &
\end{eqnarray*}
which is exactly the image under $F^c$ of the pullback-corner map of complexes induced by $i$ and the matching map of $K^{\bullet}(C)$. Since $F^c$ preserves fibrations between fibrant objects as a right Quillen functor, and since every chain complex is fibrant, the proof boils down to check that the map
\begin{eqnarray*}
(i^*,\mu^*):K^r(C)\otimes\overline{N_*}(\Delta^n) & \rightarrow & \\
M_rK^{\bullet}(C)\otimes\overline{N_*}(\Delta^n)\times_{M_rK^{\bullet}(C)\otimes\overline{N_*}(\partial\Delta^n)} K^r(C)\otimes\overline{N_*}(\partial\Delta^n) & &
\end{eqnarray*}
is a fibration of chain complexes. Dualizing the normalization functor in complexes, this map is, up to isomorphism, given by
\begin{eqnarray*}
(i^*,\mu^*):Hom_{dg}(\overline{N^*}(\Delta^n),K^r(C)) \rightarrow & & \\
Hom_{dg}(\overline{N^*}(\Delta^n),M_rK^{\bullet}(C))\times_{Hom_{dg}(\overline{N^*}(\partial\Delta^n),M_rK^{\bullet}(C))} Hom_{dg}(\overline{N^*}(\partial\Delta^n),K^r(C)) & &
\end{eqnarray*}.
We conclude that this is a fibration by applying the dual pushout-product axiom, which holds true in $Ch_{\mathbb{K}}$.
\end{proof}
\begin{prop}
The coaugmentation $\eta:C\rightarrow Tot(Res^{\bullet})$ induced by the counit $\epsilon(C):C\rightarrow Res^0(C)=Res^0(C)^{\Delta^0}$ is a quasi-isomorphism of dg coalgebras.
\end{prop}
\begin{proof}
This proof follows an argument line similar to the one of \cite[Lemma B.8]{FW}.
The totalization of $Res^{\bullet}(C)$ is defined by the end
\[
Tot(Res^{\bullet}(C)) = \int_{\underline{n}\in\Delta}F^c(K^n(C)\otimes\overline{N^*}(\Delta^n)).
\]
We consider the following composition of maps
\begin{eqnarray*}
\int_{\underline{n}\in\Delta}F^c(K^n(C)\otimes\overline{N^*}(\Delta^n)) & \underbrace{\rightarrow}_{(1)} &
\int_{\underline{n}\in\Delta}K^n(C)\otimes\overline{N^*}(\Delta^n) \\
 & \underbrace{\stackrel{\cong}{\rightarrow}}_{(2)} & \overline{N^*}K^{\bullet}(C)\\
 & \underbrace{\rightarrow}_{(3)} & \overline{\Omega}(C)\\
\end{eqnarray*}
where $(1)$ is the canonical projection, $(2)$ is a formal isomorphism and $(3)$ is the projection on maximal simplices in $K^{\bullet}(C)$ (see \cite[Appendix C.2.16]{Fre5} for the operadic case) from the reduced conormalized complex of the cosimplicial object $K^{\bullet}(C)$ to the reduced cobar construction $\overline{\Omega}(C)$ of $C$.
The composite $(1-3)$ lifts to a morphism $f:Tot(Res^{\bullet}(C))\rightarrow Bar\overline{\Omega}(C)$, where $Bar$ is the usual bar construction on augmented dg algebras. Using the weight decomposition of the cofree coalgebra functor $F^c=\bigoplus_{s\geq 1}F^c_s$, we apply the argument of \cite[Lemma B.8]{FW} to get a morphism of spectral sequences associated to the filtration of $f$ induced by this decomposition, such that on the $E_0$ page, the morphism of graded objects induced by $f$ is given by
\[
E^0f = F^c(\int_{\underline{n}\in\Delta}K^n(C)\otimes\overline{N^*}(\Delta^n)\underbrace{\rightarrow}_{(2-3)} \overline{\Omega}(C)).
\]
The arrow $(2)$ is a formal isomorphism, and the arrow $(3)$ is a quasi-isomorphism (see \cite{Fre1} for a comparison between reduced conormalization and reduced cobar construction in the operadic case), so $E^0f$ is a quasi-isomorphism, hence $f$ is a quasi-isomorphism as well. We conclude that $\eta$ is a quasi-isomorphism via the commutative triangle
\[
\xymatrix{
Tot(Res^{\bullet}(C)) \ar[rr]^-{f} & & \mathcal{B}\overline{\Omega}(C) \\
 & C\ar[ul]^-{\eta} \ar[ur]_-{\sim} & \\
}.
\]
\end{proof}
We conclude:
\begin{thm}
The functor $Tot(Res^{\bullet}(-))$ defines a functorial fibrant resolution on $dgCog^{conil}$.
\end{thm}
We can thus restrict ourselves to prove the desired properties of $\mathcal{B}_{E_1}$ on cosimplicial coalgebras of the form $Res^{\bullet}(C)$.

\subsection{Totalization versus bar construction}

In this section we prove Proposition~\ref{P: barvstot}.
\begin{lem}
For any $C\in E_1-Alg(dgCog^{conil})$, there is an $E_1$-algebra structure on $Tot_{Ass^{op}}(Res^{\bullet}(C))$ compatible with its coalgebra structure so that we can chose $Tot_{Ass^{op}}(Res^{\bullet}(C))$ as a model for $Tot_{E_1-Ass^{op}}(Res^{\bullet}(C))$.
\end{lem}
\begin{proof}
Recall that since $E_1-Alg(dgCog^{conil})^{\Delta}$ is a Reedy model category, we can define a totalization functor
\[
Tot_{E_1-Ass^{op}}(C^{\bullet}) = \int_{\underline{n}\in\Delta}^{E_1-Alg(dgCog^{conil})}(C^n)^{\Delta^n}_{E_1-Ass^{op}}
\]
where $C^{\bullet}\in E_1-Alg(dgCog^{conil})^{\Delta}$, $(-)^{\Delta^{\bullet}}_{E_1-Ass^{op}}$ is a simplicial frame in $E_1-Alg(dgCog^{conil})^{\Delta}$, and the limit $\int_{\underline{n}\in\Delta}^{E_1-Alg(dgCog^{conil})}$ is taken in $E_1-Alg(dgCog^{conil})$.

The cofree coalgebra functor is defined by the reduced tensor coalgebra $F^c(C)=\bigoplus_{n\geq 1}C^{\otimes n}$ equipped with the deconcatenation product, and comes with a diagonal morphism
\[
F^c(C\otimes D)=\bigoplus_{n\geq 1}C^{\otimes n}\otimes D^{\otimes n} \rightarrow \bigoplus_{m,n\geq 1}C^{\otimes m}\otimes D^{\otimes n} = F^c(C)\otimes F^c(D)
\]
(recall that the tensor product of coalgebras is defined in chain complexes)
natural in $C$ and $D$, so that $F^c$ is an oplax monoidal functor. Moreover, the forgetful functor $U$ is monoidal by construction of the tensor product in $dgCog^{conil}$, so the Quillen adjunction between $U$ and $F^c$ lifts to a Quillen adjunction
\[
U:E_1-Alg(dgCog^{conil})\rightleftarrows E_1-Alg(Ch_{\mathbb{K}}):F^c
\]
(see \cite[Proposition 12.3.3]{Fre3}).
This implies that if $C$ is an object of $E_1-Alg(dgCog^{conil})$, then its $E_1$-algebra structure is compatible with the cosimplicial constructions $K^{\bullet}(C)$ and $Res^{\bullet}(C)$ so that these cosimplicial coalgebras are actually cosimplicial $E_1$-algebras in coalgebras.

Now, recall that the reduced conormalization $\overline{N^*}(\Delta^n)$ is an algebra over the Barratt-Eccles operad $BE$ (see \cite{BF}). For any $C\in E_1-Alg(dgCog^{conil})$, the chain tensor product $K^n(C)\otimes \overline{N^*}(\Delta^n)$ thus forms an $E_1\otimes_H BE$-algebra in $Ch_{\mathbb{K}}$, where $\otimes_H$ is the arity-wise tensor product of operads (also called the Hadamard tensor product \cite[Section 5.3.2]{LV}).
Since $E_1$, like any operad over a field of characteristic zero, is equipped with an operad morphism $E_1\rightarrow E_1\otimes_H BE$, the tensor product $K^n(C)\otimes \overline{N^*}(\Delta^n)$ is actually an $E_1$-algebra in $Ch_{\mathbb{K}}$. This implies, in turn, that $F^c(K^n(C)\otimes \overline{N^*}(\Delta^n))$ belongs to $E_1-Alg(dgCog^{conil})$.
Moreover, limits in $E_1-Alg(dgCog^{conil})$ are determined by the forgetful functor, so
\begin{eqnarray*}
\int_{\underline{n}\in\Delta}^{E_1-Alg(dgCog^{conil})}F^c(K^n(C)\otimes \overline{N^*}(\Delta^n)) & = &
\int_{\underline{n}\in\Delta}^{dgCog^{conil}}F^c(K^n(C)\otimes \overline{N^*}(\Delta^n)) \\
 & = & Tot_{Ass^{op}}(Res^{\bullet}(C))
\end{eqnarray*}
where the last line follows from Proposition~\ref{P: simpframe}, and the totalization of $Res^{\bullet}(C)$ in coalgebras inherits a compatible $E_1$-algebra structure. We conclude by noticing that the argument line of the proof of Proposition~\ref{P: simpframe} can be literally transposed here: the simplicial-cosimplicial object $F^c(K^{\bullet}(C)\otimes \overline{N^*}(\Delta^{\bullet}))$ defines a simplicial frame of $Res^{\bullet}(C)$ in $E_1-Alg(dgCog^{conil})^{\Delta}$ because for any $X\in E_1-Alg(Ch_{\mathbb{K}})$, the simplicial object $F^c(X\otimes \overline{N^*}(\Delta^{\bullet}))$ defines a simplicial frame of $F^c(X)$ in $E_1-Alg(dgCog^{conil})$.
\end{proof}

\begin{lem}
For any $C\in E_1-Alg(dgCog^{conil})$, there is an isomorphism
\[
Tot_{Ass^{op}}(Res^{\bullet}(C))\stackrel{\cong}{\rightarrow} N^*Res^{\bullet}(C)
\]
in $E_1-Alg(dgCog^{conil})$.
\end{lem}
\begin{proof}
There is an isomorphism
\[
\int_{\underline{n}\in\Delta}^{E_1-Alg(Ch_{\mathbb{K}})}K^n(C)\otimes \overline{N^*}(\Delta^n)
=
\int_{\underline{n}\in\Delta}^{Ch_{\mathbb{K}}}K^n(C)\otimes \overline{N^*}(\Delta^n)
\stackrel{\cong}{\rightarrow} \overline{N^*}K^{\bullet}(C)
\]
in $E_1-Alg(Ch_{\mathbb{K}})$, inducing an isomorphism
\[
F^c(\int_{\underline{n}\in\Delta}^{Ch_{\mathbb{K}}}K^n(C)\otimes \overline{N^*}(\Delta^n))
\stackrel{\cong}{\rightarrow} F^c(\overline{N^*}K^{\bullet}(C))
\]
in $E_1-Alg(dgCog^{conil})$.
We have equalities
\[
F^c(\int_{\underline{n}\in\Delta}^{Ch_{\mathbb{K}}}K^n(C)\otimes \overline{N^*}(\Delta^n))
=
\int_{\underline{n}\in\Delta}^{dgCog^{conil}}F^c(K^n(C)\otimes \overline{N^*}(\Delta^n)))
\]
because $F^c$ commutes with limits and
\[
F^c(\overline{N^*}K^{\bullet}(C)) = \overline{N^*}Res^{\bullet}(C)
\]
because $Res^{\bullet}(C)=F^c(K^{\bullet}(C))$ and the cosimplicial structure of $Res^{\bullet}(C)$ is induced by the one of $K^{\bullet}(C)$ under $F^c$, hence an isomorphism
\[
\int_{\underline{n}\in\Delta}^{dgCog^{conil}}F^c(K^n(C)\otimes \overline{N^*}(\Delta^n)))
\stackrel{\cong}{\rightarrow} \overline{N^*}Res^{\bullet}(C).
\]
The source of this isomorphism is nothing but $Tot_{Ass^{op}}(Res^{\bullet}(C))$.
\end{proof}

\begin{lem}
For any $C\in E_1-Alg(dgCog^{conil})$, there is a weak equivalence
\[
\mathcal{B}_{E_1}(\overline{N^*}Res^{\bullet}(C))\stackrel{\sim}{\rightarrow} \overline{N^*}\mathcal{B}_{E_1}(Res^{\bullet}(C))
\]
in $dgCog^{conil}$.
\end{lem}
\begin{proof}
We start by proving that $F^c$ commutes with reflexive coequalizers of $E_1$-algebras.
Let $I$ be the small category whose associated diagram is given by
\[
\xymatrix{0 \ar@<1ex>[r]\ar@<-1ex>[r] & 1 \ar@/_{1pc}/[l]}
\]
and $\{X_i\}_{i\in ob(I)}$ be an $I$-diagram
\[
\xymatrix{X_0\ar@<1ex>[r]^-{d_0}\ar@<-1ex>[r]_-{d_1} & X_1 \ar@/_{1.5pc}/[l]_{s_0}}
\]
in $E_1-Alg(Ch_{\mathbb{K}})$ such that $s_0\circ d_0=s_0\circ d_1=id_{X_0}$.
The colimit $colim_i X_i$ of such a diagram is called a reflexive coequalizer.
We have the following series of isomorphisms in $E_1-Alg(dgCog^{conil})$:
\begin{eqnarray*}
F^c(colim_i X_i) & = & \bigoplus_{n\geq 1}(colim_i X_i)^{\otimes n} \\
 & \cong & \bigoplus_{n\geq 1} colim_i (X_i)^{\otimes n} \\
 & \cong &  colim_i\bigoplus_{n\geq 1} X_i^{\otimes n}\\
 & = & colim_i F^c(X_i).
\end{eqnarray*}
The first line is by definition of $F^c$ as the reduced tensor coalgebra. Since $E_1$ is a Hopf operad, $E_1-Alg(Ch_{\mathbb{K}})$ forms a symmetric monoidal category for the chain tensor product. Moreover, according to \cite[Proposition 3.3.1]{Fre3}, reflexive coequalizers of $E_1$-algebras are also determined in chain complexes, and according to \cite[Proposition 1.2.3]{Fre3} the chain tensor power functors preserves reflexive coequalizers, hence the second line. The third line is just the commutation of two colimits of $E_1$-algebras, and the last line follows from the fact that colimits of coalgebras are determined by the forgetful functor.

Now, recall from 4.2.1 that $\mathcal{B}_{E_1}$ is constructed as a reflexive coequalizer of $E_1$-algebras, so that
\[
F^c\circ\mathcal{B}_{E_1}^{\mathbb{K}} = \mathcal{B}_{E_1}\circ F^c.
\]
We have
\begin{eqnarray*}
\mathcal{B}_{E_1}(\overline{N^*}Res^{\bullet}(C)) & = & \mathcal{B}_{E_1}(F^c(\overline{N^*}K^{\bullet}(C))) \\
 & = & F^c(\mathcal{B}_{E_1}^{\mathbb{K}}(\overline{N^*}K^{\bullet}(C)))
\end{eqnarray*}
by construction of the cosimplicial structures of $K^{\bullet}(C)$ and $Res^{\bullet}(C)$ and the commutation rule above, and
\[
\overline{N^*}\mathcal{B}_{E_1}(Res^{\bullet}(C)) = F^c(\overline{N^*}\mathcal{B}_{E_1}^{\mathbb{K}}(K^{\bullet}(C)))
\]
by the commutation of $F^c$ with $\overline{N^*}$ and $\mathcal{B}_{E_1}$ as explained above.
Consequently, to build a weak equivalence
\[
\mathcal{B}_{E_1}(N^*Res^{\bullet}(C))\stackrel{\sim}{\rightarrow} N^*\mathcal{B}_{E_1}(Res^{\bullet}(C))
\]
in $dgCog^{conil}$ amounts to build a weak equivalence of cofree coalgebras
\[
F^c(\mathcal{B}_{E_1}^{\mathbb{K}}(\overline{N^*}K^{\bullet}(C))) \stackrel{\sim}{\rightarrow} F^c(\overline{N^*}\mathcal{B}_{E_1}^{\mathbb{K}}(K^{\bullet}(C))),
\]
thus a quasi-isomorphism of chain complexes
\[
\mathcal{B}_{E_1}^{\mathbb{K}}(\overline{N^*}K^{\bullet}(C))\stackrel{\sim}{\rightarrow} \overline{N^*}\mathcal{B}_{E_1}^{\mathbb{K}}(K^{\bullet}(C)).
\]

On the one hand, as explained before, the functors $U$ and $F^c$ commute with reflexive coequalizers of $E_1$-algebras, in particular with the bar construction $\mathcal{B}_{E_1}^{\mathbb{K}}$, hence the equality
\[
\overline{N^*}\mathcal{B}_{E_1}^{\mathbb{K}}(K^{\bullet}(C)) = \overline{N^*}K^{\bullet}(\mathcal{B}_{E_1}^{\mathbb{K}}(C)).
\]
By the same argument, the coalgebra structure map $C\rightarrow F^c(C)$ of $C$ induces a map $\mathcal{B}_{E_1}^{\mathbb{K}}(C)\rightarrow \mathcal{B}_{E_1}^{\mathbb{K}}(F^c(C))=F^c(\mathcal{B}_{E_1}^{\mathbb{K}}(C))$, so $\mathcal{B}_{E_1}^{\mathbb{K}}(C)$ is a coalgebra and we get the canonical quasi-isomorphism
\[
\overline{N^*}K^{\bullet}(\mathcal{B}_{E_1}^{\mathbb{K}}(C))
\stackrel{\sim}{\rightarrow}\overline{\Omega}\mathcal{B}_{E_1}^{\mathbb{K}}(C)
\]
between the reduced conormalization and the reduced cobar construction.
On the other hand, the $E_1$-algebra structure of $C$ is compatible with the quasi-isomorphism
\[
\overline{N^*}K^{\bullet}(C)\stackrel{\sim}{\rightarrow}\overline{\Omega}(C)
\]
so that this is actually a quasi-isomorphism of $E_1$-algebras. The compatibility of the $E_1$-algebra structures with this map follows, once again, from the lax monoidality of the reduced conormalization, and from the lax monoidality of the cobar construction functor of coalgebras via the Milgram map \cite{Mil}. In turn, this quasi-isomorphism of $E_1$-algebras induces a quasi-isomorphism
\[
\mathcal{B}_{E_1}^{\mathbb{K}}(\overline{N^*}K^{\bullet}(C))\stackrel{\sim}{\rightarrow}\mathcal{B}_{E_1}^{\mathbb{K}}(\overline{\Omega}(C)).
\]
Finally, the equality
\[
\mathcal{B}_{E_1}^{\mathbb{K}}(\overline{\Omega}(C)) = \overline{\Omega}(\mathcal{B}_{E_1}^{\mathbb{K}}(C))
\]
holds true for any $E_1$-algebra in coalgebras $C$. Indeed, on the one hand, the cobar functor $\overline{\Omega}$ is a left adjoint, so it commutes with colimits of chain complexes. On the other hand, the bar construction $\mathcal{B}_{E_1}^{\mathbb{K}}$ is a reflexive coequalizer of $E_1$-algebras, and such reflexive coequalizers are created in chain complexes.

This provides consequently the desired quasi-isomorphism between $\mathcal{B}_{E_1}^{\mathbb{K}}(\overline{N^*}K^{\bullet}(C))$ and $\overline{N^*}\mathcal{B}_{E_1}^{\mathbb{K}}(K^{\bullet}(C))$.
\end{proof}

\section{Bialgebras versus $E_2$-algebras}

\subsection{(Co)units and (co)augmentations}

We refer the reader to \cite{AJ} for a detailed exposition about associative algebras and coassociative coalgebras in the pointed setting. A pointed complex is a complex $X$ equipped with two maps $e:\mathbb{K}\rightarrow X$ and $\epsilon:X\rightarrow\mathbb{K}$ such that $\epsilon\circ e = Id_{\mathbb{K}}$. Pointed complexes form a category $Ch_{\mathbb{K}}^{pt}$, with morphisms defined by chain morphisms commuting with these maps. The two functors
\[
(-)_{-}:X\in Ch_{\mathbb{K}}^{pt}\mapsto X_{-}=ker(\epsilon)\in Ch_{\mathbb{K}}
\]
and
\[
(-)_{+}:X\in Ch_{\mathbb{K}}\mapsto X_{+}=X\oplus\mathbb{K}\in Ch_{\mathbb{K}}^{pt}
\]
form an equivalence of categories
\[
(-)_+:Ch_{\mathbb{K}}\rightleftarrows Ch_{\mathbb{K}}^{pt}:(-)_-.
\]
This equivalence induces on $Ch_{\mathbb{K}}^{pt}$ the structure of a cofibrantly generated symmetric monoidal model category tensored over $Ch_{\mathbb{K}}$, with the following features:
\begin{itemize}
\item the coproduct (and also product) of two pointed complexes $X$ and $Y$ is given by
\[
X\vee Y = (X_-\oplus Y_-)_+;
\]
\item the tensor product of two pointed complexes $X$ and $Y$ is given by
\[
X\wedge Y = (X_-\otimes Y_-)_+;
\]
\item the internal hom is given by
\[
Hom_{pt}(X,Y)=Hom(X_-,Y_-)_+;
\]
\item the external tensor product of a complex $X$ with a pointed complex $Y$ is given by
\[
X\otimes_e Y = X_+\wedge Y (=(X\otimes Y_-)_+ );
\]
\item the external hom is given by
\[
Hom_{dg}(X,Y)=Hom(X_-,Y_-),
\]
because of the sequence of equalities
\begin{eqnarray*}
Mor_{pt}(X\otimes_eY,Z)=Mor_{pt}((X\otimes Y_-)_+,Z) & \cong  & Mor(X\otimes Y_-,Z_-)\\
 & \cong  & Mor(X,Hom(Y_-,Z_-));
\end{eqnarray*}
\item weak equivalences are quasi-isomorphisms and fibrations are degreewise surjections.
\end{itemize}

Now let $P$ be an augmented dg operad (i.e with an operad morphism $P\rightarrow I$) presented by $P=\mathcal{F}(\mathbb{K}\mu\oplus E)/(R)$, where $\mu$ is a generator of arity $2$ (a product). We suppose moreover that $P$ is equipped with an operad morphism $P\rightarrow Com$, so that $\mathbb{K}$ with its usual commutative associative product is a $P$-algebra.
This includes associative, commutative and Lie algebras (here $E$ is the zero $\Sigma$-object), as well as $n$-Poisson algebras (here $E$ is defined by a generator of degree $1-n$ in arity $2$ giving the shifted Poisson bracket), and $E_n$-algebras via the sequence of operad morphisms
\[
E_1\hookrightarrow ...\hookrightarrow E_n\hookrightarrow ...\hookrightarrow E_{\infty}\stackrel{\sim}{\rightarrow}Com.
\]
\begin{defn}
(1) A $P$-algebra $A$ is unitary if it is equipped with a chain morphism $e:\mathbb{K}\rightarrow A$ satisfying the unitarity relations
\[
\mu(e(1_{\mathbb{K}}),-)=\mu(-,e(1_{\mathbb{K}}))=id_A.
\]
Unitary $P$-algebras are $uP$-algebras, where $uP$ is obtained from $P$ by adding a generator of arity $1$ for the unit and adding the unitarity relations to $R$.

(2) A $P$-algebra $A$ is augmented if it is equipped with a $P$-algebra morphism $\epsilon:A\rightarrow\mathbb{K}$,
where $\mathbb{K}$ is equipped with the $P$-algebra structure induced by the morphism $P\rightarrow Com$ and the standard commutative algebra structure of $\mathbb{K}$.
Such an algebra $A$ is connected if its augmentation ideal is $0$-connected, equivalently if the degree zero part of $A$ is isomorphic to $\mathbb{K}$.

(3) A $P$-algebra $A$ is pointed if it is equipped with a unit $e$ and an augmentation $\epsilon$ satisfying
\[
\epsilon\circ e = id_{\mathbb{K}}.
\]

(4) A strictly unitary $P_{\infty}$-algebra, or $suP_{\infty}$-algebra, is a $P_{\infty}$-algebra (where $P_{\infty}$ is a cofibrant resolution of $P$) equipped with a strict unit with respect to the arity $2$ generator of $P_{\infty}$ inducing the product $\mu$ in homology. Similarly, we get the notions of augmented $P_{\infty}$-algebra and pointed $P_{\infty}$-algebra.

(5) Dually, one can define counitary, coaugmented and pointed $P$-algebras and $P_{\infty}$-algebras.
\end{defn}
\begin{rem}
An augmented $uP$-algebra is an $uP$-algebra $A$ with a morphism of $uP$-algebras $\epsilon:A\rightarrow\mathbb{K}$.
The compatibility of this morphism with the unit $e:\mathbb{K}\rightarrow A$ of $A$ and the unit $id_{\mathbb{K}}$ of the $uP$-algebra $\mathbb{K}$ is exactly the relation $\epsilon\circ e = id_{\mathbb{K}}$, so augmented $uP$-algebras are exactly pointed $P$-algebras. Dually, coaugmented $uP$-coalgebras are exactly pointed $P$-coalgebras.
The same identification holds for $suP_{\infty}$-algebras and $suP_{\infty}$-coalgebras.
\end{rem}

\noindent
\textit{Notations.}
To conclude this section of definitions, let us fix some notations for the next part. When writing our various categories of algebras and coalgebras, we use the superscripts $aug$ for augmented algebras, $coaug$ for coaugmented coalgebras, $con$ for connected (co)algebras, $0-con$ for $0$-connected (co)algebras, and $pt$ for pointed (co)algebras.

\subsection{Proofs of Theorem 0.1(2) and Corollary 0.2}\label{S:proofCorollary0.2}

Part (2) of Theorem 0.1 follows from the following lemma:
\begin{lem}\label{L: equivpoint}
(1) There is an equivalence of $\infty$-categories
\[
(-)_+:E_1-Cog^{conil}(dgCog^{conil})\rightleftarrows E_1-Cog^{conil,pt}(dgCog^{conil}):(-)_-.
\]

(2) There is an equivalence of $\infty$-categories
\[
(-)_+:E_1-Alg^{0-con}(dgCog^{conil})\rightleftarrows E_1-Alg^{aug,con}(dgCog^{conil}):(-)_-.
\]
\end{lem}
\begin{proof}
Recall that the functors $(-)_+$ and $(-)_-$ defines equivalences of categories
\[
(-)_+:E_1-Alg\rightleftarrows E_1-Alg^{pt}:(-)_-
\]
between non-unitary $A_{\infty}$-algebras and augmented (strictly) unitary $A_{\infty}$-algebras (see for instance \cite{Kel})
and dually
\[
(-)_+:E_1-Cog^{conil}\rightleftarrows E_1-Alg^{conil,pt}:(-)_-
\]
between conilpotent non-counitary $A_{\infty}$-coalgebras and conilpotent coaugmented (strictly) counitary $A_{\infty}$-algebras.
The functor $(-)_+$ is defined by a direct sum and the functor $(-)_-$ is defined by a cokernel (cokernel of the coaugmentation). Since colimits in conilpotent dg coalgebras are created in chain complexes, these functors lift to the equivalence of categories of part (2). The same argument applies for part (1), combined with the fact that this equivalence restricts to an equivalence between $0$-connected algebras and connected augmented algebras. In both cases, the base field $\mathbb{K}$ gives a well defined pointing because it is an associative and coassociative bialgebra.
Moreover, in both cases the functors preserve quasi-isomorphisms by definition, so we conclude by applying Lemma~\ref{L: stricteq}.
\end{proof}
According to Lemma~\ref{L: equivpoint} and to part (1) of Theorem 0.1, the composite adjunction
\[
(\mathcal{B}^{enh,pt}_{E_1}(-)_-)_+: E_1-Alg^{aug,con}(dgCog^{conil})\rightleftarrows E_1-Cog^{conil,pt}(dgCog^{conil}):-(\Omega^{enh,pt}_{E_1}(-)_-)_+
\]
defines an equivalence of $\infty$-categories.

We now build explicitly the fully faithful $\infty$-functor of Corollary 0.2:
\begin{thm}\label{T: tildomega}
There is a fully faithful $\infty$-functor
\[
\tilde{\Omega}:E_1-Alg^{aug,con}(dgCog^{conil})\hookrightarrow E_2-Alg^{aug}
\]
from pointed homotopy associative conilpotent dg bialgebras to augmented dg $E_2$-algebras.
\end{thm}
\begin{proof}
The equivalence with the $\infty$-category of nilpotent $E_2$-algebras goes through the following sequence of $\infty$-functors:
\begin{eqnarray*}
E_1-Alg^{aug,con}(dgCog^{conil}) & \substack{\stackrel{\sim}{\rightarrow} \\ (\mathcal{B}_{E_1}^{enh}(-)_-)_+} & E_1-Cog^{conil,pt}(dgCog^{conil}) \\
 & = & suE_1\otimes uAss-Cog^{conil,coaug} \\
 & \hookrightarrow & uE_1\otimes uAss-Cog^{conil,coaug} \\
 & \substack{\stackrel{\sim}{\rightarrow} \\ \phi^*}  & uE_2-Cog^{conil,coaug} = E_2-Cog^{conil,pt} \\
 & \substack{\stackrel{\sim}{\rightarrow} \\ (-)_-}  & E_2-Cog^{conil} \\
 & \substack{\stackrel{\sim}{\rightarrow} \\ Cobar^{(2)}} & E_2-Alg^{aug}.
\end{eqnarray*}
The equivalence of the first line is the functor $(\mathcal{B}_{E_1}^{enh}(-)_-)_+$ given by Theorem 0.1(2).
The second line holds by definition of the Boardman-Vogt tensor product $\otimes$ of operads. Indeed, recall that we have
\[
E_1-Cog^{conil,pt}(dgCog^{conil}) = suE_1-Cog^{conil,coaug}(dgCog^{conil}).
\]
The operad $suE_1\otimes uAss$ encodes $suE_1$-coalgebras in $uAss$-coalgebras, and the pointing of any coalgebra $A\in E_1-Cog^{conil,pt}(dgCog^{conil})$ fixes a pointing on $A$ as an object of $dgCog^{conil}$ as well, hence the equality
\[
suE_1-Cog^{conil}(uAss-Cog^{conil})= suE_1\otimes uAss-Cog^{conil}
\]
restricts to
\begin{eqnarray*}
E_1-Cog^{conil,pt}(dgCog^{conil}) & = & suE_1-Cog^{conil,coaug}(dgCog^{conil}) \\
 & = & suE_1\otimes uAss-Cog^{conil,coaug}.
\end{eqnarray*}

The third line is the fully faithful functor induced by the operad morphism
\[
uE_1\otimes uAss\rightarrow suE_1\otimes uAss,
\]
where $suE_1\otimes uAss$ parametrizes coalgebras whose $E_1$-structure is strictly counitary, and $uE_1\otimes uAss$ parametrizes coalgebras whose $E_1$-structure is counitary up to homotopy in order to use a cofibrant model of $uE_1$ (noted $E_1$ in \cite{FV}). At the level of the corresponding categories of coalgebras, this embeds coalgebras with a strict counit into coalgebras with a homotopy counit. This is a surjection of operads, which induces consequently a fully faithful $\infty$-functor by Corollary~\ref{C: fullfaithprop}.

The functor $\phi^*$ is defined as follows. First, the precomposition of any $uE_1\otimes uAss$-algebra structure map with the operad morphism
\[
\phi:uE_2\stackrel{\sim}{\rightarrow} uE_1\otimes uE_1 \stackrel{\sim}{\rightarrow} uE_1\otimes uAss
\]
gives a left Quillen functor
\[
\phi^*:uE_1\otimes uAss-Cog^{conil}\rightarrow uE_2-Cog^{conil},
\]
where the model structures are given by Theorem 1.16 (weak equivalences and cofibrations are quasi-isomorphisms and degreewise injections).
According to the results of \cite{Lur2} and \cite{FV} this is a composite of weak equivalences of $\Sigma$-cofibrant operads,
so this functor is actually an equivalence of $\infty$-categories. According to \cite[Proposition 3.7]{FV}, we have $uE_1\otimes uAss-Cog^{conil}=u(E_1\otimes Ass)-Cog^{conil}$ (the units of two coalgebra structures related by an interchange law coincide). So this is actually an equivalence
\[
\phi^*:u(E_1\otimes Ass)-Cog^{conil}\stackrel{\sim}{\rightarrow} uE_2-Cog^{conil}.
\]
This equivalence induces, in turn, an equivalence of the slice $\infty$-categories
\[
\phi^*:\mathbb{K}\setminus u(E_1\otimes Ass)-Cog^{conil}\stackrel{\sim}{\rightarrow} \phi^*(\mathbb{K})\setminus uE_2-Cog^{conil}.
\]
By definition, these slice categories are respectively $\mathbb{K}\setminus u(E_1\otimes Ass)-Cog^{conil}=uE_1\otimes uAss-Cog^{conil,coaug}$ and $ \phi^*(\mathbb{K})\setminus uE_2-Cog^{conil} = uE_2-Cog^{conil,coaug}$.

The functor $(-)_-$ defines an equivalence of categories between pointed $E_2$-coalgebras and $E_2$-coalgebras which preserves quasi-isomorphisms, hence an equivalence of $\infty$-categories by Lemma~\ref{L: stricteq}.

Finally, the functor $Cobar^{(2)}$ is the twice iterated cobar construction on conilpotent $E_2$-coalgebras, which gives the Koszul duality functor associated to $E_2$ and forms a Quillen equivalence when restricted to conilpotent $E_2$-coalgebras and augmented $E_2$-algebras, hence an equivalence of $\infty$-categories.
\end{proof}
The functor $\tilde{\Omega}$ is thus defined, for every pointed homotopy associative conilpotent dg bialgebra $B$, by
\begin{eqnarray*}
\tilde{\Omega}(B) & = & Cobar^{(2)}(\phi^*(\mathcal{B}^{enh}_{E_1}(B_-)_+)_-) \\
 & = & Cobar^{(2)}\phi^*\mathcal{B}^{enh}_{E_1}(B_-).
\end{eqnarray*}
\begin{example}\label{Ex:Sym}
Our main working example in the last sections of the paper is the symmetric bialgebra $Sym(V)$ over a chain complex $V$.
It is actually a bicomplex, with a homological grading induced by the one of $V$ and a weight grading defined by the symmetric powers. In the remaining part we will consider the total complex associated to this bicomplex, thus we will consider $Sym(V)$ as a dg bialgebra with degree given by the total degree (sum of the homological degree and the weight), and with differential induced by the differential of $V$. This differential does not change the weight but decreases the homological degree by $1$, so it decreases the total degree by $1$. There is a decomposition
\[
Sym(V)=\mathbb{K}\oplus Sym^{\geq 1}(V)
\]
where $Sym^{\geq 1}(V)$ is the part of weight greater or equal to $1$, hence the part of total degree greater or equal to $1$. The base field $\mathbb{K}=Sym^0(V)$ is the part of total degree $0$ in $Sym(V)$, making $Sym(V)$ a pointed conilpotent bialgebra. Hence it makes sense to study $\tilde{\Omega}(Sym(V)) = Cobar^{(2)}\phi^*\mathcal{B}^{enh}_{E_1}(Sym^{\geq 1}(V))$.
\end{example}

\subsection{Proof of Theorem 0.3}

Let us recall the following consequence of Section 2:
\begin{thm}\label{T:EquivalentModuligivesEquivalentDefComplexex}
Let $P_{\infty}$ and $Q_{\infty}$ be two properads and $X$ and $Y$ be two complexes. Let $\varphi:P_{\infty}\rightarrow End_X$ be a $P_{\infty}$-algebra structure on $X$ and $\psi:P_{\infty}\rightarrow End_Y$ be a $Q_{\infty}$-algebra structure on $Y$.
Let us suppose that there is a homotopy equivalence of formal moduli problems.
\[
\underline{P_{\infty}\{X\}}^{\varphi}\simeq \underline{Q_{\infty}\{Y\}}^{\psi}.
\]
Then there exists a zigzag of quasi-isomorphisms
\[
Der_{\varphi}(P_{\infty},End_X)\underbrace{\stackrel{\sim}{\leftarrow}\bullet\stackrel{\sim}{\rightarrow}}_{L_{\infty}} Der_{\psi}(Q_{\infty},End_Y)
\]
of $L_{\infty}$-algebras.
\end{thm}
\begin{proof}
By Lurie's equivalence generalized to $L_{\infty}$-algebras (see Section 2.1), a homotopy equivalence of formal moduli problems induces a zigzag of quasi-isomorphisms of their tangent $L_{\infty}$-algebras. In the case of formal moduli problems of algebraic structures, the tangent $L_{\infty}$-algebra can be identified with a derivation complex (see Section 2.2).
\end{proof}

To get Theorem 0.3(1), one first applies Theorem~\ref{T:fiberseqDefwith+} to the functor $\tilde{\Omega}$.
In the sequence of $\infty$-functors used to define $\tilde{\Omega}$, there are equivalences plus a functor induced by precomposition with a surjection of operads. This functor is a fully faithful and conservative $\infty$-functor, so $\tilde{\Omega}$ gives a fully faithful and conservative $\infty$-functor. Moreover, the functor induced by precomposition with an operad morphism extends functorially by definition to algebras in $A$-modules, and still gives a fully faithful and conservative $\infty$-functor (it does not affect the underlying $A$-module structure of an algebra, and weak equivalences of algebras in $A$-modules are still the quasi-isomorphisms).
The construction of the equivalences in Theorem~\ref{T: tildomega} extends readily to algebras in $A$-modules as well. Indeed, the properties of chain complexes needed to build $\mathcal{B}_{E_1}^{enh}$ and $Cobar^{(2)}$ are the same in $A$-modules, which form a cofibrantly generated symmetric monoidal dg category with weak equivalences defined by quasi-isomorphisms, and these two functors are by construction compatible with linear extensions of the form $-\otimes_A B$.
Second, we have a fully faithful and conservative $\infty$-functor
\[
E_1-Alg^{aug,con}(dgCog^{conil})\hookrightarrow Bialg_{\infty}-Alg
\]
induced by the projection $Bialg_{\infty}\twoheadrightarrow (E_1,Ass^{\vee})$, where $(E_1,Ass^{\vee})$ is the properad encoding $E_1$-algebras in dg coalgebras. So the associated tangent formal moduli problems are equivalent, which concludes the proof, since the formal moduli problem of deformations of $B\in E_1-Alg^{aug,con}(dgCog^{conil})$ as a homotopy bialgebra is exactly $\underline{Bialg_{\infty}}\{B\}^{\varphi}$.

The proof of part (2) of Theorem 0.3 is completely similar for the following reasons. 
The functor $(-)^+$ takes quasi-free properads satisfying the conditions of \cite[Corollary 40]{MV2} to quasi-free properads satisfying the same conditions. By \cite[Corollary 40]{MV2}, such properads are cofibrant and by \cite[Theorem 42]{MV2},
every properad admits a cofibrant resolution of this form. Since our formal moduli problems 
are homotopy invariant under the choice of a resolution, we can always choose a resolution of this form, 
and \cite[Proposition 43]{MV2} ensures that the corresponding deformation complexes are also invariant, up to quasi-isomorphism,
under the choice of a resolution. So that all we have to do is to replace props $P_{\infty}$ by $P_{\infty}^+$ in the constructions above to transpose our results from $P_{\infty}$-algebras to $P_{\infty}^+$-algebras. We consequently get,
for any $Bialg^+_{\infty}$-algebra $\varphi^+:Bialg_{\infty}^+\rightarrow End_B$ 
and the corresponding $E_2^+$-algebra structure $\psi^+:E_2^+\rightarrow End_{\tilde{\Omega}B}$,
a zigzag of quasi-isomorphisms of $L_{\infty}$-algebras
\[
g_{Bialg_{\infty}^+,B}^{\varphi^+}\stackrel{\sim}{\leftarrow}\bullet\stackrel{\sim}{\rightarrow}g_{E_2^+,\tilde{\Omega}B}^{\psi^+}.
\]
The left-hand complex is identified with the Gerstenhaber-Schack complex $C^*_{GS}(B,B)$ (see \cite{Mer2}). In other words we
have a zigzag of quasi-isomorphisms
\begin{equation}\label{eq:gBialg=gE_2+} 
C^*_{GS}(B,B)\stackrel{\sim}{\leftarrow}\bullet\stackrel{\sim}{\rightarrow}g_{E_2^+,\tilde{\Omega}B}^{\psi^+},
\end{equation}
so it remains to identify the right-hand complex with the $E_2$-Hochschild complex $CH^{(\bullet > 0)}_{E_2}(\tilde{\Omega}B,\tilde{\Omega}B)$. This is precisely given by Lemma~\ref{L:gE2+=TA}.

\begin{rem}
As explained in \cite{Mer2}, the reason why we need the ``plus'' construction to recover the Gerstenhaber-Schack complex is that, 
one the one hand
\[
C^*_{GS}(B,B)\cong \prod_{m,n\geq 1}Hom_{dg}(B^{\otimes m},B^{\otimes n})[2-m-n],
\]
but on the other hand
\[
g_{Bialg_{\infty},B}^{\varphi}\cong \prod_{m,n\geq 1,m+n\geq 3}Hom_{dg}(B^{\otimes m},B^{\otimes n})[2-m-n]
\]
is an $L_{\infty}$-algebra with differential also given by the Gerstenhaber-Schack differential, but without the term $Hom_{dg}(B,B)$.
\end{rem}

\section{Identification of Deformation complexes with higher Hochschild, Gerstenhaber-Schack and Tamarkin complexes of $Pois_n$-algebras}

In this section we identify the deformation complex $g_{E_2^+,\tilde{\Omega}B}^{\psi^+}$ and prove Corollary~\ref{C:GS=HH}. More generally we identify the underlying $L_\infty$-structures of several complexes related to the deformation theory of $Pois_n$, $E_n$ and bialgebra structures.

\subsection{Tamarkin deformation complexes of $Pois_n$-algebras}
We have already mentioned the higher Hochschild complexes controlling the deformation theory of $E_2$-algebras in Section~\ref{S:EnHoch} and their identification with deformation complexes and 
tangent complexes (ee Corollary~\ref{L:gE2+=TA}).
We now introduce  Tamarkin deformation complexes of a $Pois_n$-algebra~\cite{Ta-deformationofd-algebra}
and prove that these complexes \emph{do control deformations} of (dg-)$Pois_n$-algebras. Relying on the formality of $E_n$-operads, we will use these complexes to simplify the computations of the deformation
complex of symmetric bialgebras in Section~\ref{S:GSofSym}. We only need the case $n=2$, but the proofs are just as easy for a general $n$ so we do it in this generality.

\smallskip

We denote by $Pois_n$ the operad of $Pois_n$-algebras and $uPois_n$ the operad of unital $Pois_n$-algebras.

\smallskip

Let $A$ be a dg $Pois_n$-algebra, with structure morphism $\psi:Pois_n\rightarrow End_A$.  We denote by $CH_{Pois_n}^*(A,A)$ its $Pois_n$-Hochschild cochain complex, also referred to 
as its $Pois_n$-deformation complex as defined by Tamarkin~\cite{Ta-deformationofd-algebra} and Kontsevich~\cite{Ko1}.
Following
Calaque-Willwacher~\cite{CaWi}, we note that this complex is given by the suspension
\begin{equation}\label{eq:CHPoisDef}
 CH_{Pois_n}^{*}(A,A) \, := \, Hom_{\Sigma}(u {Pois_n}^* \{n\}, End_A)[-n] 
\end{equation}
of the underlying chain complex of the convolution Lie algebra. 
Here  $(-)^*$ is the linear dual and $\{n\}$ is the operadic $n$-iterated suspension.
The inclusion of $Pois_n$ in $uPois_n$ induces a splitting (as a graded space)
\begin{equation}\label{eq:CHPoisDefSplit}
CH_{Pois_n}^{*}(A,A) \, \cong \, A \oplus Hom_{\Sigma}({Pois_n}^* \{n\}, End_A)[-n]
\end{equation}
and also gives rise to 
the \emph{truncated} deformation complex
\begin{equation}\label{eq:CHPoisDefTruncated} CH_{Pois_n}^{(\bullet>0)}(A,A) =  Hom_{\Sigma}({Pois_n}^* \{n\}, End_A)[-n] 
\end{equation}
obtained by deleting the \lq\lq{}unit part\rq\rq{} $A$,
which is more relevant to deformations of $Pois_n$-algebras\footnote{as opposed to deformation of categories of modules}, 
see Lemma~\ref{L:gPoiss=Tam}. Note that both complexes are naturally bigraded with respect to the internal grading of $A$ 
and the \lq\lq{}operadic\rq\rq{} grading coming from $u {Pois_n}^*$.
The notation $CH_{Pois_n}^{(\bullet>0)}(A,A)$ is there to suggest that we are taking the 
subcomplex with positive weight with respect to the operadic grading.

\smallskip

The suspensions $CH_{Pois_n}^{*}(A,A) [n]$ and $CH_{Pois_n}^{(\bullet>0)}(A,A)[n] $ have \emph{canonical 
$L_\infty$-structures} since they are convolution algebras, and $CH_{Pois_n}^{(\bullet>0)}(A,A)[n]$ is canonically a sub $L_{\infty}$-algebra of $CH_{Pois_n}^{*}(A,A)[n]$.
Tamarkin~\cite{Ta-deformationofd-algebra}
(see also~\cite{Ko1, CaWi}) proved that \emph{the complex $CH_{Pois_n}^{*}(A,A)$ actually inherits a (homotopy) $Pois_{n+1}$-algebra 
structure} lifting this $L_{\infty}$-structure. 
Further, by~\eqref{eq:CHPoisDefSplit} we have an exact sequence of cochain complexes
\begin{equation}
 \label{eq:fibersequencePoisn}
 0\longrightarrow CH_{Pois_n}^{(\bullet>0)}(A,A)  \longrightarrow CH_{Pois_n}^{*}(A,A) \longrightarrow A \longrightarrow 0
\end{equation}
which yields after suspending the exact triangle
\begin{equation}
 \label{eq:fibersequencePoisnasLie}
  A[n-1] \stackrel{\partial_{Pois_n}[n-1]}\longrightarrow CH_{Pois_n}^{(\bullet>0)}(A,A)[n]  \longrightarrow CH_{Pois_n}^{*}(A,A)[n]. 
\end{equation}
\begin{rem}\label{R: hofibPois}
The map $\partial_{Pois_n}: A \subset CH_{Pois_n}^{*}(A,A)\to CH_{Pois_n}^{(\bullet>0)}(A,A)$ is the part of the 
differential in the cochain complex
$CH_{Pois_n}^{*}(A,A) =A \oplus CH_{Pois_n}^{(\bullet>0)}(A,A)$ which comes from the operadic structure. 
That is $\partial_{Pois_n}(x) \in Hom(A,A)$ is the map $a\mapsto \pm [x,a]$ where the bracket is the bracket of the $Pois_n$-algebra.
The Jacobi identity for the Lie algebra $A[n-1]$ implies that the sequence~\eqref{eq:fibersequencePoisnasLie} is a sequence of
$L_\infty$-algebras.
\end{rem}
\begin{rem}
 The operad $Pois_n$ is denoted $e_n$ in \cite{CaWi, Ta-deformationofd-algebra} and the complex
 $CH_{Pois_n}^{*}(A,A)$ is simply denoted $def(A)$ in Tamarkin~\cite{Ta-deformationofd-algebra}. We prefer to use the notations we have introduced by analogy with (operadic) Hochschild complexes.
\end{rem}
The next Lemma compares the $L_\infty$-algebra structure of the truncated $Pois_n$ Hochschild complex and the 
one associated to the derived algebraic group of homotopy automorphisms of a $Pois_n$-algebra:
\begin{lem}\label{L:gPoiss=Tam}
Let $A$ be a dg $Pois_n$-algebra with structure map $\psi:Pois_n\rightarrow End_A$. There is an equality of dg Lie algebras
$$ g_{Pois_n^+,A}^{\psi^+} = CH_{Pois_n}^{(\bullet>0)}(A,A) $$
where the right hand side is the truncated cochain complex of a $Pois_n$-algebra defined by Tamarkin as above.
\end{lem}

\begin{proof}
According to the definition of the plus construction $(-)^+$ given in Section~\ref{S:Plus}, we have
\[
Pois_{n\infty}^+=\Omega(Pois_n^*\{n\})^+=(\mathcal{F}(\overline{Pois_n^*\{n+1\}}^+),\partial^+)
\]
where $Pois_{n\infty}$ is the minimal model of $Pois_n$, $(-)^*$ is the linear dual, $\{n\}$ is the operadic $n$-iterated suspension, $\Omega$ is the operadic cobar construction and $\overline{-}$ is the coaugmentation ideal of a coaugmented cooperad.
Recall that the collection of generators $\overline{Pois_n^*\{n+1\}}^+$ is given by
\[
\overline{Pois_n^*\{n+1\}}^+(1)=\overline{Pois_n^*\{n+1\}}(1)\oplus \mathbb{K}[1]=\overline{Pois_n^*\{n+1\}}(1)\oplus \mathbb{K}d
\]
where $d$ is a generator of degree $1$
and
\[
\overline{Pois_n^*\{n+1\}}^+(r)=\overline{Pois_n^*\{n+1\}}(r)
\]
for $r>1$. The restriction of the differential $\partial^+$ on the generators decomposes into $\partial^+=\partial + \delta$
where $\partial$ is the differential of the minimal model and $\delta$ is defined by $\delta(d)=d\otimes d$ and zero when evaluated on the other generators (note that, by the Koszul sign rule and for degree reasons, we have $\delta^2(d)=0$ so we get a differential indeed). Now let $\psi^+:Pois_{n\infty}^+\rightarrow End_A$ be the operad morphism induced by $\psi$, thus a Maurer-Cartan element of the convolution graded Lie algebra $g_{Pois_n^+,A}$. We twist this Lie algebra by $\psi$ to get a dg Lie algebra $g_{Pois_n^+,A}^{\psi^+}$ with the same Lie bracket and whose differential is defined by
\[
\pm (d_A)_*+[\psi,-]
\]
where $(-)_*$ denotes the post-composition, $d_A$ is the differential on $End_A$ induced by the differential of $A$, the $\pm$ sign is defined according to the Koszul sign rule and $[-,-]$ is the convolution Lie bracket. Note here that the Koszul dual cooperad has no internal differential. We refer the reader to \cite[Chapter 12]{LV} for more details about such convolution Lie algebras. Now let us point out that
\[
\overline{Pois_n^*\{n+1\}}^+(1)=\overline{Pois_n^*\{n+1\}}(1)\oplus \mathbb{K}[1]=(\overline{Pois_n^*\{n\}}(1)\oplus \mathbb{K})[1],
\]
which implies that
\[
g_{Pois_n^+,A}^{\psi^+}=Hom_{\Sigma}(\overline{Pois_n^*\{n\}}\oplus I,End_A)^{\psi}=Conv(Pois_n^*\{n\},End_A)
\]
where $Conv(Pois_n^*\{n\},End_A)$ is the convolution Lie algebra of \cite[Section 2.2]{CaWi}.
This is an equality of dg Lie algebras, because the convolution bracket is defined by the action of the infinitesimal cooperadic coproduct on the coaugmentation ideal,  so is the same on both sides.
\end{proof}
\begin{rem}\label{R:VariousDefPoisn}
Lemma~\ref{L:gPoiss=Tam} together with Theorem~\ref{T:Def+=hAut} implies that 
the truncated Tamarkin deformation complex $CH_{Pois_n}^{(\bullet>0)}(A,A)$ controls deformations
of $A$ into dg $Pois_n$-algebras, in other words is the tangent Lie algebra of the derived algebraic group $\underline{haut}_{Pois_{n\infty}}(A)$, where $Pois_{n\infty}$ is a cofibrant resolution of $Pois_n$. 

\smallskip

The proof of Lemma~\ref{L:gPoiss=Tam} also shows that the deformation complex 
$g_{Pois_n,A}^{\psi}$ of the formal moduli problem $\underline{{Pois_{n}}_{\infty} \{A \}}^{\psi}$ is given by the $L_\infty$-algebra 
$CH_{Pois_n}^{(\bullet>1)}(A,A)[n]$, which is the kernel  
\begin{equation}\label{eq:DDefPoisnPas+} CH_{Pois_n}^{(\bullet>1)}(A,A)[n] := \ker\Big(CH_{Pois_n}^{(\bullet>0)}(A,A)[n]\twoheadrightarrow Hom (A, A)[n]\Big) 
\end{equation}
and is thus a even further truncation of $CH_{Pois_n}^*(A,A)$.
The situation is thus similar to what happens in deformation theory of associative algebras.

\smallskip

One can also wonder which deformation problem controls the full complex 
 $CH_{Pois_n}^*(A,A)$. In  view of Theorem~\ref{T:Defcomplexesallagree} below and classical 
 results on deformation theory of $E_n$-algebras (\cite{KellerLowen, Preygel, Fra}), we can conjecture
 that $CH_{Pois_n}^*(A,A)$ shall control deformations of categories of modules over  $Pois_n$-algebras into  $E_{|n-1|}$-monoidal dg-categories, with some shift on the linear enrichment of the category when $n\leq 1$ according to the red shift trick~\cite{Kap-TFT, Toen-ICM, Toe}.
\end{rem}

We now compare the deformation complexes of $Pois_n$-algebras with those of $E_n$-algebras.
Let us fix a formality morphism $\varphi:E_n\stackrel{\sim}{\rightarrow}Pois_n$ ($n\geq 2$). 
This allows to see any dg $Pois_n$-algebra 
as an $E_n$-algebra. 
\begin{lem}\label{L:gE2+=gPois2+} Let $A$ be a dg $Pois_n$-algebra with structure map $\psi:Pois_n\rightarrow End_A$. Then, there is a quasi-isomorphism of dg Lie algebras 
$$CH_{Pois_n}^{(\bullet>0)}(A,A)[n] = g_{Pois_n^+,A}^{\psi^+}\stackrel{\sim}{\rightarrow} g_{E_n^+, A}^{(\psi\circ\varphi)^+}  $$
where the left hand side is the truncated cohomology complex~\eqref{eq:CHPoisDefTruncated} of a $Pois_n$-algebra defined by Tamarkin~\cite{Ta-deformationofd-algebra}.
\end{lem}
\begin{proof}
The formality quasi-isomorphism $\varphi$ induces a quasi-isomorphism of cofibrant resolutions
\[
\varphi_{\infty}:E_{n\infty}\stackrel{\sim}{\rightarrow}Pois_{n\infty}
\]
fitting in a commutative diagram
\[
\xymatrix{
E_{n\infty}=\Omega(E_n^{\textrm{!`}})\ar[d]_-i^-{\sim}\ar[r]^-{\sim} & Pois_{n\infty}=\Omega(Pois_n^{\textrm{!`}}) \\
\Omega\mathcal{B}(E_n)\ar[r]_-{\Omega\mathcal{B}(\varphi)}^-{\sim} & \Omega\mathcal{B}(Pois_n)\ar[u]_-{\Omega(\pi)}^-{\sim} \\
}
\]
where $\mathcal{B}$ is the operadic bar construction, $\Omega$ is the operadic cobar construction,
$(-)^{\textrm{!`}}$ is the Koszul dual cooperad given by $H^0\mathcal{B}(-)$ (see \cite[Chapter 7]{LV}), 
$i$ is the inclusion and $\pi:\mathcal{B}(Pois_n)\twoheadrightarrow H^0\mathcal{B}(Pois_n)$ is the projection. 
This construction implies that $\varphi_{\infty}$ is a tangent quasi-isomorphism of dg operads in the sense of \cite[Theorem 7]{MV2},
that is, the composite map
\[
E_n^{\textrm{!`}}\{1\}\hookrightarrow E_{n\infty}\stackrel{\sim}{\rightarrow}Pois_{n\infty}\twoheadrightarrow Pois_n^{\textrm{!`}}\{1\}
\]
is a quasi-isomorphism. Given that, according to Section 1.2.2, the plus construction $(-)^+$ preserves quasi-isomorphisms of dg props,
this implies that the construction above works with the plus construction as well and induces a tangent quasi-isomorphism
\[
\varphi_{\infty}^+:E_{n\infty}^+\stackrel{\sim}{\rightarrow}Pois_{n\infty}^+,
\]
which by \cite[Theorem 7]{MV2} induces the desired quasi-isomorphism of dg Lie algebras
\[
(\varphi_{\infty}^+)^*:g_{Pois_n^+,A}^{\psi^+} \stackrel{\sim}{\rightarrow} g_{E_n^+, A}^{(\psi\circ\varphi)^+}.
\]
\end{proof}

\begin{cor} \label{C:ModE2+=ModPos2+}
Let $A$ be a $Pois_n$-algebra with structure map $\psi:Pois_n\rightarrow End_A$.
  The formality map $\varphi$ induces an equivalence of formal moduli problems 
 $$ \varphi_{\infty}^*:\underline{{Pois_n^{+}}_{\infty}}^{\psi^+}\{A\}\stackrel{\simeq}\longrightarrow \underline{{E_n^+}_{\infty}}\{\varphi^*A\}^{(\psi\circ\varphi)^+}$$
(where $\varphi^*A$ is the $E_n$-algebra obtained from the $Pois_n$-algebra $A$ by restriction of structures along $\varphi$).
\end{cor}
Further, Lemma~\eqref{L:gE2+=gPois2+} and Corollary~\eqref{L:gE2+=TA} combine to give a quasi-isomorphism of $L_{\infty}$-algebras
\begin{equation}\label{eq:gE2+=gPois2+}
CH_{Pois_n}^{(\bullet > 0)}(A,A)[n] \stackrel{\simeq}\longrightarrow T_A 
\end{equation}
when $A$ is dg-$Pois_n$-algebra. Here $T_A$ is the tangent complex of $A$ viewed as an $E_n$-algebra using the chosen formality $\varphi$ (so strictly speaking it shall be noted $T_{\varphi^*A}$).

\smallskip

Recall now that if $A$ is a $Pois_n$-algebra, then $A[n-1]$ is a dg Lie algebra, in particular an $L_\infty$-algebra.
We relate its deformation complexes (the full and the truncated) as a $Pois_n$-algebra with its deformation complexes as an $E_n$-algebra: 
\begin{thm}\label{T:Defcomplexesallagree} Let $A$ be a $Pois_n$-algebra ($n\geq 2$).

 (1) The sequence~\eqref{eq:fibersequencePoisnasLie} fits into a morphism of fiber sequences of $L_\infty$-algebras
 \[\xymatrix{ A [n-1]\ar[r]  & T_A \ar[r] &  CH_{E_n}^{*}(A,A)[n]\\ 
 A[n-1] \ar[u]^-{\simeq} \ar[r] & CH_{Pois_n}^{(\bullet>0)}(A,A)[n] \ar[u]_-{\simeq}  \ar[r] & CH_{Pois_n}^{*}(A,A)[n] \ar[u]_-{\simeq}  }\]
 whose vertical arrows are $L_\infty$-quasi-isomorphisms and the middle vertical arrow is~\eqref{eq:gE2+=gPois2+}.
 
 (2) The $L_\infty$-structure on $CH_{Pois_n}^{(\bullet>0)}(A,A)[n]$ and $T_A$ is the one controling the derived algebraic group
 $\underline{haut}_{Pois_{n\infty}}(A)$ of homotopy automorphisms of $A$.
\end{thm}
\begin{proof} We saw that the equivalence 
$\varphi_{\infty}:E_{n\infty}\stackrel{\sim}{\rightarrow}Pois_{n\infty}$ induced by the formality morphism $\varphi$
induces the weak-equivalence $\varphi_{\infty}^{+}:E_{n\infty}^{+}\stackrel{\sim}{\rightarrow}Pois_{n\infty}^{+}$. Adding a generator for strict units gives us also a weak-equivalence
$su\varphi_{\infty}^{+}:suE_{n\infty}^{+}\stackrel{\sim}{\rightarrow}suPois_{n\infty}^{+}$.
Let $\psi_\infty: Pois_{n\infty}\to End_A$ be a $Pois_n$-algebra structure on $A$ (seen as a homotopy $Pois_n$-algebra structure),
and $\psi^+_\infty:Pois_{n\infty}^{+}\twoheadrightarrow Pois_{n\infty}\stackrel{\psi_\infty}\to End_A$ be the map obtained by first sending the generator $\delta$ to $0$.
In the proof of Lemma~\ref{L:gE2+=gPois2+}, we saw that $\varphi^+_\infty$ is a tangent quasi-isomorphism, hence inducing an $L_\infty$-algebra quasi-isomorphism at the level of the associated deformation complexes. Following readily the proof of Lemma~\ref{L:gE2+=gPois2+}, we get that $su\varphi^+_\infty$ is a tangent quasi-isomorphism as well, thus
\[
(\varphi^+_\infty)^*:g_{Pois_n^+,A}^{\psi^+_\infty}\stackrel{\sim}{\longrightarrow}g_{E_{n}^{+},A}^{(\psi\circ \varphi)^{+}}
\] and
\[
(su\varphi^+_\infty)^*:g_{suPois_n^+,A}^{\psi^+_\infty}\stackrel{\sim}{\longrightarrow}g_{suE_{n}^{+},A}^{(\psi\circ \varphi)^{+}}
\]
are both $L_\infty$-algebra quasi-isomorphisms fitting in a commutative diagram
\begin{equation}\label{eq:diagEnPoisnandU+}
\xymatrix{ g_{suE_{n}^{+},A}^{(\psi\circ \varphi)^{+}} 
& g_{E_{n}^{+},A}^{(\psi\circ \varphi)^{+}}\ar@{_{(}->}[l]   \\ 
g_{suPois_n^+,A}^{\psi^+_\infty}\ar[u]^-{(su\varphi^+_\infty)^*}_-{\simeq}& g_{Pois_n^+,A}^{\psi^+_\infty} \ar@{_{(}->}[l]\ar[u]_-{(\varphi^+_\infty)^*}^-{\simeq}}.
\end{equation}
As in Remark~\ref{R: hofibPois}, the homopy fibers of these inclusions are both $A[n-1]$ equipped with its Lie structure defined by shifting the Poisson bracket of $A$. We thus get a quasi-isomorphism of homotopy fiber sequences
\begin{equation}
\xymatrix{ g_{suE_{n}^{+},A}^{(\psi\circ \varphi)^{+}} 
& g_{E_{n}^{+},A}^{(\psi\circ \varphi)^{+}}\ar@{_{(}->}[l] & A[n-1]\ar[l]  \\ 
 g_{suPois_n^+,A}^{\psi^+_\infty}\ar[u]^-{(su\varphi^+_\infty)^*}_-{\simeq}& g_{Pois_n^+,A}^{\psi^+_\infty} \ar@{_{(}->}[l]\ar[u]_-{(\varphi^+_\infty)^*}^-{\simeq} & A[n-1]\ar[l]\ar[u]^-{\sim}}.
\end{equation}
Following readily the proof of Lemma~\ref{L:gPoiss=Tam}, we see that $g_{suPois_n^+,A}^{\psi^+_\infty}=CH_{Pois_n}^{*}(A,A)[n]$: the generator of arity zero in $suPois_n^+$ (defining the unit) is of degree $0$, hence of degree $n$ in $\overline{Pois_n^*\{n\}}$. In the associated convolution Lie algebra $g_{suPois_n^+,A}^{\psi^+_\infty}$, it gives an additional term $A[n]$ in operadic degree $0$ which corresponds to the part of weight $0$ in $CH_{Pois_n}^*(A,A)[n]$. The inclusion
$g_{Pois_n^+,A}^{\psi^+_\infty}\hookrightarrow g_{suPois_n^+,A}^{\psi^+_\infty}$ is thus nothing but the inclusion
$CH_{Pois_n}^{(\bullet>0)}(A,A)[n] \hookrightarrow  CH_{Pois_n}^{*}(A,A)[n]$. This means that we have a quasi-isomorphism of homotopy fiber sequences
\begin{equation}
\xymatrix{ g_{suE_{n}^{+},A}^{(\psi\circ \varphi)^{+}} 
& g_{E_{n}^{+},A}^{(\psi\circ \varphi)^{+}}\ar@{_{(}->}[l] & A[n-1]\ar[l]  \\ 
 CH_{Pois_n}^{*}(A,A)[n]\ar[u]^-{(su\varphi^+_\infty)^*}_-{\simeq}& CH_{Pois_n}^{(\bullet>0)}(A,A)[n] \ar@{_{(}->}[l]\ar[u]_-{(\varphi^+_\infty)^*}^-{\simeq} & A[n-1]\ar[l]\ar[u]^-{\sim}}.
\end{equation}

Now we want to prove that the full Poisson complex $CH_{Pois_n}^{*}(A,A)$ is quasi-isomorphic to the full $E_n$-Hochschild complex $CH_{E_n}^*(A,A)$. For this, we use the fact (proved later in Lemma~\ref{L:Cobarn=Pois}) that the formality morphism
\[
\varphi:E_n\stackrel{\sim}{\rightarrow}Pois_n
\]
induces a natural quasi-isomorphism
\[
Bar^{(n)}\circ\varphi^*\stackrel{\sim}{\rightarrow} Bar^{(n)}_{Pois_n}, 
\]
where $Bar^{n}$ is the operadic bar construction for $E_n$-algebras, $Bar^{(n)}_{Pois_n}$ is the one for $Pois_n$-algebras and $\varphi^*$ is the restriction of structures from $Pois_n$-algebras to $E_n$-algebras. So, for a given $n$-Poisson algebra $A$, we have a quasi-isomorphism
\[
Bar^{(n)}(\varphi^*A)\stackrel{\sim}{\rightarrow} Bar^{(n)}_{Pois_n}(A)
\]
inducing a quasi-isomorphism
\[
Hom_{dg}(Bar^{(n)}_{Pois_n}(A),A) \stackrel{\sim}{\rightarrow} Hom_{dg}(Bar^{(n)}(\varphi^*A),A)
\]
between the complexes computing respectively the operadic $Pois_n$-cohomology of $A$ and the operadic $E_n$-cohomology of $\varphi^*A$ (see \cite[Section 12.4]{LV} for the general construction of operadic cohomology).
By \cite[Theorem 12.4.5]{LV}, the operadic complex of an algebra over an operad is quasi-isomorphic to its André-Quillen complex so we get a quasi-isomorphism of André-Quillen complexes
\[
C_{AQ}^{Pois_n}(A,A)\stackrel{\sim}{\rightarrow}C_{AQ}^{E_n}(\varphi^*A,\varphi^*A).
\]
By definition, the André-Quillen complex $C_{AQ}^{E_n}(\varphi^*A,\varphi^*A)$ is the right derived functor, applied to $\varphi^*A$, of the hom functor $Hom_A^{E_n}(-,\varphi^*A)$ in operadic $A$-modules over $E_n$, that is, exactly $ CH_{E_n}^*(\varphi^*A,\varphi^*A)$. Similarly, we have $C_{AQ}^{Pois_n}(A,A)=CH_{Pois_n}^{*}(A,A)$. Tensoring the map above by $\mathbb{K}[n]$, we get a quasi-isomorphism
\[
CH_{Pois_n}^{*}(A,A)[n]\stackrel{\sim}{\rightarrow} CH_{E_n}^*(A,A)[n].
\]
The equality  $g_{suPois_n^+,A}^{\psi^+_\infty}=CH_{Pois_n}^{*}(A,A)[n]$ gives to $CH_{Pois_n}^{*}(A,A)[n]$ an $L_{\infty}$-algebra structure that we can transfer to $CH_{E_n}^*(A,A)[n]$ by the quasi-isomorphism above to make it a quasi-isomorphism of $L_{\infty}$-algebras.

We now compare this structure with the one given by the solution to the Deligne conjecture. The $E_n$-Hochschild cohomology
has a natural (and unital) $E_{n+1}$-structure (see~\cite{Lur2, Fra, GTZ}) given by the $suE_1\otimes suE_n$-algebra 
structure\footnote{The centralizer is canonically endowed with a structure of $E_1$-algebras in $E_n$-algebras.
By Dunn additivity theorem~\cite{Lur2}, this is equivalent to an $E_{n+1}$-structure. }
on the $E_n$-centralizer 
$\mathfrak{z}_{suE_n}(\varphi^*(A)):=\mathfrak{z}_{suE_n}(id_{\varphi^*(A)})$ of $A$, seen as an $E_n$-algebra using our chosen formality map. 
By \cite{Fra}, the upper horizontal sequence of chain complexes in claim (1) of the theorem is a natural sequence of $L_\infty$-algebras. Here, 
by the theory of formal moduli problems~\cite{Lur0}
and deformation theory of $E_n$-algebras~\cite{Lur2, Fra},  
$CH_{E_n}^*(A,A)[n]$
is endowed with the underlying $L_\infty$-algebra structure characterized by the commutative formal moduli problem defined by
\begin{equation}
 \label{eq:FMPCenterEn}  Map_{E_{n+1}-Alg^{aug}}(\mathcal{D}^{(n+1)}\circ I^{(n+1)} (-), 
 \mathfrak{z}_{suE_n}(\varphi^*(A)))
\end{equation}
where $I^{(n+1)}:cdga^{aug}\to  E_n^{aug}$ is the canonical functor sending a cdga to the $E_n$-algebra with trivial higher products, and  $\mathcal{D}^{(n+1)}$ is the appropriate
Koszul duality $\infty$-functor as used in \cite{Lur0}. Note that  $\mathcal{D}^{(n+1)}$ can be identified with the centralizer of the augmentation $\mathfrak{z}_{suE_n}(R\to \mathbb{K})$,
hence with the dual of the $n+1$-fold bar construction, see~\cite{Lur0, GTZ, AF}.

\smallskip

Let us denote $ \mathfrak{z}_{suPois_n}(A)$ the centralizer of the identity map $A\to A$ in the $\infty$-category
of $Pois_n$-algebras; it is a (unital) $E_1$-algebra in $Pois_n$-algebras.
Similarly as above, the formality map $\varphi$ induces an equivalence
of $E_1$-algebras in $E_n$-algebras (hence of unital $E_{n+1}$-algebras):
\begin{equation} \varphi^*\big(\mathfrak{z}_{suPois_n}(A)\big) \, \cong \, \mathfrak{z}_{suE_n}(\varphi^*(A)). \end{equation}
By~\cite{Ta-deformationofd-algebra} (also see~\cite{CaWi, GTZ}), the $E_1\otimes Pois_n$-algebra structure of $ \mathfrak{z}_{suPois_n}(A)$ 
is represented by an explicit 
compatible\footnote{the compatibility means that we have an $E_1\otimes Pois_{n \infty}$-algebra structure} dg algebra
structure on the dg $Pois_n$-coalgebra 
$Bar^{(n)}_{Pois_n}(\mathfrak{z}_{suPois_n}(A)))$ whose differential encodes its canonical $Pois_n\infty$-algebra structure\footnote{
 Note that the (non-unital) 
centralizer and is corresponding dg $Pois_n$-coalgebra  are denoted $def(A)[-n]$ and  $\underline{Hom}^{Id}(A,A)$
in~\cite{Ta-deformationofd-algebra}}. 

\smallskip

We wish to rewrite the formal moduli problem~\eqref{eq:FMPCenterEn}. To do so, following~\cite[Sections 2.3 and  3.3]{Lur2}, we note 
that we have an adjunction of $\infty$-categories $C_{CE}^*: dgLie  \rightleftarrows dgArt_{\mathbb{K}}^{aug} :\mathcal{D}_{Lie}$, where
$C_{CE}^*$ is the $\infty$-functor induced by the Chevalley-Eilenberg cochain algebra functor, and $\mathcal{D}_{Lie}$ is an equivalence onto its essential image in  $ dgLie$ (with inverse given by $C_{CE}^*$).
Furthermore, by \cite[Corollary 4.2.2]{AF}, there is an equivalence 
\begin{equation}\label{eq:DnUn} \mathcal{D}^{(n+1)}\circ U_{n+1} \big(\mathcal{D}_{Lie}(R)\big) \, 
 \cong \, C_{CE}^* \big(\mathcal{D}_{Lie}(R)\big)
\end{equation}
where $U_{n+1}$ is the the $E_{n+1}$-envelopping algebra functor; in particular we have an adjunction
$U_{n+1}: E_{n+1}-Alg^{aug} \leftrightarrows dgLie: [n] $. Since $\mathcal{D}^{(n+1)}$ is self-dual, for an artinian augmented cdga $R$, we get from~\eqref{eq:DnUn} and~\cite[Theorem 4.4.5]{Lur0}, an equivalence 
$$ U_{n+1} \big(\mathcal{D}_{Lie}(R)\big) \, 
 \cong \, \mathcal{D}^{(n+1)}\circ C_{CE}^* \big(\mathcal{D}_{Lie}(R)\big).$$
Combining all the above  we get an equivalence of formal moduli problems given, for every artinian augmented cdga $R$, by
\begin{multline}
 Map_{E_{n+1}-Alg^{aug}}\Big(\mathcal{D}^{(n+1)}\circ I^{(n+1)} (R), 
 \mathfrak{z}_{suE_n}(\varphi^*(A))\Big)   \\
 \cong \,Map_{E_{n+1}-Alg^{aug}}\Big(\mathcal{D}^{(n+1)}\circ  C_{CE}^*(\mathcal{D}_{Lie}(R)), \,
 \varphi^*\big(\mathfrak{z}_{suPois_n}(A))\Big) \\
 \cong \, Map_{E_{n+1}-Alg^{aug}}\Big(U_{n+1} (\mathcal{D}_{Lie}(R)), \,
 \varphi^*\big(\mathfrak{z}_{suPois_n}(A))\Big) \\ 
 \cong \, Map_{dgLie}\Big(\mathcal{D}_{Lie}(R),\mathfrak{z}_{suPois_n}(A)[n]\Big).
 \end{multline}
This identifies the $L_{\infty}$-algebra structure underlying $CH_{E_n}^*(A,A)\cong \mathfrak{z}_{suE_n}(\varphi^*(A))$
with the one obtained by $n$-fold desuspension of 
$\mathfrak{z}_{suPois_n}(A)$. It remains to prove that this structure on the center coincides with the one on 
$CH_{Pois_n}^{*}(A,A)[n]=g_{suPois_n^+,A}^{\psi^+_\infty}$; so far we have proved that they are quasi-isomorphic as chain complexes. 
Tamarkin~\cite[Section 5]{Ta-deformationofd-algebra} has proved that $CH_{Pois_n}^{*}(A,A)[n]$ has the underlying $L_\infty$-algebra structure of a certain $Pois_{n+1 \infty}$-algebra structure on $\mathfrak{z}_{suPois_n}(A)\cong CH_{Pois_n}^{*}(A,A)[n]$. This  $Pois_{n+1 \infty}$-algebra structure is given by the differential on
$Bar^{(n+1)}_{Pois_{n+1}}(\mathfrak{z}_{suPois_n}(A))$ and we use the later notation (or simply 
$CH_{Pois_n}^{*}(A,A)[n]$) to refer to this precise homotopy $Pois_{n+1}$-algebra
structure on $\mathfrak{z}_{suPois_n}(A)$, in order to distinguish it from the one given by the centralizer.
Further Tamarkin showed that this structure induces a cocommutative Hopf bialgebra structure on
$Bar^{(n)}_{Pois_n}(\mathfrak{z}_{suPois_n}(A))$, an additional compatible degree $n$ Lie cobracket on its primitive part  and 
finally an isomorphism of dg $Pois_{n+1}$-coalgebras:
\begin{equation}\label{eq:centerasHopfequiv}
 Bar^{(n+1)}_{Pois_{n+1}}(\mathfrak{z}_{suPois_n}(A)) \cong C^{CE}_*(Prim (Bar^{(n)}_{Pois_n}(\mathfrak{z}_{suPois_n}(A))))
\end{equation}
where $C^{CE}_*$ is the standard Chevalley-Eilenberg commutative coalgebra functor (whose associated derived functor computes 
$\mathbb{K}\otimes_{U(-)}^{\mathbb{L}} \mathbb{K}$).

As above, Tamarkin $L_\infty$-structure on  $CH_{Pois_n}^{*}(A,A)[n]$ is characterized by the following 
equivalence of moduli functors
\begin{multline} \label{eq:ModuliforTamarkinstruct}
 Map_{dgLie}\big(\mathcal{D}_{Lie}(R),CH_{Pois_n}^{*}(A,A)[n]\Big) \\
 \cong Map_{Pois_{n+1}-Alg^{aug}}\Big(\mathcal{D}_{Pois_{n+1}}^{(n+1)}\circ I^{(n+1)} (R), 
CH_{Pois_n}^{*}(A,A))\Big) \\
\cong Map_{Pois_{n+1}-Cog}\Big(I_{(n+1)} (R'), 
Bar^{(n+1)}_{Pois_{n+1}}(\mathfrak{z}_{suPois_n}(A))\Big) \\
\end{multline}
where $I_{(d)}:  Com-Cog \to Pois_{d}-Cog$ is the canonical functor and $R'$ is the (linear) dual of the artinian algebra $R$. 
The last equivalence follows from the bar-cobar equivalence and the definition of $Pois_{n+1\infty}$-morphisms. 

\smallskip

Recall that the centralizer $\mathfrak{z}_{suPois_n}(A)$ is canonically an $E_1$-algebra in $Pois_{n\infty}$-algebras and 
this structure is given by a compatible algebra structure on $Bar^{(n)}_{Pois_n}(\mathfrak{z}_{suPois_n}(A))$. In particular
it is canonically isomorphic as a cocommutative Hopf algebra to the envelopping algebra of its primitive elements $
Prim(Bar^{(n)}_{Pois_n}(\mathfrak{z}_{suPois_n}(A)))$. Also if $C$ is a dg cocommutative coalgebra, then its cobar
construction $coBar^{(1)}(C)$ is canonically a cocommutative Hopf algebra. Using moreover
the natural equivalence $Bar^{(1)}\circ U(-) \cong C^{CE}_*(-)$ of dg cocommutative coalgebras 
and the adjunction between $Bar^{(1)}$ and $coBar^{(1)}$,
we get an equivalence of moduli problems
\begin{multline}\label{eq:ModuliforTamarkinascenter}
 Map_{dgLie}\Big(\mathcal{D}_{Lie}(R),\mathfrak{z}_{suPois_n}(A)[n]\Big)\\
 \cong Map_{Pois_{n+1}-Alg^{aug}}\Big(\mathcal{D}_{Pois_{n+1}}^{(n+1)}\circ I^{(n+1)} (R), 
\mathfrak{z}_{suPois_n}(A)\Big) \\
\cong Map_{E_1\otimes Pois_{n}-Cog}\Big( I_{(n)} \circ coBar^{(1)} ( R'), 
Bar^{(n)}_{Pois_n}(\mathfrak{z}_{suPois_n}(A))\Big) \\
\cong Map_{Pois_{n+1}-Cog}\Big(I_{(n+1)}(R'), Bar^{(1)}\circ Bar^{(n)}_{Pois_n}(\mathfrak{z}_{suPois_n}(A))\Big)\\
\cong Map_{Pois_{n+1}-Cog}\Big( I_{(n+1)}(R'), C^{CE}_*(Prim (Bar^{(n)}_{Pois_n}(\mathfrak{z}_{suPois_n}(A))))\Big).
\end{multline}
Therefore the isomorphism~\eqref{eq:centerasHopfequiv} implies that the two moduli functors~\eqref{eq:ModuliforTamarkinstruct} 
and~\eqref{eq:ModuliforTamarkinascenter} are equivalent. Consequently, the two $L_\infty$-algebras that they represent are equivalent as well by Lurie's equivalence theorem, which yields the right vertical equivalence of $L_\infty$-algebras in Claim~(1).

Moreover, by Corollary~\ref{L:gE2+=TA} we have an equivalence of $L_{\infty}$-algebras 
$g_{E_n^+,A}^{(\psi\circ\varphi)^+}\stackrel{\sim}{\rightarrow}T_A$ and
we finally get the commutative diagram of $L_\infty$-algebras
\[\xymatrix{ CH_{E_n}^*(A,A)[n] &T_A \ar@{_{(}->}[l]\\ 
CH_{Pois_n}^{*}(A,A)[n] \ar[u]^-{\simeq} & CH_{Pois_n}^{(\bullet>0)}(A,A)[n] \ar[u]_{\simeq}\ar@{_{(}->}[l]}.\] 
Hence both lines have quasi-isomorphic fibers as well which are $A[n-1]$ according to \eqref{eq:fibersequencePoisnasLie} and~\cite{Fra}.

 \smallskip
 
Now Claim (2) follows from Lemma~\ref{L:gPoiss=Tam}.
\end{proof}
\begin{rem}
Claim (2) can be proved directly by using the equivalence $
(\varphi_{\infty}^+)^*:\underline{{Pois_n}_{\infty}}^+\{A\}\stackrel{\sim}{\rightarrow} \underline{{E_n}_{\infty}}^+\{\varphi^*A\}$ of Corollary~\ref{C:ModE2+=ModPos2+}.
Indeed,
applying Theorem 3.5 gives us directly a quasi-isomorphism of $L_{\infty}$-algebras
\[
CH_{Pois}^{*>0}(A,A)[n]= g_{Pois_n^+,A}^{\psi^+}\stackrel{\sim}{\rightarrow} T_{\varphi^*A}.
\]
\end{rem}

\begin{rem} In particular, we also get out of Theorem~\ref{T:Defcomplexesallagree} that the
  $L_\infty$-algebra $  g_{Pois_n,A}^{\psi} $ controlling the moduli space  $\underline{{Pois_n}_{\infty}}\{A\}$ is thus isomorphic to a sub-$L_\infty$-algebra of $T_{A}$.
\end{rem}
\begin{rem} Let  $\tilde{C}_{(d)}: Pois_{d}-Cog\to Com-Cog$ be the right adjoint of the canonical functor 
$I_{(d)}:Com-Cog\to Pois_{d}-Cog$.  The moduli functor~\eqref{eq:ModuliforTamarkinascenter}
encoding the $L_\infty$-structure on $CH_{Pois}(A,A)[n]$ is then equivalent to
\begin{multline*}
 Map_{Pois_{n+1}-Alg^{aug}}\Big(\mathcal{D}_{Pois_{n+1}}^{(n+1)}\circ I^{(n+1)} (R), 
CH_{Pois_n}^{*}(A,A))\Big) \\
\cong Map_{Pois_{n+1}-Cog}\Big(I_{(n+1)}(R'), Bar^{(1)}\circ Bar^{(n)}_{Pois_n}(\mathfrak{z}_{suPois_n}(A))\Big)\\
\cong Map_{E_1\otimes Com-Cog}\Big( coBar^{(1)}(R'), \tilde{C}_{(n)}(Bar^{(n)}_{Pois_n}(\mathfrak{z}_{suPois_n}(A)))\Big)\\
\cong  Map_{E_1\otimes Com-Cog}\Big( coBar^{(1)}(R'), U(\mathfrak{z}_{suPois_n}(A)[n])\Big).
\end{multline*}
wher the last line is \cite[Corollary 4.7]{Ta-deformationofd-algebra}. By~\eqref{eq:ModuliforTamarkinstruct}, this is also equivalent to
 \begin{multline*} Map_{Pois_{n+1}-Cog}\Big(I^{(n+1)} (R'), 
Bar^{(n+1)}_{Pois_{n+1}}(\mathfrak{z}_{suPois_n}(A))\Big) \\
\cong Map_{Com-Cog}\Big(R', 
\tilde{C}_{(n+1)} (Bar^{(n+1)}_{Pois_{n+1}}(\mathfrak{z}_{suPois_n}(A)))\Big).
\end{multline*}
\end{rem}

Let $\eta: B \to k$ be an augmented unital cdga and $\overline{B}:= \ker(\eta)$ be its augmentation ideal, so that 
$B \cong k\oplus \overline{B}$. Both $B$ and $\overline{B}$ are cdgas, hence (non-unital) $E_n$-algebras for any $n$. 
In particular, $B$ and $\overline{B}$ are canonically $E_n$-$B$-modules as well as $E_n$-$\overline{B}$-modules. Let us denote 
$\iota: \overline{B}\to B$ the canonical inclusion.
\begin{lem}\label{L:reducedHochCochain} The canonical maps 
 \[CH^*_{E_n}(B,B) \stackrel{\iota^{*}}\to CH^*_{E_n}(\overline{B}, B), \quad 
 CH^*_{E_n}({B}, \overline{B}) \stackrel{\iota^{*}}\to  CH^*_{E_n}(\overline{B}, \overline{B})\]
 are quasi-isomorphisms and the following diagram
 \[\xymatrix{ & CH^*_{E_n}(B,B) \ar[rd]^{\iota^{*}}_{\simeq} & \\ 
CH^*_{E_n}(B,\overline{B}) \ar[rd]^{\iota^{*}}_{\simeq}  \ar[ru]_{\iota_*} & & CH^*_{E_n}(\overline{B}, B)  \\
& CH^*_{E_n}(\overline{B},\overline{B}) \ar[ru]^{\iota_{*}} &  } \]
 is a commutative diagram of $E_{n+1}$-algebras. 
\end{lem}
\begin{proof}
 The $E_n$-Hochschild
 cohomology of a cdga  $B$ with value in a bimodule $M$ can be computed using Pirashvili type higher Hochschild cochains (see~\cite{GTZ}):
 \begin{eqnarray*} CH^*_{E_n}(B, M) \,\cong \, 
  CH^{S^n_\bullet}( B, M) \,\cong\, Hom_{B}(B^{\otimes S^n_\bullet}, M)
 \end{eqnarray*}
 for any simplicial set model $S^n_\bullet$ of the $n$-sphere. Being obtained as a cosimplicial cochain complex, the later one 
 is naturally quasi-isomorphic to its non-degenerate cochains, which amounts to quotienting
 by the submodule generated by the unit (see~\cite{GTZ1}).
  Hence, denoting $*$ the base-point of $S^n_\bullet$, since $B$ is augmented, we have a quasi-isomorphism 
  \[q_*:Hom_{B}(B^{\otimes S^n_\bullet}, M) \stackrel{\simeq}\to Hom_{B}(B\otimes \overline{B}^{\otimes S^n_\bullet\setminus \{*\}}, M) 
  \cong CH_{E_n}^{*}(\overline{B},M).\]
  Since $q_* \circ \iota^*$ is the identity, we obtain that $\iota^*$ is a quasi-isomorphism as well, hence the first claim by taking $M=B$
  or $\overline{B}$. The commutativity of the diagram follows from the same identification with higher Hochschild cochains and 
  a straightforward computation.
  
 Since all modules structures are induced by maps of cdgas,  the same identification with higher Hochschild cochains implies 
 that all complexes in the diagram above
 are canonically $E_{n+1}$-algebras, this structure being functorial with respect to maps of $E_n$-algebras, see~\cite{GTZ}; this also follows 
 from \cite{Fra} using moduli problems.  
\end{proof}
From Lemma~\ref{L:reducedHochCochain}, we get a canonical (zigzag of) cochain maps 
\begin{equation}\label{eq:mapaugmentationtofull} \iota: \, CH^*_{E_n}(\overline{B},\overline{B})
\simeq CH^*_{E_n}({B},\overline{B})\stackrel{\iota_*}\longrightarrow CH^*_{E_n}(B,B). \end{equation} 
Note that we endow $CH^*_{E_n}({B},\overline{B})$ with its $E_{n+1}$-algebra structure given by transfer along 
$\iota^*:CH^*_{E_n}(B,B) \stackrel{\iota^{*}}\to CH^*_{E_n}(\overline{B}, B)$.
\begin{prop}\label{P:augmentedtononaugmentedDeligne} Assume $n\geq 2$ and $B$ is an augmented cdga  essentially of finite type.
 The canonical map~\eqref{eq:mapaugmentationtofull}
 $\iota: CH^*_{E_n}(\overline{B}, \overline{B}) \to CH^*_{E_n}(B,B)$ is a (non-unital) 
 $E_{n+1}$-algebra homomorphism. 
 Further, we have a cofiber sequence of cochain complexes
 \[CH^*_{E_n}(\overline{B}, \overline{B}) \stackrel{\iota}\longrightarrow CH^*_{E_n}(B,B) \stackrel{\pi_*}\longrightarrow CH^*_{E_n}(B,\mathbb{K}).  
 \]
\end{prop}
\begin{proof} Taking a cofibrant resolution of $B$ as a cdga, we may assume $B = \big(Sym(W), d\big)$ is an augmented cofibrant
cdga and each $W$ has finite dimensional cohomology; in particular 
$d (W)\subset Sym^{\geq 1}(W)$ and $\overline{B}= Sym^{\geq 1}(W)$. Being cdgas, these algebras are canonically 
$Pois_{n}$-algebras, and this structure is compatible with the induced $E_n$-algebra structure. 
By \cite{CaWi}, \cite[Corollary 6.39]{GTZ}\footnote{beware that one uses an homological grading in \emph{loc. cit.}
while we are using a  cohomological grading}, we have, for any  \emph{cofibrant} cdga $A$,
an equivalence of $E_{n+1}$-algebras 
\begin{equation} \label{eq:CHEnSymshifted} CH^*_{E_n}(A, A) \simeq \widehat{Sym}_A(Der(A,A)[-2]) \end{equation}
where the shifted Lie bracket is given by the Lie bracket on derivations and the Leibniz rule. Here $Der(A,A)$ is
the cochain complex of derivations and the $E_{n+1}$-algebra structure on the right hand side of~\eqref{eq:CHEnSymshifted} is induced by formality 
from its canonical $Pois_{n+1}$-structure. 
By Theorem~\ref{T:Defcomplexesallagree}, we can replace $CH^*_{E_n}(A, A) $ by 
$CH_{Pois_n}^*(A,A)= Hom_{\Sigma}(u {Pois_n}^* \{n\}, End_A)[-n] $ in  the above equivalence.

Since $B = \big(Sym(W), d\big)$, we have an isomorphism of complexes:
\[ \iota^*: Der(\overline{B}, B)  \;\cong\; Der(B,B)\]
as well as  canonical quasi-isomorphisms
\[(W)^*\otimes B \stackrel{\simeq}\to Der(B,B), \quad  (W)^*\otimes \overline{B} \stackrel{\simeq}\to Der(\overline{B},
\overline{B})\] identifying the canonical map $\iota_*: Der(\overline{B},
\overline{B})\hookrightarrow Der(B,B)$ (extending a derivation on $\overline{B}$ 
to a derivation on $B$) with the map $id_{(W)^*}\otimes \iota$.  
The equivalence~\eqref{eq:CHEnSymshifted} is given by the map
\[ \widehat{Sym}_A(Der(A,A)[-n]\subset Hom_{\Sigma}(u {Com}^* \{n\}, End_A)[-n] \hookrightarrow Hom_{\Sigma}(u {Pois_n}^* \{n\}, End_A)[-n]\] 
where the right map is induced by the canonical operad morphism $Com\to Pois_n$. Hence we have a commutative diagram
\begin{equation}\label{eq:HKRforAugmentationEn}
 \xymatrix{ \widehat{Sym}\big((W)^*\big)\otimes Sym^{\geq 1}\big(W\big) \ar[r]^{\cong} \ar@{^{(}->}[d]
 &\widehat{Sym}_B\big(Der(\overline{B},\overline{B})[-n]\big)\ar[d]_{Sym_B(\iota_*)} 
 \ar[r]^{\qquad\simeq} & CH_{E_n}^*(\overline{B}, \overline{B}) \ar[d]^{\iota} \\
 \widehat{Sym}\big((W)^*\big)\otimes Sym\big(W\big) \ar[r]^{\cong} & \widehat{Sym}_B\big(Der(B,B)[-n]\big) 
 \ar[r]^{\qquad\simeq} & CH_{E_n}^*({B}, {B}) .}
\end{equation}
Since the left vertical map is a map of dg $Pois_{n+1}$-algebras and the horizontal maps are equivalences of $E_{n+1}$-algebras, 
the first claim follows. The second claim follows from the fact that the left hand side is a cofiber sequence 
with cofiber $\widehat{Sym}\big((W)^*\big)\otimes k$.
\end{proof}

\subsection{Gerstenhaber-Schack cochain complexes for bialgebras and proof of Corollary~\ref{C:GS=HH}}\label{SS:GScomplex}

We now precisely relate Gerstenhaber-Schack complexes~\cite{GS} and $E_2$-Hochschild cohomology. 
As for $Pois_n$ and $E_n$-algebras, there are 
several more or less truncated complexes one can consider and encounter in the literature\footnote{beware that, unfortunately, 
the terminology, notations or
degree shifting are not consistent in the literature.},
each related to different deformation problems. Let us  fix some notations.
What we call the \emph{Gerstenhaber-Schack complex} is the total complex of a bicomplex, defined by
\begin{equation}\label{eq:DefCGS}
C_{GS}^*(B,B)\cong \prod_{m,n\geq 1}Hom_{dg}(B^{\otimes m},B^{\otimes n})[-m-n].
\end{equation}
The horizontal differential is defined, for every $n$, by the Hochschild differential associated to the Hochschild complex of $B$ seen as an associative algebra with coefficients in the $B$-bimodule $B^{\otimes n}$.
The vertical differential is defined, for every $m$, by the co-Hochschild differential associated to the co-Hochschild complex of $B$ seen as a coassociative coalgebra with coefficients in the $B$-bicomodule $B^{\otimes m}$.
The compatibility between these differentials, which gives us a well defined bicomplex,
follows from the distributive law relating the product and the coproduct of the bialgebra $B$ (see~\cite{GS, Mer2} for details).
This is the complex relevant for Etingof-Kazdhan quantization in Section~\ref{S:EKQ} because 
\begin{prop}[\cite{Mer2}]The Gerstenhaber-Schack complex is quasi-isomorphic to the deformation complex of dg-bialgebras 
(up to isomorphisms):
$$C_{GS}^*(B,B) \cong  g_{Bialg_{\infty}^+,B}^{\varphi^+}.$$
\end{prop}
Note here that in the definition of $C_{GS}^*(B,B) $ we use truncated versions of 
Hochschild and co-Hochschild complexes. We can also use their full versions and define an extended Gerstenhaber-Schack complex
\begin{equation}
 \label{eq:DefCGStilde}
\tilde{C}_{GS}^*(B,B)\cong \prod_{m\geq 0, n\geq 1}Hom_{dg}(B^{\otimes m},B^{\otimes n})[-m-n],
\end{equation}
which sits inside the fully extended Gerstenhaber-Schack complex
\[
C_{GS}^{full}(B,B)\cong \prod_{m,n\geq 0}Hom_{dg}(B^{\otimes m},B^{\otimes n})[-m-n].
\]
The differentials are defined as for $C_{GS}^*(B,B)$, which is naturally a (non-split) subcomplex of each of the preceding ones 
(\cite{GS}). 
\begin{cor}\label{C:DefcomplexesGS=Hoch}  Let $B$ be a conilpotent dg-bialgebra. 
 There is a commutative diagram of $L_\infty$-algebras
 \[\xymatrix{ \tilde{\Omega}B [1] \ar[r] & T_{\tilde{\Omega}B} \ar[r] &  CH_{E_2}^{*}( \tilde{\Omega}B, \tilde{\Omega}B)[2] \\ 
  \tilde{\Omega}B[1] \ar[u]^{\simeq} \ar[r] & C_{GS}^{*}(B,B)[2] \ar[u]^{\simeq} \ar[r] & \tilde{C}_{GS}^*(B,B))[2] \ar[u]_{\simeq} }\]
 whose vertical arrows are $L_\infty$-quasi-isomorphisms and the middle vertical arrow is given by Lemma~\ref{L:gE2+=TA} and
 the quasi-isomorphism~\eqref{eq:gBialg=gE_2+}.
 The horizontal arrows are further fiber sequences.
\end{cor}
\begin{proof}
 We endow $\tilde{\Omega}B$ with its $E_2$-algebra structure given by Corollary~\ref{C:Barenhanced}.
 The upper horizontal sequence 
 is then given by the main result of~\cite{Fra}. The middle vertical equivalence is assertion (2) in Theorem~\ref{T:EquivDef}.
 A proof similar to the one  of Theorem~\ref{T:Defcomplexesallagree} gives a commutative diagram of $L_\infty$-algebras \begin{equation}
\xymatrix{ g_{suE_{2}^{+},\tilde{\Omega}B}^{(\psi)^{+}} 
& g_{E_{2}^{+},\tilde{\Omega}B}^{(\psi)^{+}}\ar@{_{(}->}[l]   \\ 
 g_{suBialg_{\infty}^+,B}^{\varphi^+}\ar[u]_-{\simeq}& g_{Bialg_{\infty}^+,B}^{\varphi^+}
 \ar@{_{(}->}[l]\ar[u]^-{\simeq} & }.
\end{equation}
, which induces a morphism of fiber sequences,
and an identification of $g_{suBialg_{\infty}^+,B}^{\varphi^+}$ with $C_{GS}^{*}(B,B)[2]$ similar to the one of 
$g_{suPois_{2}^{+},A}^{(\psi)^{+}}$ with $CH^*_{Pois_2}(A,A)$.
 The upper line has been identified with $T_{\tilde{\Omega}B} \longrightarrow  CH_{E_2}^{*}( \tilde{\Omega}B, \tilde{\Omega}B)[2]$
 in the proof of Theorem~\ref{T:Defcomplexesallagree}. The diagram being a diagram of fiber sequences of $L_\infty$-algebras, we 
 get that the fibers are also isomorphic as $L_\infty$-algebras.
\end{proof}

 Let $B$ be a conilpotent dg-bialgebra. Then $\tilde{\Omega}B$ is an $E_2$-algebra by Corollary~\ref{C:Barenhanced} and 
 the sequence of $L_\infty$-algebras 
 \[ \xymatrix{\tilde{\Omega}B [1] \ar[r] & T_{\tilde{\Omega}B} \ar[r] &  CH_{E_2}^{*}( \tilde{\Omega}B, \tilde{\Omega}B)[2]}\]
 lits to a sequence of (non-unital) $E_3$-algebras (after shifting it down by $2$) 
 by the solution to (Kontsevich) higher Deligne conjecture, and more precisely the main result of Francis in~\cite{Fra} (also see~\cite{Lur2, GTZ} for 
 the relationship with Pirashvili Hochschild cohomology and derived centralizers). 
 
 Finally we improved Corollary~\ref{C:DefcomplexesGS=Hoch} into
\begin{cor}[Gerstenhaber-Schack conjecture] \label{C:GSconjecture}
(1) The $E_3$-algebra structures of $CH_{E_2}^{(\bullet>0)}(\tilde{\Omega}B, \tilde{\Omega}B)\cong T_{\tilde{\Omega}B}[-2]$ and 
$CH_{E_2}^{*}( \tilde{\Omega}B, \tilde{\Omega}B)$
induce $E_3$-algebra structures on $C^*_{GS}(B,B)$ and $\tilde{C}_{GS}(B,B))$ 
such that the following diagram  \[\xymatrix{ \tilde{\Omega}B [-1] \ar[r] & T_{\tilde{\Omega}(B)} \ar[r] &  CH_{E_2}^{*}( \tilde{\Omega}B, \tilde{\Omega}B) \\ 
  \tilde{\Omega}B[-1] \ar@{=}[u] \ar[r] & C_{GS}^{*}(B,B) \ar[u]^{\simeq} \ar[r] & \tilde{C}_{GS}(B,B) \ar[u]_{\simeq} }\] 
  is 
a commutative diagram of non-unital $E_3$-algebras with vertical arrows being equivalences. 

(2) The $E_3$-algebra structure on $C^*_{GS}(B,B)$ is
 a refinement of its $L_{\infty}$-algebra structure controlling the deformation theory of the bialgebra $B$.
\end{cor}
\begin{proof}
The diagram and the existence of the lifts of the $E_3$-structures are obtained by transfer of the structure of the first line~\cite{Fra}
along the equivalences given by 
Corollary~\ref{C:DefcomplexesGS=Hoch}. 
We further know that the underlying $L_\infty$-algebra structure of $T_{\tilde{\Omega}(B)}$ is equivalent to the 
one of the deformation complexes of Lie bialgebra $g_{Bialg_{\infty}^+,B}^{\varphi^+}\cong C_{GS}^*(B,B)$ 
by Theorem~\ref{T:EquivDef}. Claim (2) follows.
\end{proof}
\begin{rem}
 We conjecture that the underlying $L_\infty$-algebra structure of the full Gerstenhaber-Schack complex $C_{GS}^{full}(B,B)$ 
 controls deformations as a monoidal dg category of the dg category of representations of $B$. 
\end{rem}
\begin{rem}[\emph{Relevance of the choice of $E_3$-lift}] In view of Theorem~\ref{T:Defcomplexesallagree}, we can also endow the Gerstenhaber-Schack complexes $C_{GS}^*(B,B))$ 
and $C_{GS}^{full}(B,B))$ with $E_3$-algebra structures given by the ones on Tamarkin deformation complexes $CH_{Pois_n}^{\bullet >0}(-,-)$, 
$CH_{Pois_n}^*(-,-)$ (\cite{Ta-deformationofd-algebra, CaWi}), which carry the same underlying $L_\infty$-structure. 
In view of Theorem~\ref{T:EquivDef}, our choice of solutions to the Gerstehaber-Schack conjecture seems more natural. 

With respect to applications to quantization of Lie bialgebras, this choice is however \emph{not} important: 
any $E_3$-algebra structure inducing the correct $L_\infty$-structure will be good enough to compute the deformation complex of bialgebras; the additional datum of the $E_3$-structure is in fact used to simplify computations of deformation complexes by adding more rigidity, and gives in quantization the independence from the choice of an associator. In our case of interest for Etingof-Kazdhan quantization in \S~\ref{S:EKQ}, we are in fact
in a case where both $E_3$-structures are the same. We actually believe that the diagram of Theorem~\ref{T:Defcomplexesallagree} is 
in fact a commutative diagram of $E_3$-algebras in general so that the aforementioned two $E_3$-structures are \emph{always} the same. 
\end{rem}

As a consequence, the Gerstenhaber-Schack complex $C_{GS}^*(B,B)$ inherits a homotopy associative multiplication
(that is an $E_1$-algebra structure obtained through the canonical map $E_1\to E_3$ of operads). 
There is a standard dg associative multiplication on the Gerstenhaber-Schack  complex given by the cup-product~\cite{GS},
which is a model for the Yoneda extension product~\cite{Taillefer-HopfBimod}.
On the other hand, following Corollary~\ref{C:GSconjecture}, the $E_3$-algebra structure on $C_{GS}^*(B,B)$ 
is given by its identification with the derived center 
$\mathbb{R}Hom^{E_2}_{\tilde{\Omega}(B)}(\tilde{\Omega}(B),\tilde{\Omega}(B))$
of $\tilde{\Omega}(B)$. By definition, its  $E_1$-algebra structure is given by composition of $E_2$-$\tilde{\Omega}(B)$-modules
endomorphisms which is the same as derived homomorphisms of left modules over factorization homology
$\int_{S^1}\tilde{\Omega}(B)$ (\cite{Fra, GTZ}); that is, 
it models the Yoneda extension product for $E_2$-$\tilde{\Omega}(B)$-modules. Then, from Theorem~\ref{T:EquivDef}, one can deduce:
\begin{prop} The $E_1$-structure induced by Corollary \ref{C:GSconjecture} on  $C_{GS}^*(B,B)$ is equivalent to the $E_1$-algebra 
structure induced  by the cup-product.
\end{prop}

\section{The $E_3$-formality Theorem} \label{S:GSofSym}

\subsection{The Gerstenhaber-Schack complexes of $Sym(V)$}
This section is devoted to the proof of our $E_3$-formality theorem (Theorem 0.5). Corollary 0.4 tells us two facts.
First, we can choose the higher Hoschchild complex $CH^{(\bullet > 0)}_{E_2}(\tilde{\Omega}B,\tilde{\Omega}B)$ as a model for the deformation complex
$Def(B)$ of a dg bialgebra $B$. By (Kontsevich) higher Deligne conjecture~\cite{Fra, Lur2, GTZ}, the $L_\infty$-structure
controlling these deformations has a lift to an $E_3$-algebra structure (and an $E_3$-moduli problem). 
Then, the Gerstenhaber-Schack cohomology groups $H_{GS}^*(B,B)$  of $B$ inherits a $3$-Poisson algebra structure,
that is, an algebra over the homology $H_*(E_3)\cong Pois_3$ of the little $3$-disks operad, so that
$H^*Def(B)\cong H^*(T_{\tilde{\Omega}(B)})$ and $H_{GS}^*(B,B)$ are isomorphic as $3$-Poisson algebras.

The formality of the little $3$-disks operad then gives 
an $E_3$-algebra structure on $H_{GS}^*(B,B)$. 
The main question of interest here is to know whether $C_{GS}^*(B,B)$ and $H_{GS}^*(B,B)$
are then quasi-isomorphic as $E_3$-algebras.
We cannot expect such a result to holds true in full generality, however, for our purposes it will be sufficient to prove it in the case
where $B=Sym(V)$ is the symmetric bialgebra over a $\mathbb{Z}$-graded cochain complex $V$ with cohomology of finite dimension in each degree.
The general strategy is very similar to the famous obstruction theoretic approach used by Tamarkin to prove the $E_2$-formality
of the Lie algebra of polyvector fields on an affine space (\cite{Hin},\cite{Tam1}, \cite{Tam2}).

Now let $(V,d)$ be a ($\mathbb{Z}$-graded) cochain complex with cohomology of finite dimension in each degree.
Its symmetric algebra $Sym(V)$ equipped with the induced differential 
(extending $d: V^\bullet \to V^{\bullet+1}\subset Sym(V)$ 
into a derivation) is a conilpotent dg bialgebra (see Example~\ref{Ex:Sym}) so that we can apply Theorem~\ref{T:EquivDef}.  

\smallskip

Furthermore, there is a natural (strict) dg $Pois_3$-algebra structure on (possibly completed) symmetric powers of 
$\big(V\oplus (V)^*\big)[-1]$ mimicking the Poisson structure of vector fields. 
\begin{defn}\label{D:strictPois3onHGS} Let $V$ be a cochain complex. We define a dg $Pois_3$-algebra structure on
\[\widehat{Sym}(V[-1]\oplus 
(V)^*[-1]) \cong \prod_{m, n\geq 1} Sym^m(V[-1])\otimes Sym^n((V)^*[-1])\]
the (fully completed) \emph{symmetric algebra}  (on $(V[-1]\oplus 
(V)^*[-1])$), with the usual differential given by the extension as a derivation of the one of $V$, 
and the degree $-2$  \emph{Poisson bracket} induced
by  the evaluation pairing
$ev:(V^{\bullet})^*\otimes V\rightarrow\mathbb{K}$, $[\chi, v] := ev(\chi \otimes v)= \chi(v)$  and the Leibniz rule :
$$[a \cdot b, c] = a\cdot [b,c] + (-1)^{|a| |b|} b\cdot [a,c]. $$
The subalgebra \[\widehat{Sym}
\big((V)^*[-1]\big) \otimes Sym(V[-1]) \cong 
\prod_{m\geq 1} \bigoplus_{n\geq 1} Sym^m(V[-1])\otimes Sym^n((V)^*[-1])\] is a dg sub-$Pois_3$-algebra and such
also are the sub-complexes
\[\widehat{Sym}^{\geq 1}
\big((V)^{*}[-1]\big) \otimes Sym(V[-1]), \qquad \widehat{Sym}
\big((V)^{*}[-1]\big) \otimes Sym^{\geq 1}(V[-1]). \]
\end{defn}
Note that if $V^{1} =0$, then the completed symmetric 
algebra
is just the usual symmetric algebra. 

\begin{thm}\label{T:IdentificationGSforSym(V)} Let $(V,d)$ be a ($\mathbb{Z}$-graded) cochain complex with finite dimensional
cohomology in each degre.
\begin{enumerate}
\item  The Gertenhaber-Schack cohomology of $Sym(V))$ is given by 
\begin{align*}
H^*_{GS}(Sym(V),Sym(V)) &\cong \prod_{m \geq 1}\big(\bigoplus_{n\geq 1} Sym^n(H^*(V)[-1])\otimes 
Sym^m(H^*(V)^*[-1])\big)\\
&\cong 
\widehat{Sym}(
H^*(V)^*[-1])\otimes Sym (H^*(V)[-1])
\end{align*} as a  $Pois_3$-algebra. Here, the algebra structure is the one of the symmetric algebra and the the degree $-2$  Poisson bracket is induced by  the evaluation pairing
$ev:H^*(V)^*\otimes H^*(V)\rightarrow\mathbb{K}$ as above. 
\item The Gertenhaber-Schack complex  $C^*_{GS}(Sym(V),Sym(V)) $ is equivalent as an $E_3$-algebra 
to $\widehat{Sym}(V[-1]\oplus 
(V)^*[-1])$, where the later is equipped with the $E_3$-algebra structure induced by the chosen formality morphism
$E_3\stackrel{\sim}{\rightarrow}Pois_3$. Furthermore, there is a commutative diagram of $E_3$-algebras
\[
\xymatrix{
C^*_{GS}(Sym(V),Sym(V)) \ar[r]^-{\simeq} \ar@{^{(}->}[d] & {Sym}^{\geq 1}(V[-1]){\otimes} \widehat{Sym}^{\geq 1}(V)^*[-1]) \ar@{^{(}->}[d] \\
\tilde{C}_{GS}^*(Sym(V),Sym(V)) \ar@{^{(}->}[d]\ar[r]^-{\simeq} & {Sym}(V[-1]){\otimes}  
\widehat{Sym}^{\geq 1}((V)^*[-1]) \ar@{^{(}->}[d] \\
C^{full}_{GS}(Sym(V),Sym(V))\ar[r]^-{\simeq} & {Sym}(V[-1]){\otimes}  
\widehat{Sym}((V)^*[-1])
}
\]
whose horizontal arrows are natural (with respect to $V$) equivalences of $E_3$-algebras and the vertical arrows are the canonical inclusions.
\item The three Gertenhaber-Schack complexes are formal as $E_3$-algebras (in particular, there is an equivalence of $E_3$-algebras between 
$C^*_{GS}(Sym(V),Sym(V)) $ and $H^*_{GS}(Sym(V),Sym(V)) $). 
\end{enumerate}
\end{thm}
The theorem also holds with $T_{\tilde{\Omega}Sym(V)}$ and $CH_{E_2}(\tilde{\Omega}Sym(V),
\tilde{\Omega}Sym(V))$  instead of $C_{GS}^*(Sym(V),Sym(V))$ and $\tilde{C}_{GS}^*(Sym(V),Sym(V))$ 
since the $E_3$-structure on the later one is induced by transfer from
the first one along the weak equivalence given by Theorem~\ref{T:EquivDef}.
\begin{proof}[Proof of Theorem~\ref{T:GSforSym}]
Since $Def(Sym(V))$ is precisely the (shifted) Gerstenhaber-Schack complex with its underlying $L_\infty$-structure, 
Theorem~\ref{T:EquivDef} shows that it suffices to prove that  $CH_{E_2}^*(\tilde{\Omega}Sym(V), \tilde{\Omega}Sym(V))$
is formal as an $L_\infty$-algebra. The previous Theorem~\ref{T:IdentificationGSforSym(V)} 
(in the case where $V$ has trivial differential) shows the stronger statement that the later is actually formal as an $E_3$-algebra,
which concludes the proof.
\end{proof}

\smallskip
 
Theorem~\ref{T:IdentificationGSforSym(V)} is a consequence of Proposition~\ref{P:EnhancedCobarofSym(V)} and Proposition~\ref{P:CHE2Symshifted} below. 

\begin{prop} \label{P:CHE2Symshifted} There is a commutative diagram of $E_3$-algebras 
\[ \xymatrix{ T_{Sym(V [-1])} [-2] \ar[r] &
CH^*_{E_2}((Sym(V [-1]), (Sym(V [-1])) \\
 Sym (V[-1])\otimes \widehat{Sym}^{\geq1}((V)^*[-1]) \ar@{^{(}->}[r] \ar[u]_{\simeq} & 
Sym(V[-1])\otimes \widehat{Sym}((V)^*[-1]) \ar[u]_{\simeq}  
} \] whose vertical arrows are quasi-isomorphisms (and the upper map is the canonical one from Theorem~\ref{T:Defcomplexesallagree}).
\end{prop}
\begin{proof}
By \cite{CaWi}, \cite[Corollary 6.39]{GTZ}\footnote{beware that one uses an homological grading in \emph{loc. cit.} 
while we are using a  cohomological grading}, we have, for a Sullivan algebra $A$ (that is a cofibrant cdga), an equivalence 
of $E_3$-algebras 
\begin{equation} \label{eq:CHE2Symshifted} CH^*_{E_2}(A, A) \simeq \widehat{Sym}_A(Der(A,A)[-2]) \end{equation}
where the shifted Lie bracket is given by the Lie bracket on derivations and the Leibniz rule.
Since the differential on $Sym(V [-1])$ is the unique derivation induced by the desuspension of
$d: V^\bullet \to V^{\bullet+1}$, it is in particular a Sullivan algebra so that we can apply this result to 
$A= Sym(V [-1])$. For any cochain complex $W$ with cohomology of finite dimension in each degree,
we have an isomorphism of cochain complexes $Der(Sym(W), Sym(W))\cong (W)^*\otimes Sym(W)$, 
and, under this equivalence,  the Lie bracket of derivations is induced by the pairing $W \otimes  (W)^* \to \mathbb{K}$. 
Thus, for $A=Sym(V [-1])$,  the right hand side of the equivalence~\eqref{eq:CHE2Symshifted}
is precisely $Sym (V[-1])\otimes \widehat{Sym}^{\geq1}((V)^*[-1])$ as an $E_3$-algebra.
Thanks to Theorem~\ref{T:Defcomplexesallagree}  (since our $E_3$-algebra structures are induced by the one on Tamarkin complexes,
we are in fact mainly using the sequence~\eqref{eq:fibersequencePoisn}), we can identify $T_{Sym(V [-1])}$ with the kernel of 
the canonical projection $$ Sym(V[-1])\otimes \widehat{Sym}((V)^*[-1])\to Sym(V[-1])$$ 
and the equivalences as well as the diagram follow.
\end{proof}

\begin{proof}[Proof of Theorem~\ref{T:IdentificationGSforSym(V)}] 
Recall that the $E_3$-algebra structure on $\tilde{C}_{GS}(Sym(V),Sym(V))$ is obtained from the one on 
$CH_{E_2}^{*}(\tilde{\Omega}Sym(V),\tilde{\Omega}Sym(V))$ by transfer thanks to Theorem~\ref{T:EquivDef}. 
Then, from Proposition~\ref{P:EnhancedCobarofSym(V)} we obtain weak equivalences of $E_3$-algebras
\begin{align*}C_{GS}(Sym(V),Sym(V))& \simeq  T_{\tilde{\Omega}Sym(V)} \cong  T_{Sym^{\geq 1}(V [-1])}, \\
\tilde{C}_{GS}(Sym(V),Sym(V))& \simeq CH^*_{E_2}(\tilde{\Omega}Sym(V),\tilde{\Omega}Sym(V))\\
&\cong CH^*_{E_2}(Sym^{\geq 1}(V [-1]), Sym^{\geq 1}(V [-1])). \end{align*} 
We are going to use
 Proposition~\ref{P:CHE2Symshifted} and Proposition~\ref{P:augmentedtononaugmentedDeligne} to compute these complexes. 
 Indeed,
 we have a canonical augmentation $Sym (V [-1])) \cong \mathbb{K}\oplus Sym^{\geq 1}(V [-1]))$ and these 
 propositions as well as diagram~\eqref{eq:HKRforAugmentationEn} yields the equivalence of $E_3$-algebras
 \[Sym(V[-1])\otimes \widehat{Sym}^{\geq 1}((V)^*[-1])
 \stackrel{\simeq}\longrightarrow \tilde{C}_{GS}(Sym(V),Sym(V)). \]
 The commutative
 left square of the diagram of claim  (2) follows similarly. 
Further, by Proposition~\ref{P:augmentedtononaugmentedDeligne}, the right upper map of the diagram fits 
into a diagram of cofiber sequences 
\[\xymatrix{
CH^*_{E_2}(Sym^{\geq 1}(V [-1])) \ar[r]^{\iota} & 
CH^*_{E_2}(Sym(V [-1]))
\ar[r]^{\pi_*} & CH^*_{E_2}(Sym(V [-1]),\mathbb{K}) \\ 
Sym(V[-1])\otimes \widehat{Sym}^{\geq 1}((V)^*[-1])\ar@{^{(}->}[r] \ar[u]^{\simeq} \ar[d]_{\simeq}&  
Sym(V[-1])\otimes \widehat{Sym}((V)^*[-1])\ar@{->>}[r] \ar[u]^{\simeq}
\ar[d]_{\simeq}& \widehat{Sym}((V)^*[-1]) \ar[u]^{\simeq} \ar[d]_{\simeq}\\
\tilde{C}_{GS}(Sym(V),Sym(V)) \ar@{^{(}->}[r] & 
{C}^{full}_{GS}(Sym(V),Sym(V)).
\ar[r]^{\pi_*} & {C}^{full}_{GS}(Sym(V),\mathbb{K}).
}
\]
Hence the commutativity of the right hand square. 
 This proves assertion (2). 

Assertion (1) follows from assertion (2) by passing to the cohomology groups, while
assertion (3) also follows from (2) by choosing any quasi-isomorphism $H^*(V)\stackrel{\simeq}\to V$ 
since we work over a characteristic zero field.
\end{proof}

\subsection{Enhanced Cobar functor on $Sym(V)$}
The main goal of this section is to compute the value of our enhanced cobar functor $\tilde{\Omega}: E_1-Alg(dgCog^{conil}) 
\to E_2-Alg^{aug}$ on the commutative and cocommutative dg bialgebra $Sym(V)$, where $(V, d)$ is a cochain complex with cohomology of finite dimension in each degree: 
\begin{prop}\label{P:EnhancedCobarofSym(V)}
 There is a quasi-isomorphism of $E_2$-algebras $\tilde{\Omega}(Sym(V)) \stackrel{\sim}\to Sym^{\geq 1}(V [-1])$.
\end{prop}
 
We first start by exhibiting a cofibrant resolution of $ Sym(V)$ in $E_1-Alg(dgCog^{conil})$. 
\begin{lem}\label{L:cofibrantresolutionSym(V)} There is a cofibrant resolution
$Cobar(Sym(V[1])) \to Sym(V)$ of $Sym(V)$ in $E_1-Alg(dgCog^{conil})$. 
\end{lem}
\begin{proof}
This is dual to the usual Hochschild-Kostant-Rosenberg theorem for symmetric algebras:
the composition $p: Sym(V[1])[-1]\to V\hookrightarrow Sym(V)$ of the canonical projection
with the canonical inclusion, yields the map of dg algebras  
$\pi: T(Sym(V[1])[-1])\to  Sym(V)$, which is further a map of coalgebras with respect to the shuffle coproduct on the source.
Then $Cobar(Sym(V[1]))$ is the semi-free dg algebra obtained from  $T(Sym(V[1])[-1])$ by adding the extra-differential
$\delta (x_1\otimes x_n):= \sum_{i=1}^n \pm \, x_1\otimes \cdots \otimes x_i^{(1)} \otimes x_i^{(2)}\otimes \cdots \otimes x_n
\in T^{n+1}(Sym(V[1])[-1])$, using the coalgebra structure of  $Sym(V[1])$. Since the latter is cofree cocommutative on
the shift of $V$, we get that $p\circ \delta =0$. 
Hence $\pi: Cobar(Sym(V[1])) \to Sym(V)$  is indeed a map in $E_1-Alg(dgCog^{conil})$. 
This is a quasi-isomorphism by the same argument (replacing algebras by coalgebras) as for the bar construction of a symmetric algebra. Indeed, $Cobar(Sym(V[1]))$ is the derived cotensor product of $\mathbb{K}$ and $Sym(V[1])$ as $Sym(V[1])$-bicomodules and we have a resolution $Sym(V[1])\stackrel{\Delta}\to  Sym(V[1])\otimes \bigoplus Sym^n(V) \otimes Sym(V[1])$, where the target is the tensor product of cocommutative coalgebras equipped with the extra differential $\delta(x\otimes f \otimes y) = x^{(1)}\otimes p(x^{(2)}\cdot f \otimes y \pm \, x^\otimes f\cdot p(y^{(1)})\otimes y^{(2)}$.

In fact one can also directly apply the usual HKR theorem which gives a quasi-isomorphism of Hopf algebras
$Sym(V[1])\to Bar(Sym(V))$ (given by the unique dg coalgebra map extending the linear map $Sym(V[1])\stackrel{p}\to V[1] \hookrightarrow Sym(V)[1]$) and then conclude by the counit of the bar-cobar adjunction. 
\end{proof}

Let
\[
\phi^*\mathcal{B}^{enh}_{E_1}(-)_-:E_1-Alg^{con,aug}(dgCog^{conil})\rightarrow E_2-Cog^{conil}
\]
be the equivalence given by Theorem~\ref{T:Barenhanced} and  Section~\ref{S:proofCorollary0.2}. 
We wish to evaluate $\phi^*\mathcal{B}^{enh}_{E_1}(-)_-$ on $Sym(V)$.
 
\smallskip
 
Note that we have  a commutative diagram of ($\infty$-)operads 
\[\label{d:compatibilityEmCog} \xymatrix{E_1\otimes Ass  \ar[r]^{a} & E_\infty \ar[r]^{\simeq} & Com \\ 
E_2 \ar[u]^{b} \ar[ur] & &}  \]
which in turn induces a commutative diagram 
\begin{equation}\label{d:compatibilityEmCogbig}\xymatrix{  E_\infty-Cog^{conil} \ar[rd]\ar[r]^{b^*\qquad } & E_1-Cog^{conil}(dgCog^{conil}) \ar[r]^{\qquad obl} \ar[d]_{a^*} 
& E_1-Cog^{conil} \\ 
dgCom-Cog^{conil} \ar[u]^{\simeq} \ar@/_3pc/[rru]&  E_2-Cog^{conil} \ar[ru] & 
}  \end{equation}
 
\bigskip
 
\noindent where $obl$ is induced by forgetting the $dg$-coalgebra structure and the non-labbeled arrows
are given by the standard restriction functors between $E_m$-algebras.

The differential $d$ on a cochain complex $(V,d)$ extends canonically 
to the coartinian cofree cocommutative
coalgebra $Sym(V)$, giving it a canonical dg-cocommutative coalgebra structure and thus  $ E_2-Cog^{conil}$ structure 
by restriction.

\begin{lem}\label{L:psistar}
One has an equivalence $\phi^*\mathcal{B}^{enh}_{E_1}(Sym(V))_- \cong Sym^{\geq 1}(V[1])$ in $E_2-Cog^{conil}$.
\end{lem}
\begin{proof}
Applying the argument line of the proof of Lemma~\ref{L:cofibrantresolutionSym(V)} to the augmentation ideal $Sym^{\geq 1}(V)$ of $Sym(V)$, we have
\[
Sym^{\geq 1}(V)\simeq Cobar(Sym^{\geq 1}(V[1]))
\]
in $E_1-Alg(dgCog^{conil})$.

Applying the functor $\mathcal{B}_{E_1}^{enh}$ gives an equivalence
\[
\mathcal{B}_{E_1}^{enh}Sym^{\geq 1}(V)\simeq \mathcal{B}_{E_1}^{enh}Cobar(Sym^{\geq 1}(V[1]))
\]
in $E_1-Cog^{conil}(dgCog^{conil})$,
hence
\[
\phi^*\mathcal{B}_{E_1}^{enh}Sym^{\geq 1}(V)\simeq \phi^*\mathcal{B}_{E_1}^{enh}Cobar(Sym^{\geq 1}(V[1]))
\]
in $E_2-Cog^{conil}$. It thus remains to prove that $\mathcal{B}_{E_1}^{enh}Cobar(Sym^{\geq 1}(V[1]))\simeq Sym^{\geq 1}(V[1])$ in $E_1-Cog^{conil}(dgCog^{conil})$.

Recall that the Koszul duality of $E_1$-operads provides us a Quillen equivalence
\[
Cobar:E_1-Cog^{conil}\rightleftarrows E_1-Alg^{aug}:Bar
\]
where $Bar$ is weakly equivalent to the derived functor of indecomposables. Let us note that the unit and counit of this adjunction are actually always weak equivalences, because every object of $E_1-Cog^{conil}$ is cofibrant (cofibrations of coalgebras are determined in chain complexes) and every object of $E_1-Alg^{aug}$ is fibrant (fibrations of algebras are determined in chain complexes). The fact that $E_1$ is a Hopf operad gives a distributive law between $E_1$ and $Ass$ allowing to lift these functors as follows
\[
\xymatrix{
 & E_1-Alg^{aug}(dgCog^{conil}) \ar[dr]^-{oblv} & \\
E_1-Cog^{conil}(dgCog^{conil}) \ar[ur]^-{Cobar}\ar[r]_-{oblv} & E_1-Cog^{conil} \ar[r]_-{Cobar} & E_1-Alg^{aug}
}
\]
where $oblv$ is the forgeful functor, and
\[
\xymatrix{
 & E_1-Cog^{conil}(dgCog^{conil}) \ar[dr]^-{oblv} & \\
E_1-Alg^{aug}(dgCog^{conil}) \ar[ur]^-{Bar}\ar[r]_-{oblv} & E_1-Alg^{aug}\ar[r]_-{Bar} & E_1-Cog^{conil}
}.
\]
This gives us a new adjunction
\[
Cobar:E_1-Cog^{conil}(dgCog^{conil})\rightleftarrows E_1-Alg^{aug}(dgCog^{conil}):Bar
\]
whose unit and counit are still weak equivalences (recall that weak equivalences here are all determined in cochain complexes).
This functor $Cobar$ is the one of Lemma~\ref{L:cofibrantresolutionSym(V)} 
and the functor $Bar$ is actually weakly equivalent to $\mathcal{B}_{E_1}^{enh}$ (they are both equivalent to the derived functor of indecomposables), which means that we have a natural quasi-isomorphism
\[
Id\stackrel{\sim}{\rightarrow}\mathcal{B}_{E_1}^{enh}\circ Cobar,
\]
giving in particular
\[
Sym^{\geq 1}(V[1])\stackrel{\sim}{\rightarrow}\mathcal{B}_{E_1}^{enh}Cobar(Sym^{\geq 1}(V[1]))
\]
in $E_1-Cog^{conil}(dgCog^{conil})$.
\end{proof}
The last step in the construction of the functor $\tilde{\Omega}$ of Corollary~\ref{C:Barenhanced} is the Koszul duality
equivalence $Cobar^{(n)}: E_n-Cog^{conil} \to E_n-Alg^{aug} $ in the case $n=2$.

We wish to prove that $Cobar^{(n)}(Sym^{\geq 1}(W)) \cong Sym^{\geq 1}(W[-n])$.
A first step is to prove that, by formality of $E_n$-operads,
it is enough to prove this result with the operad $Pois_n$ (instead of $E_n$), given that $Sym^{\geq 1}(W)$
can be seen as a $Pois_n$-coalgebra.
For this, we use several features of the theory of right modules over operads as thoroughly studied in \cite{Fre3}.

Given any operad $R$ equipped with an operad morphism $E_1\rightarrow R$, one can extend the bar construction from $E_1$-algebras
to $R$-algebras: this extended bar construction is the functor $S_R(B_R,-):R-Alg\rightarrow Ch_{\mathbb{K}}$
associated to a certain cofibrant quasi-free right $R$-module $B_R$. 
Moreover, any weak equivalence of operads $\varphi:R\stackrel{\sim}{\rightarrow}S$ induces a weak equivalence
of right $R$-modules $B_R\stackrel{\sim}{\rightarrow}\varphi^*B_S$, where $\varphi^*B_S$ is $B_S$ equipped
with the right $R$-module structure induced by $\varphi$. In the case of $E_n$-operads, it turns out that this extended
bar construction is given by the iterated bar construction $Bar^{(n)}$ \cite{Fre-it} (right adjoint to $Cobar^{(n)}$).

\smallskip

Let us fix a formality morphism $\varphi:E_n\stackrel{\sim}{\rightarrow}Pois_n$ and denote respectively  $Bar^{(n)}_{Pois_n}$ and
$Cobar^{(n)}_{Pois_n}$ the Koszul duality Bar and Cobar construction for $Pois_n$-algebras. 
\begin{lem}\label{L:Cobarn=Pois} Let $n\geq 2$. 
One has natural equivalences 
\[
Bar^{(n)}\circ\varphi^*\sim Bar^{(n)}_{Pois_n}, 
\]
where $\varphi^*$ is the restriction of structures from $Pois_n$-algebras to $E_n$-algebras and
\[
Cobar^{(n)}\circ\varphi^*\sim Cobar^{(n)}_{Pois_n}
\]
 where $\varphi^*$ is the restriction of structures from $Pois_n$-coalgebras to $E_n$-coalgebras.
\end{lem}
\begin{proof}
Given the formality morphism $\varphi:E_n\stackrel{\sim}{\rightarrow}Pois_n$, according to \cite[Theoerem 7.2.2]{Fre3},
there is a natural isomorphism
\[
S_{E_n}(B_{E_n},-)\circ\varphi^*\cong S_{Pois_n}(\varphi_{!}B_{E_n},-)
\]
where $\varphi^*$ is the restriction of structures fitting in the adjunction
\[
\varphi_{!}:Pois_n-Alg\rightleftarrows E_n-Alg:\varphi^*
\]
and $\varphi_{!}$ is the extension of structures fitting in the adjunction
\[
\varphi_{!}:Pois_n-Mod\rightleftarrows E_n-Mod:\varphi^*
\]
between right $Pois_n$-modules and right $E_n$-modules.
Since $\varphi$ is a weak equivalence, by \cite[Theorem 16.B]{Fre3} the adjunction above between right modules is a Quillen
equivalence, so the weak equivalence of right $E_n$-modules $B_{E_n}\stackrel{\sim}{\rightarrow}\varphi^*B_{Pois_n}$ corresponds by adjunction to a weak equivalence of right $Pois_n$-modules $\varphi_{!}B_{E_n}\stackrel{\sim}{\rightarrow}B_{Pois_n}$.
Moreover, since $\varphi_{!}$ is a left Quillen functor, the right module $\varphi_{!}B_{E_n}$ is still cofibrant, so this is a weak equivalence of cofibrant right $Pois_n$-modules. By \cite[Theorem 15.1.A]{Fre3}, it induces consequently a natural weak equivalence
\[
S_{E_n}(B_{E_n},-)\circ\varphi^*\cong S_{Pois_n}(\varphi_{!}B_{E_n},-)\stackrel{\sim}{\rightarrow}S_{Pois_n}(B_{Pois_n},-),
\]
hence
\[
Bar^{(n)}\circ\varphi^*\sim Bar^{(n)}_{Pois_n}
\]
where $\varphi^*$ is the restriction of structures. This proves the first claim. 
Since $Bar^{(n)}$ fits in a Quillen equivalence with $Cobar^{(n)}$ (for both operads) 
as a left adjoint, the first claim also implies that
\[
Cobar^{(n)}\circ\varphi^*\sim Cobar^{(n)}_{Pois_n}
\]
as well, where $\varphi^*$ is now the restriction of structures from $Pois_n$-coalgebras to $E_n$-coalgebras.
\end{proof}

We now compute $Cobar^{(n)}_{Pois_n}$ on $Sym^{\geq 1}(W)$. Recall that to a  a dg Lie coalgebra $(\mathfrak{g},\delta, d)$, we can associate 
a $Pois_{n}$-coalgebra defined by $Sym^{\geq 1}(\mathfrak{g}[n-1])$ with cocommutative cobracket given by the one of $Sym^{\geq 1}(-)=Cocom$ (the free 
cocommutative conilpotent coalgebra functor)
and Lie cobracket 
induced by the one of $\mathfrak{g}$ and the Leibniz rule. Note that $\mathfrak{g} \mapsto Sym^{\geq 1}(\mathfrak{g}[1-n])$ is the right adjoint
of the canonical functor $Pois_{n}-Coalg \to Lie-Coalg$ given by $P\mapsto P[1-n]$. The following Lemma is rather standard
\begin{lem}\label{L:CobarPoisn} One has 
  $$Cobar^{(n)}_{Pois_n}(Sym^{\geq 1}(\mathfrak{g}[n-1])) \cong \big(Com(\mathfrak{g}[-1]), d_{CE}\big)$$ 
where $Cobar^{(n)}_{Pois_n}(-) $ is the $E_n$-Koszul duality functor $Pois_n-Cog^{conil} \to Pois_n-Alg^{aug}$ and 
the right hand side is the part of positive weight of the Chevalley-Eilenberg algebra of the dg-Lie coalgebra $\mathfrak{g}$. 
\end{lem}
\begin{proof}
By Koszul
duality of the $Pois_n$-operad, the functor $Cobar^{(n)}_{Pois_n}:dg-Pois_n-Cog^{conil} \to dg-Pois_n-Alg^{aug}$
on $C$ is given by the free $Pois_n^{!}$-algebra on $C$ endowed with the differential induced by its $Pois_n$-coalgebra structure.
Recall that $Pois_n^{!}= \Lambda^n Pois_n$. 
Hence, one has an equivalence 
$$Cobar^{(n)}(C)= Com\big( Lie(C[-1])[1-n]\big)$$
see \cite{Ta-deformationofd-algebra, Gi-HomotopyGerstenhaber, Hin}.
We are interested in the case where $C= Sym^{\geq 1}(\mathfrak{g}[n-1])$ 
which is  cofree as a conilpotent cocommutative coalgebra (in cochain complexes). 
It follows that the Harrison cochain complex $Har_*\big(Sym^{\geq 1}(\mathfrak{g}[n-1])\big) = Lie(Sym^{\geq 1}(\mathfrak{g}[n-1])[-1])$ 
of  $Sym^{\geq 1}(\mathfrak{g}[n-1])$ is quasi-isomorphic to $(\mathfrak{g}[n-1])[-1]$. 
Hence  we have a quasi-isomorphism of cochain complexes
$$Cobar^{(n)}(Sym(\mathfrak{g}[n-1]))= Com\big( Lie(Sym^{\geq 1}(\mathfrak{g}[n-1])[-1])[1-n]\big) 
\simeq Com (\mathfrak{g}[-1])\big)$$ where the right hand side is 
identified with the part of positive weight in the Chevalley-Eilenberg algebra of the Lie coalgebra $\mathfrak{g}$. 
(see~\cite{Tam1, Gi-HomotopyGerstenhaber, Hin}
for detailled computation of this dg-$Pois_n$-algebra cohomology of the free $Pois_n$-algebra generated by a $dg$-Lie algebra). 
\end{proof}

We finally compute $Cobar^{(n)}$ on $Sym^{\geq 1}(W)$.
\begin{lem}\label{L:Cobarn} Let $n\geq 2$. One has $Cobar^{(n)}(Sym^{\geq 1}(W)) \cong Sym^{\geq 1}(W[-n])$ 
where $Cobar^{(n)}$ is the $E_n$-Koszul duality functor $E_n-Cog^{conil} \to E_n-Alg^{aug}$. 
\end{lem}
\begin{proof}
By Lemma~\ref{L:Cobarn=Pois}, 
we conclude that proving the result for the $E_n$-cobar construction $Cobar^{(n)}(Sym^{\geq 1}(W))$ on $Sym^{\geq 1}(W)$ 
considered as an $E_n$-coalgebra boils down to check it for the $Pois_n$-cobar construction $Cobar^{(n)}_{Pois_n}(-)$ on $Sym^{\geq 1}(W)$
considered as a $Pois_n$-coalgebra with trivial cobracket. 
Applying the previous Lemma~\ref{L:CobarPoisn} to the case $\mathfrak{g}= W[1-n]$ equipped with zero cobracket (or simply redoing the proof of the Lemma in this simpler case), we get  
 $$Cobar^{(n)}(Sym^{\geq 1}(W))= Com\big( Lie(Sym^{\geq 1}(W)[-1])[1-n]\big) \simeq Com (W[-n])\big)$$ where the right hand side is identified with the part of positive weight of the Chevalley-Eilenberg algebra of a Lie coalgebra with a null cobracket, since $ Sym^{\geq 1}(W)$ has  a null cobracket. 
Hence it is formal and we get that $Cobar^{(n)}(Sym^{\geq 1}(W)) \simeq Sym^{\geq 1}(W[-n])$.
\end{proof}

\smallskip

\begin{proof}[Proof of Proposition~\ref{P:EnhancedCobarofSym(V)}] 
By definition of the functor $\tilde{\Omega}$ (Section~\ref{S:proofCorollary0.2}), it is the composition of functors $Cobar^{(2)}\circ \phi^*\mathcal{B}_{E_1}^{enh}(-)_-$. Hence, one has 
\begin{eqnarray*}
\tilde{\Omega}(Sym(V)) &\simeq&  Cobar^{(n)}\Big(\phi^*\mathcal{B}_{E_1}^{enh}(Sym^{\geq 1}(V))\Big) \\
 &\simeq&  Cobar^{(n)}\Big( Sym^{\geq 1}(V[1])\Big) \qquad \mbox{by Lemma~\ref{L:psistar} } \\
 &\simeq& Sym^{\geq 1}(V[1][-2]) =Sym^{\geq 1}(V[-1])
\end{eqnarray*}
by Lemma~\ref{L:Cobarn} in the case $n=2$.   
\end{proof}

\section{Etingof-Kazdhan deformation quantization}\label{S:EKQ}
We now apply our computations of the $L_\infty$-structure of the Gerstenhaber-Schack complex to quantization of Lie bialgebras and prove Corollary~\ref{C:EKQuant}.

If $V$ is a finite dimensional vector space, by~\cite[Corollary 5.1]{Mer1}, the Maurer-Cartan elements of the graded Lie algebra
$$\prod_{m \geq 1}\big(\bigoplus_{n\geq 1} Sym^n(V[-1])\otimes 
Sym^m((V)^*[-1])\big)[2]$$
(for the Lie bracket induced by the Poisson bracket of Definition~\ref{D:strictPois3onHGS})
are exactly the prop morphisms 
$BiLie_{\infty}^{+}\rightarrow End_V$, that is, the homotopy Lie bialgebra structures on $V$.
Here $BiLie$ is the prop governing (dg-)Lie bialgebras structures on (dg-)vector spaces,  $BiLie_{\infty}$ is its standard cofibrant 
resolution, and $BiLie_{\infty}^{+}$ is the result of the $+$ construction (Section~\ref{S:Plus}) on the later one. 

By Theorem~\ref{T:IdentificationGSforSym(V)}, this Lie algebra
is the one of $H^*_{GS}(Sym(V),Sym(V))$ and is quasi-isomorphic to the deformation complex $C^*_{GS}(Sym(V),Sym(V))$. This is the first step
to derive the Etingof-Kazdhan quantization theorem from the formality theorem. 

Next, let $\varphi^+: Bialg_{\infty}^+\to End^{Poly}_{Sym(V)}$ be the map canonically associated to the (cocommutative and commutative) 
bialgebra structure on $Sym(V)$.
The Etingof-Kazdhan quantization theorem assert that any Maurer-Cartan element in the above Lie algebra gives rise functorially to 
a formal deformation of $\psi^+$, in other words a point of the moduli space $\underline{Bialg_{\infty}^+}\{Sym(V)\}^{\varphi^+}(k[[\hbar]])$.
These observations  (a dg extension of the quantization of Lie bialgebra was established
in the Appendix of \cite{GH}) can be made for a dg vector space $V$ as well.  

\smallskip

According to Merkulov~\cite{Mer2}, 
there is an explicit completed\footnote{completion for the filtrations induced by the genus and the number
of vertices} dg prop $(DefQ^+,d^+)$ such that, for any cochain complex $(V, d_V)$,  the representations of 
$(DefQ^+,d^+)$ into $(V, d_V)$ are in bijection with polydifferential representations of homotopy bialgebra structures on 
the cochain algebra $(Sym(V), d_V)$. In other words
\begin{equation} Mor_{Prop}(DefQ^+, End_{V}) \; \cong \; Mor_{Prop}\big(Bialg_{\infty}^+, End^{Poly}_{Sym(V)} 
\big). \end{equation}
Polydifferential representations are explained in~\cite{Mer1, Mer2}; they amount to replace the endomorphism prop $End_{Sym(V)} $  
with its polydifferential endomorphism $End^{Poly}_{Sym(V)}$ which is the dg-sub-prop of $End_{Sym(V)} $ spanned by 
those multilinear maps that are (normalized) polydifferential operators on $Sym(V)$. Merkulov~\cite{Mer2} has proved that the natural map  
\begin{equation}
 Mor_{Prop}(DefQ^+, End_{V}) \stackrel{\simeq}\longrightarrow Mor_{Prop}\big(Bialg_{\infty}^+, End_{Sym(V)} 
\big).
\end{equation}
is a quasi-isomorphism. 

\medskip

Using the above analysis, the formality Theorem~\ref{T:GSforSym} implies the following quantization theorem. Following the previous notation,
we have the map $\varphi^+: Bialg_{\infty}^+\to End_{Sym(V)}$ giving the dg bialgebra structure of $Sym(V)$ for a cochain complex $V$.
Further, we also have $0^+:BiLie_{\infty}^+\to End_{V}$ the (trivial) map corresponding to the trivial (= null) Lie bialgebra structure on $V$.
\begin{thm}\label{T:EK} Let $(V,d)$ be a ($\mathbb{Z}$-graded) cochain complex with finite dimensional
cohomology in each degree. 
There is a weak-equivalence of formal moduli problems
\[ \underline{Bialg_{\infty}^+}\{Sym(V)\}^{\varphi^+} \simeq \underline{BiLie_{\infty}^+}\{V\}^{0^+}\]
\end{thm}
In particular, given an artinian cdga $R$, any $R$-deformation of the trivial Lie bialgebra structure on $V$ 
(in other words a dg-bialgebra structure on $V\otimes R$)
gives rise to a (essentially unique isomorphism class of) $R$-deformation of  $Sym(V)$.
\begin{proof}
 By Theorem~\ref{T:IdentificationGSforSym(V)} and Theorem~\ref{T:EquivDef}, we have an equivalence of deformation complexes 
 \begin{eqnarray*}
  Def(Sym(V))\simeq g_{Bialg_{\infty}^+,Sym(V)}^{\varphi^+} &\simeq & C^*_{GS}(Sym(V), Sym(V))\\
  &\simeq &  {Sym}^{\geq 1}(V[-1]){\otimes}  
\widehat{Sym}^{\geq 1}(V)^*[-1]) \\
& \simeq & g_{BiLie_{\infty}^+, V}^{0^+}
 \end{eqnarray*}
 where the last equivalence is from~\cite[Corollary 5.1]{Mer1}. 
 Then the result follows from Theorem~\ref{T:Yal2} or Lurie classification of (commutative) formal moduli problem~\cite{Lur0}.
\end{proof}
The theorem also implies that formal deformations of the null Lie bialgebra structure quantize into formal deformation
s of the symmetric bialgebra on $V$, hence Corollary~\ref{C:EKQuant}. 

\begin{proof}[Proof of Corollary~\ref{C:EKQuant}]
 The first claim of the corollary is equivalent to the existence of a  natural (in $V$)  weak-equivalence
 \begin{equation}\label{eq:EKQuant.1}
\underline{Bialg_{\infty}^+}\{Sym(V)\}^{\varphi^+}(k[[\hbar]])
 \; \simeq \;  \underline{BiLie_{\infty}^+}\{V\}^{0^+}(k[[\hbar]])
 \end{equation}
where,  for any formal moduli problem
$\mathcal{F}$, we denote $$\mathcal{F}(k[[\hbar]]) \; := \;\mathop{lim}\limits_{i}\, \mathcal{F}(k[\hbar]/h^i).$$
By~\cite[Corollaire 2.11]{To-Bourbaki} (or \cite{Lur0}), there is an natural weak-equivalence
\begin{equation}\label{eq:FMPinhbar}
 \mathcal{F}(k[[\hbar]])  \; \simeq \; Map(k[-1], g_\mathcal{F})
\end{equation}
where $g_\mathcal{F}$ is the $L_\infty$-algebra associated to the formal moduli problem $\mathcal{F}$ and $k[-1]$ is the abelian one dimensional
Lie algebra concentrated in degree $1$.

The equivalence~\eqref{eq:FMPinhbar} now implies the first claim~\eqref{eq:EKQuant.1} in virtue of Theorem~\ref{T:EK} 
(and Theorem~\ref{T:EquivalentModuligivesEquivalentDefComplexex}).

\smallskip

For claim (2), we note that the above proof relies essentially on the formality Theorem~\ref{T:IdentificationGSforSym(V)} 
for the Gerstenhaber-Schack $L_\infty$-algebra of $Sym(V)$. This theorem relies on the formality of the $E_n$-operad for $n\geq 3$,
which does not rely on the choice of a Drinfeld associator. 
\end{proof}
Our proof somehow goes  in the converse way as Merkulov~\cite{Mer2} proof
of a dg prop quasi-isomorphism $EK: (DefQ^+,d^+)\to (\widehat{BiLie},0)$, which relies in \textit{loc. cit.} on the existence
of the Etingof-Kazdhan quantization functor for dg Lie bialgebras. 
Our proof establishes first this equivalence, from which the above statement actually follows. 
Applied to an ordinary finite dimensional vector space, our result gives a new proof 
of Etingof-Kazhdan deformation quantization. Indeed, homotopy Lie bialgebra structures $BiLie_{\infty}\rightarrow End_V$ are in this case strict Lie bialgebra structures: since the endomorphism prop is concentrated in degree $0$ and prop morphisms preserve the degree, the generators of positive degree in $BiLie_{\infty}$ are sent to zero. In particular, the $1$-cycles whose images under the differential of $BiLie_{\infty}$ are the $0$-boundaries inducing the Jacobi, co-Jacobi and Drinfeld's compatibility relations in homology, are sent to $0$. Since dg prop morphisms commute with the differentials, this means that these $0$-boundaries are also sent to $0$. That is, the two generators of degree $0$ in $BiLie_{\infty}$ are sent to a bracket and a cobracket satisfying the Jacobi, co-Jacobi and Drinfeld's compatibility relations, i.e. a strict Lie bialgebra structure.
This is a particular incarnation of the general fact that any $P_{\infty}$-algebra structure on a vector space is actually a strict one. Non trivial homotopy algebra structures appear only in the differential graded setting.
By the same argument, homotopy bialgebra structures $Bialg_{\infty}\rightarrow End_{Sym(V)}$ are only the ones factoring through $Bialg$.

Moreover, it further applies to any $\mathbb{Z}$-graded dg Lie bialgebra $g$ with cohomology 
of finite total dimension as well as homotopy Lie bialgebra structures on $g$.

\end{document}